\documentclass[a4paper,12pt]{amsart}
\pdfoutput=1
\usepackage{setspace}
\usepackage[left=18mm,head=30mm,bottom=30mm,foot=20mm,right=18mm]{geometry}
\usepackage{amsmath,amsthm,amsfonts,amssymb,mathtools,mathrsfs,url,bm,enumitem,stackengine}
\usepackage[varvw]{newtxmath}
\usepackage{newtxtext,newtxmath}
\usepackage{hyperref}
\hypersetup{
  colorlinks=true,
  linkcolor=blue,
  citecolor=red,%magenta,
}
\mathtoolsset{showonlyrefs=true}
\theoremstyle{definition}
  
  \newtheorem{rem}{Remark}[section]
\theoremstyle{plain}
  \newtheorem{thm}{Theorem}[section]
  \newtheorem{cor}{Corollary}[section]
  \newtheorem{prop}{Proposition}[section]
  \newtheorem{lem}{Lemma}[section]
  
\renewcommand{\Re}{\operatorname{Re}}

\newcommand\restr[2]{{% we make the whole thing an ordinary symbol
  \left.\kern-\nulldelimiterspace % automatically resize the bar with \right
  #1 % the function
  \vphantom{\big|} % pretend it's a little taller at normal size
  \right|_{#2} % this is the delimiter
}}
\DeclareMathOperator{\sgn}{sgn}
\begin{document}

\title[Refined pointwise estimates for 1D viscous compressible flows]{Refined pointwise estimates for solutions to the 1D barotropic compressible Navier--Stokes equations: An application to the long-time behavior of a point mass}

\author{Kai Koike}
\address{Graduate School of Engineering, Kyoto University, Kyoto 615-8540, Japan}
\email{koike.kai.42r@st.kyoto-u.ac.jp}
\date{\today}
\maketitle

\begin{abstract}
  We study the long-time behavior of a point mass moving in a one-dimensional viscous compressible fluid. Previously, we showed that the velocity of the point mass $V(t)$ satisfies a decay estimate $V(t)=O(t^{-3/2})$~[K. Koike, J. Differential Equations \textbf{271} (2021) 356--413]. This result was obtained as a corollary to pointwise estimates of solutions to a free boundary problem of barotropic compressible Navier--Stokes equations. In this paper, we give a simple necessary and sufficient condition on the initial data for the decay estimate $V(t)=O(t^{-3/2})$ to be optimal. This is achieved by refining the pointwise estimates previously obtained: we make use of \textit{inter-diffusion waves} that, together with the classical \textit{diffusion waves}, give an improved approximation of the fluid behavior around the point mass; this then leads to a sharper understanding of the long-time behavior of the point mass.
\end{abstract}

\setcounter{tocdepth}{2}
\tableofcontents

\section{Introduction}
The study of phenomena arising from interaction of moving or deforming solids with fluid flows are called \textit{the fluid--structure interaction problems}. Fluid--structure interaction is a source of interesting phenomena that motivates development of new mathematical ideas.

In this paper, we study the interaction of a point mass immersed in a one-dimensional barotropic viscous compressible fluid. Here, we are especially interested in the long-time behavior of the point mass. Previously, in~\cite{Koike21}, we showed that the point mass velocity $V(t)$ satisfies a decay estimate $V(t)=O(t^{-3/2})$. However, whether this decay estimate $V(t)=O(t^{-3/2})$ is optimal or not wasn't answered with sufficient generality.\footnote{\label{ft:optimal}In this paper, we say that the decay estimate $V(t)=O(t^{-3/2})$ is optimal if there exist $C>1$ and $T>0$ such that $C^{-1}t^{-3/2}\leq |V(t)|\leq Ct^{-3/2}$ for $t\geq T$.}~To answer this problem is the purpose of this paper. This problem requires a very detailed understanding of the fluid behavior around the point mass and leads us to revisit a classical result on pointwise estimates for solutions to quasilinear hyperbolic-parabolic equations~\cite{LZ97}.

It is known that \textit{diffusion waves} --- self-similar solutions to generalized viscous Burgers' equations --- give nice approximations of solutions to quasilinear hyperbolic-parabolic equations~\cite[Theorem~2.6]{LZ97}. In fact, we extended their result for the Cauchy problem to our free-boundary (fluid--structure interaction) problem and obtained as a corollary the decay estimate $V(t)=O(t^{-3/2})$. However, diffusion waves can not give an accurate approximation of flow behavior around the point mass, and this is the principal difficulty in answering the problem of the optimality of the decay estimate $V(t)=O(t^{-3/2})$; using only diffusion waves, we cannot precisely capture the long-time behavior of $V(t)$. To overcome this problem, we make use of another type of waves which we call \textit{inter-diffusion waves} that are capable of describing fluid behavior around the point mass precisely. These inter-diffusion waves have partially and implicitly appeared in the previous works, e.g.~\cite{Koike21,LZ97}, but here we define them completely and explicitly then make full use of them to refine the previously obtained pointwise estimates. As a corollary, we show that the decay estimate $V(t)=O(t^{-3/2})$ is optimal if and only if the initial perturbations of the total density and the momentum are both non-zero.

In the rest of this introduction, we present the equations we consider and explain the question we address in more detail. The main theorem and its consequences are presented in Section~\ref{sec:main_thm}. We then prove them in Section~\ref{sec:proof}.

\subsection{Model equations}
Consider a system of a one-dimensional viscous compressible fluid and a point mass. The point mass is immersed in the fluid, and we denote its position and velocity at time $t$ by $X=h(t)$ and $V(t)=h'(t)$. Here, $X$ is a Cartesian coordinate on the real line $\mathbb{R}$. For the fluid, we denote the density and the velocity by $\rho=\rho(X,t)$ and $U=U(X,t)$. We impose a simplifying assumption that the flow is barotropic, that is, the pressure $P$ is a function only of the density $\rho$.

Under the notations and assumptions above, the flow is described by the following one-dimensional barotropic compressible Navier--Stokes equations:
\begin{equation}
  \label{eq:NS_Euler}
  \begin{dcases}
    \rho_t+(\rho U)_X=0,                                          & X\in \mathbb{R}\backslash \{ h(t) \},\, t>0, \\
    (\rho U)_t+(\rho U^2)_X+P(\rho)_X=\nu U_{XX},                 & X\in \mathbb{R}\backslash \{ h(t) \},\, t>0.
  \end{dcases}
\end{equation}
Here, the viscosity coefficient $\nu$ is assumed to be a positive constant. As boundary conditions, we require that the fluid does not penetrate through the point mass:
\begin{equation}
  \label{eq:BC_Euler}
  U(h(t)_{\pm},t)=V(t), \quad t>0.
\end{equation}
Here, $f(X_+,t)\coloneqq \lim_{Y\searrow X}f(Y,t)$ and $f(X_-,t)\coloneqq \lim_{Y\nearrow X}f(Y,t)$ are limits from the right and the left, respectively, and $f(X_{\pm},t)=g(t)$ means $f(X_+,t)=f(X_-,t)=g(t)$.

Let us next consider the equations of motion for the point mass. Denote the fluid force acting on the point mass by $F=F(t)$. Then Newton's second law for the point mass is $mV'(t)=F(t)$, where $m>0$ is the mass of the point particle. Requiring the conservation of the total momentum
\begin{equation}
  \int_{-\infty}^{\infty}(\rho U)(X,t)\, dX+mV(t)=\int_{-\infty}^{h(t)}(\rho U)(X,t)\, dX+\int_{h(t)}^{\infty}(\rho U)(X,t)\, dX+mV(t) 
\end{equation}
and using~\eqref{eq:NS_Euler} and~\eqref{eq:BC_Euler}, we find that $F(t)=\llbracket -P(\rho)+\nu U_X \rrbracket(h(t),t)$, where the double brackets denote the jump: $\llbracket f \rrbracket(X,t)\coloneqq f(X_+,t)-f(X_-,t)$. Therefore, Newton's second law for the point mass reads
\begin{equation}
  \label{eq:Newton_Euler}
  mV'(t)=\llbracket -P(\rho)+\nu U_X \rrbracket(h(t),t).
\end{equation}

In sum, equations~\eqref{eq:NS_Euler}--\eqref{eq:Newton_Euler} together with the initial conditions 
\begin{equation}
  \label{eq:IC_Euler}
  h(0)=h_0,\, V(0)=V_0;\, \rho(X,0)=\rho_0(X),\, U(X,0)=U_0(X), \quad X\in \mathbb{R}\backslash \{ h_0 \}
\end{equation}
describe the time evolution of the system. However, they are posed in a time-dependent domain $\mathbb{R}\backslash \{ h(t) \}$; for ease of mathematical analysis, we transform the domain into a time-independent one. To do this, we introduce \textit{the Lagrangian mass coordinate}.

Fix $x\in \mathbb{R}_* \coloneqq \mathbb{R}\backslash \{ 0 \}$ and $t\geq 0$. Then, let $X=X(x,t)$ be the solution to
\begin{equation}
  x=\int_{h(t)}^{X(x,t)}\rho(X',t)\, dX'.
\end{equation}
We assume that $\rho(X,t)\geq \rho_0$ for some $\rho_0>0$. Then the equation above is uniquely solvable and determines a one-to-one map
\begin{equation}
  \mathbb{R}_* \ni x \mapsto X(x,t)\in \mathbb{R}\backslash \{ h(t) \}.
\end{equation}
This new variable $x$ is the Lagrangian mass coordinate. We also change the dependent variables as follows:
\begin{equation}
  v(x,t)=\frac{1}{\rho(X(x,t),t)}, \quad u(x,t)=U(X(x,t),t), \quad p(v)=P\left( \frac{1}{v} \right).
\end{equation}
The variable $v$ is called \textit{the specific volume}. Using the first equation in~\eqref{eq:NS_Euler}, we get
\begin{equation}
  \label{eq:change_of_variables}
  \frac{\partial X(x,t)}{\partial x}=v, \quad \frac{\partial X(x,t)}{\partial t}=u.
\end{equation}

With the new variables introduced above,~\eqref{eq:NS_Euler}--\eqref{eq:IC_Euler} are transformed into
\begin{equation}
  \label{eq:FundamentalEquations}
  \begin{dcases}
    v_t-u_x=0,                                      & x\in \mathbb{R}_*,\, t>0, \\
    u_t+p(v)_x=\nu \left( \frac{u_x}{v} \right)_x,  & x\in \mathbb{R}_*,\, t>0, \\
    u(0_{\pm},t)=V(t),                              & t>0, \\
    V'(t)=\llbracket -p(v)+\nu u_x/v \rrbracket(t), & t>0, \\
    V(0)=V_0;\, v(x,0)=v_0(x),\, u(x,0)=u_0(x),     & x\in \mathbb{R}_*.
  \end{dcases}
\end{equation}
Here and in what follows, we set $m=1$ for simplicity; the double brackets denote the jump at $x=0$: $\llbracket f \rrbracket(t)\coloneqq \llbracket f \rrbracket(0,t)=f(0_+,t)-f(0_-,t)$; and
\begin{equation}
  v_0(x)=\frac{1}{\rho_0(X(x,0))}, \quad u_0(x)=U_0(X(x,0)).
\end{equation}
These are the equations we analyze in this paper. Note that~\eqref{eq:FundamentalEquations} does not contain $h(t)$, but we can recover it by $h(t)=h_0+\int_{0}^{t}V(s)\, ds$.

The first two equations in~\eqref{eq:FundamentalEquations} are called the $p$-system in the literature. We give a remark on its perturbative form. Since we consider solutions that are close to the steady state $(v,u)=(1,0)$, it is natural to write the $p$-system as follows:
\begin{equation}
  \label{eq:vector_form}
  \bm{u}_t+A\bm{u}_x=B\bm{u}_{xx}+
  \begin{pmatrix}
    0 \\
    N_x
  \end{pmatrix},
\end{equation}
where
\begin{equation}
  \label{eq:vector_form_variables}
  \bm{u}=
  \begin{pmatrix}
    v-1 \\
    u
  \end{pmatrix},
  \quad A=
  \begin{pmatrix}
    0 & -1 \\
    -c^2 & 0
  \end{pmatrix},
  \quad B=
  \begin{pmatrix}
    0 & 0 \\
    0 & \nu
  \end{pmatrix},
  \quad N=-p(v)+p(1)-c^2(v-1)-\nu \frac{v-1}{v}u_x,
\end{equation}
and $c=\sqrt{-p'(1)}>0$ is the speed of sound for the state $(v,u)=(1,0)$; for $c$ to be well-defined, we assume that $p'(1)<0$.

\subsection{Long-time behavior of a point mass}
The question we address in this paper is the long-time behavior of the point mass velocity $V(t)$. Physically, it is natural to expect that $V(t)$ decays over time. In fact, in our previous paper~\cite{Koike21}, we showed a power-law type decay estimate $V(t)=O(t^{-3/2})$. It is, however, harder to understand whether this decay estimate is optimal or not (cf.~Footnote~\ref{ft:optimal}). Let us explain where the difficulty lies and how we overcome it.

To prove the decay estimate $V(t)=O(t^{-3/2})$, we employed the approach using pointwise estimates of Green's function developed in~\cite{LZ97,Zeng94}. Although this method originally targeted initial value problems, the Fourier--Laplace transform technique developed in~\cite{LY11,LY12} paved the way to analyze initial-boundary value problems (see for example~\cite{Deng16,DWY15,DW18}). We showed that these tools are also useful for the analysis of our free-boundary value problem. As a result, we obtained pointwise error bounds of a diffusion wave approximation of the solution $(u,v,V)$ to~\eqref{eq:FundamentalEquations}; more precisely, we constructed an approximation $(\bar{v},\bar{u})$ of $(u,v)$ using \textit{diffusion waves} (defined later in Section~\ref{sec:diffusion_bidiffusion}) and obtained pointwise error bounds of the form $|(v-\bar{v},u-\bar{u})(x,t)|\leq C\phi(x,t)$. Here, the function $\phi=\phi(x,t)$ satisfies $\phi(0_{\pm},t)=O(t^{-3/2})$ and $\bar{u}(0_{\pm},t)$ decays exponentially fast. From this, we were able to conclude that $V(t)=u(0_{\pm},t)=O(t^{-3/2})$.

The diffusion wave approximation $(\bar{v},\bar{u})$ provides an accurate approximation --- the leading order long-time asymptotics --- of the solution $(v,u)$ around $x=\pm ct$ ($c$ is the speed of sound). However, as we mentioned above, the diffusion wave approximation $\bar{u}$ decays exponentially fast around $x=0$; on the other hand, the velocity of the fluid $u$ is expected to decay algebraically there --- as $t^{-3/2}$ in most cases. Therefore, the diffusion wave approximation is not a valid approximation of the fluid behavior around the point mass. This is the main reason why it is difficult to answer whether the decay estimate $V(t)=O(t^{-3/2})$ is optimal or not.

In this paper, we answer the problem of the optimality of the decay estimate $V(t)=O(t^{-3/2})$ by refining the diffusion wave approximation. This is done by the help of what we call \textit{inter-diffusion waves} (defined in Section~\ref{sec:diffusion_bidiffusion}). The principal role of inter-diffusion waves is to extract the leading order long-time asymptotics of the fluid behavior around the point mass, thus complementing the approximation provided by diffusion waves. They also help describe the second order asymptotics of the solution at $x=\pm ct$ where the diffusion wave approximation gives the leading order asymptotics. In this way, we construct a new approximation $(\tilde{v},\tilde{u})$ satisfying pointwise error bounds $|(v-\tilde{v},u-\tilde{u})(x,t)|\leq C\psi(x,t)$ with $\psi(0_{\pm},t)=O(t^{-7/4})$ in particular. Hence, we can study the optimality of the decay estimate $V(t)=O(t^{-3/2})$ by studying the long-time behavior of $\tilde{u}$ in detail. This is the approach taken in this paper.

\section{Main theorem}
\label{sec:main_thm}
In this section, we state the main theorem and its corollaries. To do so, we first need to define diffusion waves and inter-diffusion waves mentioned in the introduction.

\subsection{Diffusion waves and inter-diffusion waves}
\label{sec:diffusion_bidiffusion}
We first note that the matrix $A$ in~\eqref{eq:vector_form_variables} has two eigenvalues $\lambda_1=c$ and $\lambda_2=-c$; as right and left eigenvectors corresponding to the eigenvalue $\lambda_i$, we may take $r_i$ and $l_i$ given by
\begin{equation}
  \label{right_ev}
  r_1=\frac{2c}{p''(1)}
  \begin{pmatrix}
    -1 \\
    c
  \end{pmatrix},
  \quad r_2=\frac{2c}{p''(1)}
  \begin{pmatrix}
    1 \\
    c
  \end{pmatrix}
\end{equation}
and
\begin{equation}
  \label{def:li}
  l_1=\frac{p''(1)}{4c}
  \begin{pmatrix}
    -1 & 1/c
  \end{pmatrix},
  \quad l_2=\frac{p''(1)}{4c}
  \begin{pmatrix}
    1 & 1/c
  \end{pmatrix}.
\end{equation}
Here, we assume that $p''(1)\neq 0$.

Now, using the left eigenvector $l_i$ and the initial data $(v_0,u_0,V_0)$, we define $M_i \in \mathbb{R}$ by
\begin{equation}
  \label{eq:mass}
  M_i \coloneqq \int_{-\infty}^{\infty}l_i
  \begin{pmatrix}
    v_0-1 \\
    u_0
  \end{pmatrix}
  (x)\, dx+l_i
  \begin{pmatrix}
    0 \\
    V_0
  \end{pmatrix}.
\end{equation}
Then, the $i$-th diffusion wave with mass $M_i$ is defined as the solution $\theta_i$ to the generalized viscous Burgers equation
\begin{equation}
  \label{def:thetai}
  \partial_t \theta_i+\lambda_i \partial_x \theta_i+\partial_x \left( \frac{\theta_{i}^{2}}{2} \right)=\frac{\nu}{2}\partial_{x}^{2}\theta_i, \quad x\in \mathbb{R},\, t>0
\end{equation}
with the initial condition
\begin{equation}
  \label{def:thetai_init}
  \lim_{t\searrow -1}\theta_i(x,t)=M_i \delta(x), \quad x\in \mathbb{R},
\end{equation}
where $\delta(x)$ is the Dirac delta function. Note that the initial condition is imposed at $t=-1$ because we don't want $\theta_i$ to be singular at $t=0$. By the Cole--Hopf transformation, we can obtain an explicit formula for $\theta_i$:
\begin{equation}
  \label{eq:theta_explicit}
  \theta_i(x,t)=\frac{\sqrt{\nu}}{\sqrt{2(t+1)}}\left( e^{\frac{M_i}{\nu}}-1 \right) e^{-\frac{(x-\lambda_i(t+1))^2}{2\nu(t+1)}}\left[ \sqrt{\pi}+\left( e^{\frac{M_i}{\nu}}-1 \right) \int_{\frac{x-\lambda_i(t+1)}{\sqrt{2\nu(t+1)}}}^{\infty}e^{-y^2}\, dy \right]^{-1}.
\end{equation}
We can see from this formula that $\theta_i$ diffuses around $x=\lambda_i t$ with width of the order of $t^{1/2}$.

As we mentioned in the introduction, diffusion waves are not enough to describe the long-time behavior of the solution around $x=0$. In order to overcome this problem, we introduce new waves: let $\xi_i$ be the solution to the variable coefficient inhomogeneous convective heat equation
\begin{equation}
  \label{def:xii}
  \partial_t \xi_i +\lambda_i \partial_x \xi_i +\partial_x (\theta_i \xi_i)+\partial_x \left( \frac{\theta_{i'}^{2}}{2} \right)=\frac{\nu}{2}\partial_{x}^{2}\xi_i, \quad x\in \mathbb{R},t>0
\end{equation}
with the initial condition
\begin{equation}
  \label{def:xii_init}
  \xi_i(x,0)=0, \quad x\in \mathbb{R},
\end{equation}
where $i'=3-i$, that is, $1'=2$ and $2'=1$. We call $\xi_i$ the $i$-th inter-diffusion wave with mass pair $(M_1,M_2)$.

Let $\zeta_i$ be the solution to~\eqref{def:xii} and~\eqref{def:xii_init} without the variable coefficient term $\partial_x (\theta_i \xi_i)$. The importance of analyzing $\zeta_i$ was already noticed in~\cite{LZ97}; in fact,~\cite[Lemma~3.4]{LZ97} is used to obtain bounds of $\zeta_i$. Our contribution lies in the recognition that the addition of the variable coefficient term leads to a refinement of the previously known pointwise estimates of solutions relying only on diffusion waves and the execution of complicated analysis required for this purpose; see Remark~\ref{rem:previous_pwe}.\footnote{A heuristic argument explaining the need for the variable coefficient term is given in one of our conference report~\cite{Koike21RIMS}.}

We now comment on how the inter-diffusion wave $\xi_i$ looks like. By Lemma~\ref{lem:xi}, we see that $\xi_i$ possesses the decay properties listed in Table~\ref{table1}. In particular, $\xi_i$ decays as $t^{-3/2}$ around $x=0$; note that the diffusion wave $\theta_i$ decays exponentially fast there. For visual understanding, two snapshots of $\xi_1$ for the mass pair $(M_1,M_2)=(1,1)$ with $c=\nu=1$ are shown in Figure~\ref{fig:0}.\footnote{These plots are created using a pseudospectral method (cf.~\cite{Koike21RIMS}).}~In contrast to the diffusion wave $\theta_i$, the region where the inter-diffusion wave $\xi_i$ decays non-exponentially (algebraically) extends not only around $x=\lambda_i t$ but includes the intermediate region between $x=\lambda_i t$ and $x=\lambda_{i'}t$; this is one of the origin of its name: ``inter-'' for intermediate. Alternatively, we can interpret ``inter-'' as meaning interaction of waves of speed $\lambda_i$ and $\lambda_{i'}$, which points to the presence of the terms $\lambda_i \partial_x \xi_i$ and $\partial_x(\theta_{i'}^{2}/2)$ in~\eqref{def:xii}; see also~\eqref{eq:zetai_IE}.

\begin{figure}[htpb]
  \centering
  \includegraphics[width=0.5\linewidth]{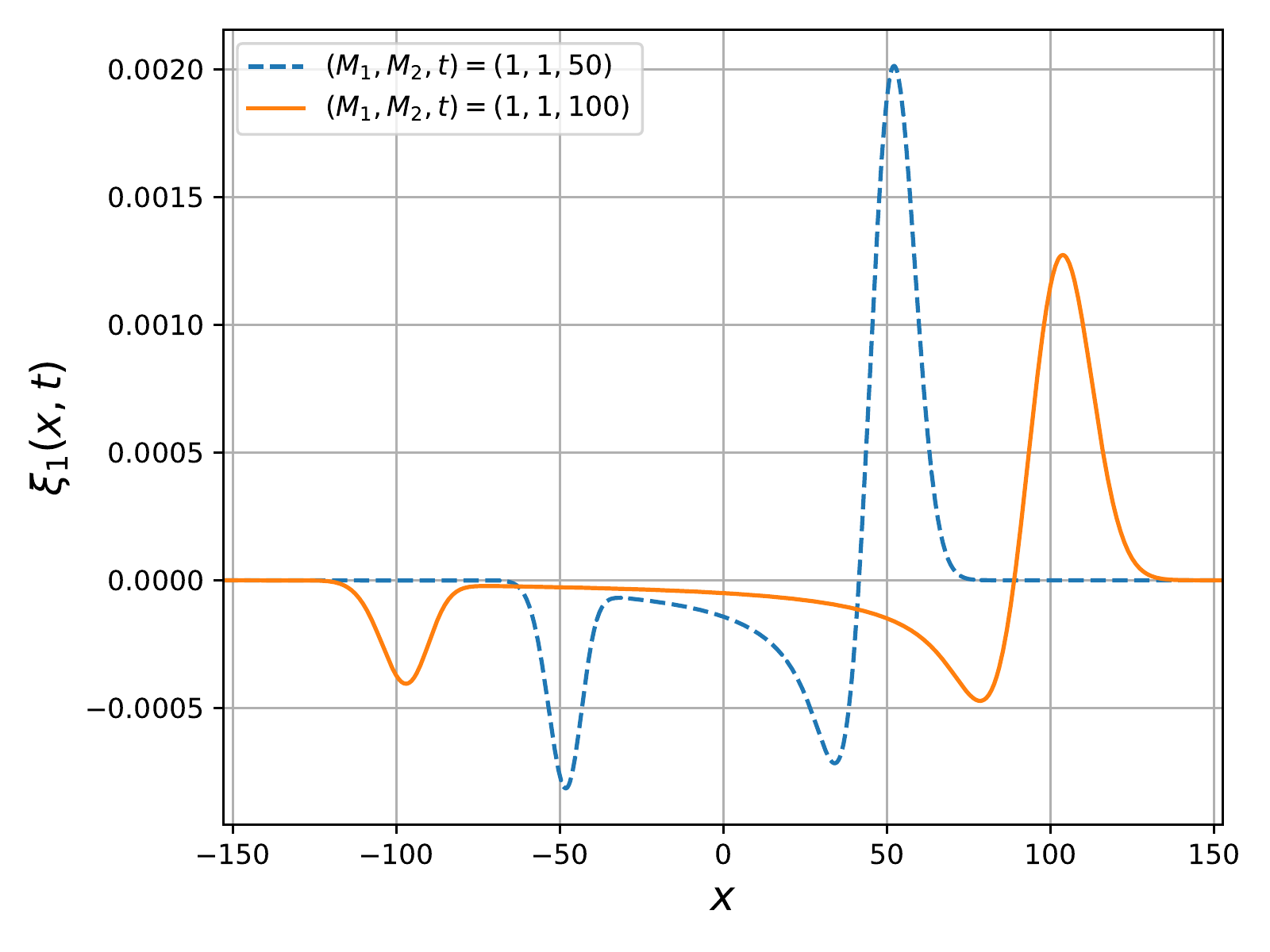}
  \caption{\small{The inter-diffusion wave $\xi_1$ with mass pair $(M_1,M_2)=(1,1)$. Here, we set $c=\nu=1$. The dotted and solid line represent $\xi_1$ at time $t=50$ and $t=100$, respectively.}}
  \label{fig:0}
\end{figure}

Next, to relate the diffusion wave $\theta_i$ and the inter-diffusion wave $\xi_i$ to the solution $(v,u,V)$ to~\eqref{eq:FundamentalEquations}, we decompose $\bm{u}$ --- recall~\eqref{eq:vector_form_variables} --- with respect to the eigenbasis $(r_1,r_2)$:
\begin{equation}
  \label{eq:decomposition}
  \bm{u}=u_1 r_1+u_2 r_2.
\end{equation}
Using the relation
\begin{equation}
  \label{eq:lr_identity}
  \begin{pmatrix}
    l_1 \\
    l_2
  \end{pmatrix}
  \begin{pmatrix}
    r_1 & r_2 \\
  \end{pmatrix}
  =
  \begin{pmatrix}
    1 & 0 \\
    0 & 1
  \end{pmatrix},
\end{equation}
we can calculate the component $u_i$ by
\begin{equation}
  \label{def:ui}
  u_i=l_i\bm{u}.
\end{equation}
Now, if we multiply $l_i$ to~\eqref{eq:vector_form}, we get
\begin{equation}
  \label{eq:pde_ui}
  \partial_t u_i+\lambda_i \partial_x u_i=l_iB\partial_{x}^{2}
  \begin{pmatrix}
    r_1 & r_2
  \end{pmatrix}
  \begin{pmatrix}
    u_1 \\
    u_2
  \end{pmatrix}
  +\partial_x N_i=\frac{\nu}{2}\partial_{x}^{2}(u_1+u_2)+\partial_x N_i,
\end{equation}
where
\begin{equation}
  \label{def:Ni}
  N_i=l_i
  \begin{pmatrix}
    0 \\
    N
  \end{pmatrix}
  =\frac{p''(1)}{4c^2}N.
\end{equation}
This $N_i$ in fact does not depend on $i$ (we add the subscript $i$ just to distinguish $N_i$ from $N$). We can now observe some resemblance between~\eqref{def:thetai},~\eqref{def:xii}, and~\eqref{eq:pde_ui}; note that $N_i$ is a quadratic nonlinear term (recall~\eqref{eq:vector_form_variables} again).

\subsection{Refined pointwise estimates of solutions}
We just need few more notations to state our main theorem on refined pointwise estimates of solutions. First, let
\begin{equation}
  \label{eq:psi74}
  \psi_{7/4}(x,t;\lambda_i)\coloneqq [(x-\lambda_i(t+1))^2+(t+1)]^{-7/8},
\end{equation}
\begin{equation}
  \label{eq:psibar}
  \bar{\psi}(x,t;\lambda_i)\coloneqq [|x-\lambda_i(t+1)|^7+(t+1)^5]^{-1/4},
\end{equation}
and
\begin{equation}
  \label{def:Psi}
  \Psi_i(x,t)\coloneqq \psi_{7/4}(x,t;\lambda_i)+\bar{\psi}(x,t;\lambda_{i'});
\end{equation}
here, we recall that $\lambda_1=c$, $\lambda_2=-c$, and $i'=3-i$. Next, to state the compatibility conditions, we introduce
\begin{align}
  \label{def:C1C2}
  \begin{aligned}
    \mathcal{C}_1(v,u) & \coloneqq -p(v)+\nu \frac{u_x}{v}, \\
    \mathcal{C}_2(v,u) & \coloneqq -p'(v)u_x+\frac{\nu}{v}\mathcal{C}_1(v,u)_{xx}-\nu \frac{u_{x}^{2}}{v^2}.
  \end{aligned}
\end{align}
Here, note that if $(v,u,V)$ is the solution to~\eqref{eq:FundamentalEquations}, we have $\mathcal{C}_1(v,u)_x=u_t$, $\mathcal{C}_2(v,u)=\mathcal{C}_1(v,u)_t$, and $\llbracket \mathcal{C}_1(v,u) \rrbracket(t)=V'(t)$. Next, to state the conditions on the spatial decay of initial data, we introduce
\begin{equation}
  \label{def:u0ipm}
  u_{0i}^{-}(x)\coloneqq \int_{-\infty}^{x}u_{0i}(y)\, dy,\quad u_{0i}^{+}(x)\coloneqq \int_{x}^{\infty}u_{0i}(y)\, dy,
\end{equation}
where
\begin{equation}
  \label{def:u0i}
  u_{0i}\coloneqq l_i
  \begin{pmatrix}
    v_0-1 \\
    u_0
  \end{pmatrix}.
\end{equation}
Finally, we introduce the notation $\llbracket f \rrbracket \coloneqq f(0_+)-f(0_-)$, write $f(0_{\pm})=g$ to mean $f(0_+)=f(0_-)=g$, and denote by $||\cdot ||_{k}$ ($k\in \mathbb{N}$) the Sobolev $H^k(\mathbb{R}_*)$-norm.

The main theorem of this paper is then stated as follows.

\begin{thm}
  \label{thm:main}
  Let $v_0-1,u_0 \in H^6(\mathbb{R}_*)$ and $V_0 \in \mathbb{R}$. Assume that they satisfy the following compatibility conditions:
  \begin{equation}
    \label{eq:compatibility}
    u_0(0_{\pm})=V_0, \quad \mathcal{C}_1(v_0,u_0)_x(0_{\pm})=\llbracket \mathcal{C}_1(v_0,u_0) \rrbracket, \quad \mathcal{C}_2(v_0,u_0)_x(0_{\pm})=\llbracket \mathcal{C}_2(v_0,u_0) \rrbracket.
  \end{equation}
  Under these assumptions, there exist $\delta_0,C>0$ such that if
  \begin{equation}
    \label{ass:smallness}
    \delta \coloneqq \sum_{i=1}^{2}\left\{ ||u_{0i}||_6+\sup_{x\in \mathbb{R}_*}\left[ (|x|+1)^{7/4}|u_{0i}(x)| \right]+\sup_{x>0}\left[ (|x|+1)^{5/4}(|u_{0i}^{-}(-x)|+|u_{0i}^{+}(x)|) \right] \right\} \leq \delta_0,
  \end{equation}
  then the unique global-in-time solution $(v,u,V)$ to~\eqref{eq:FundamentalEquations} --- which exists by Theorem~\ref{thm:global_existence} below --- satisfies the pointwise estimates
  \begin{equation}
    \label{thm:main:ineq}
    |(u_i-\theta_i-\xi_i-\gamma_{i'}\partial_x \theta_{i'})(x,t)|\leq C\delta \Psi_i(x,t) \quad (x\in \mathbb{R}_*,t\geq 0;i=1,2),
  \end{equation}
  where $i'=3-i$ and $\gamma_i=(-1)^i \nu/(4c)$. Here, $u_i$ is defined by~\eqref{def:ui} with~\eqref{eq:vector_form_variables} and~\eqref{def:li}; the definitions of $\theta_i$ and $\xi_i$ are given in Section~\ref{sec:diffusion_bidiffusion}; and $\Psi_i$ is defined by~\eqref{def:Psi}.
\end{thm}

From this theorem, we obtain two corollaries on the long-time behavior of the point mass velocity $V(t)$; these give a simple necessary and sufficient condition for the optimality of the decay estimate $V(t)=O(t^{-3/2})$. Before stating these results, we note that by~\eqref{def:li} and~\eqref{eq:mass}, we have
\begin{equation}
  \label{eq:mass_sum_difference}
  M_1+M_2=\frac{p''(1)}{2c^2}\left( \int_{-\infty}^{\infty}u_0(x)\, dx+V_0 \right), \quad M_1-M_2=-\frac{p''(1)}{2c}\int_{-\infty}^{\infty}(v_0-1)(x)\, dx.
\end{equation}

\begin{cor}
  \label{cor:V_lower}
  Define $M_i$ by~\eqref{eq:mass} and assume that $M_{1}^{2}-M_{2}^{2}\neq 0$, that is,
  \begin{equation}
    \label{ass:nonzero_perturbation}
    \left( \int_{-\infty}^{\infty}(v_0-1)(x)\, dx \right) \cdot \left( \int_{-\infty}^{\infty}u_0(x)\, dx+V_0 \right) \neq 0.
  \end{equation}
  Then under the assumptions of Theorem~\ref{thm:main}, there exist $\delta_0>0$, $C>1$, and $T(\delta)>0$ such that if~\eqref{ass:smallness} holds, then the solution $(v,u,V)$ to~\eqref{eq:FundamentalEquations} satisfies
  \begin{equation}
    \label{cor:eq:V_lower}
    C^{-1}|M_{1}^{2}-M_{2}^{2}|(t+1)^{-3/2}\leq (\sgn(M_{1}^{2}-M_{2}^{2}))V(t)\leq C\delta (t+1)^{-3/2} \quad (t\geq T(\delta)),
  \end{equation}
  where $\sgn$ is the sign function. In particular, we have
  \begin{equation}
    C^{-1}|M_{1}^{2}-M_{2}^{2}|(t+1)^{-3/2}\leq |V(t)|\leq C\delta (t+1)^{-3/2} \quad (t\geq T(\delta)).
  \end{equation}
\end{cor}

The result above shows that the condition $M_{1}^{2}-M_{2}^{2}\neq 0$ is a sufficient condition for the optimality of the decay estimate $V(t)=O(t^{-3/2})$; the following result shows that this condition is also a necessary condition.

\begin{cor}
  \label{cor:V_improved_decay}
  Define $M_i$ by~\eqref{eq:mass} and assume that $M_{1}^{2}-M_{2}^{2}=0$, that is,
  \begin{equation}
    \left( \int_{-\infty}^{\infty}(v_0-1)(x)\, dx \right) \cdot \left( \int_{-\infty}^{\infty}u_0(x)\, dx+V_0 \right)=0.
  \end{equation}
  Then under the assumptions of Theorem~\ref{thm:main}, there exist $\delta_0,C>0$ such that if~\eqref{ass:smallness} holds, then the solution $(v,u,V)$ to~\eqref{eq:FundamentalEquations} satisfies
  \begin{equation}
    \label{cor:eq:V_improved_decay}
    |V(t)|\leq C\delta(t+1)^{-7/4} \quad (t\geq 0).
  \end{equation}
\end{cor}

Now, before going into their proofs, we discuss some aspects of the results above.

\begin{rem}
  As we mentioned above, the condition $M_{1}^{2}\neq M_{2}^{2}$, that is,~\eqref{ass:nonzero_perturbation} is a necessary and sufficient condition for the optimality of the decay estimate $V(t)=O(t^{-3/2})$. This condition is used to obtain a lower bound of $|(\xi_1+\xi_2)(0,t)|$, the sum of inter-diffusion waves evaluated at $x=0$; see Lemma~\ref{lem:xi_lower}. By~\eqref{eq:change_of_variables}, the condition~\eqref{ass:nonzero_perturbation} is expressed in the Eulerian coordinate as
  \begin{equation}
    \left( \int_{-\infty}^{\infty}(\rho_0-1)(X)\, dX \right) \cdot \left( \int_{-\infty}^{\infty}(\rho_0 U_0)(X)\, dX+V_0 \right) \neq 0.
  \end{equation}
  This is the requirement that the initial perturbations of the total density and the total momentum be non-zero. Since we are dealing with the velocity $V(t)$ of the point mass, for~\eqref{cor:eq:V_lower} to hold, it seems natural to require that the initial perturbation of the total momentum is non-zero: $M_1+M_2 \neq 0$. For example, if $v_0-1$ is even, $u_0$ is odd, and $V_0=0$, then $M_1+M_2=0$. In this case, we have $V(t)=0$ for all $t\geq 0$ by symmetry, and the decay estimate $V(t)=O(t^{-3/2})$ is obviously not optimal. However, the requirement that the initial perturbation of the total density be non-zero (i.e.~$M_1-M_2 \neq 0$) is more subtle; and this subtlety is in fact reflected in the proof of Lemma~\ref{lem:xi_cancellation}, which we use in the proof of Corollary~\ref{cor:V_improved_decay}.
\end{rem}

\begin{rem}
  \label{rem:previous_pwe}
  In our previous work, we showed the following pointwise estimates instead of~\eqref{thm:main:ineq}~\cite[Theorem~1.2]{Koike21}:
  \begin{equation}
    \label{eq:pwe_old}
    |(u_i-\theta_i)(x,t)|\leq C\delta \Phi_i(x,t) \quad (x\in \mathbb{R}_*,t\geq 0;i=1,2),
  \end{equation}
  where
  \begin{equation}
    \label{def:Phi}
    \Phi_i(x,t)\coloneqq [(x-\lambda_i(t+1))^2+(t+1)]^{-3/4}+[|x-\lambda_{i'}(t+1)|^3+(t+1)^2]^{-1/2}.
  \end{equation}
  Since $\Psi_i$ decays faster than $\Phi_i$, Theorem~\ref{thm:main} says that we can improve the diffusion wave approximation of the solution $u_i$ by adding the inter-diffusion wave $\xi_i$ (and $\gamma_{i'}\partial_x \theta_{i'}$).

  In Table~\ref{table1}, we listed decay estimates (optimal in general) of functions appearing in~\eqref{thm:main:ineq}.\footnote{For example, the table reads as: $\xi_i(x,t)=O(t^{-3/2})$ in the $O(1)$-neighborhood of $x=0$.}~From the table, we observe the following: (I) In the $O(1)$-neighborhood of $x=\lambda_i t$, the leading order long-time asymptotics of $u_i$ is given by $\theta_i$ and the second order asymptotics by $\xi_i$; (II) in the $O(1)$-neighborhood of $x=0$, the leading order asymptotics of $u_i$ is given by $\xi_i$; (III) in the $O(1)$-neighborhood of $x=\lambda_{i'}t$, the leading order asymptotics of $u_i$ is given by $\xi_i+\gamma_{i'}\partial_x \theta_{i'}$. In particular, the leading order asymptotics of $V(t)=u(0_{\pm},t)=(2c^2/p''(1))(u_1+u_2)(0_{\pm},t)$ is described by a linear combination of inter-diffusion waves, that is, $(2c^2/p''(1))(\xi_1+\xi_2)(0,t)$; this is so, however, only when $M_{1}^{2}\neq M_{2}^{2}$, and cancellation occurs in the sum $(\xi_1+\xi_2)(0,t)$ when $M_{1}^{2}=M_{2}^{2}$. Therefore, analysis of inter-diffusion waves is the key in the proofs of Corollaries~\ref{cor:V_lower} and~\ref{cor:V_improved_decay}.
\end{rem}

\begin{table}[htbp]
  \caption{Decay estimates of functions appearing in Theorem~\ref{thm:main}}
  \centering
  \begin{tabular}{l | c c c c}
    \hline
    \addstackgap{Location} & \addstackgap{$\theta_i(x,t)$} & \addstackgap{$\xi_i(x,t)$} & \addstackgap{$\partial_x \theta_{i'}(x,t)$} & \addstackgap{$\Psi_i(x,t)$} \\ [0.5ex]
    \hline
    \addstackgap{$x-\lambda_i t=O(1)$} & \addstackgap{$O(t^{-1/2})$} & \addstackgap{$O(t^{-3/4})$} & \addstackgap{$O(e^{-t/C})$} & \addstackgap{$O(t^{-7/8})$} \\ [0.5ex]
    \addstackgap{$x=O(1)$} & \addstackgap{$O(e^{-t/C})$} & \addstackgap{$O(t^{-3/2})$} & \addstackgap{$O(e^{-t/C})$} & \addstackgap{$O(t^{-7/4})$} \\ [0.5ex]
    \addstackgap{$x-\lambda_{i'}t=O(1)$} & \addstackgap{$O(e^{-t/C})$} & \addstackgap{$O(t^{-1})$} & \addstackgap{$O(t^{-1})$} & \addstackgap{$O(t^{-5/4})$} \\ [1.0ex]
    \hline
  \end{tabular}
  \label{table1}
\end{table}

\begin{rem}
  Under the assumptions of Corollary~\ref{cor:V_lower}, the sign of $V(t)$ becomes constant after sufficiently long time has elapsed, which is $\sgn(M_{1}^{2}-M_{2}^{2})$. Hence, $V(t)$ does not decay in an oscillatory manner.
\end{rem}

\begin{rem}
  Concerning Corollary~\ref{cor:V_improved_decay}, the question whether the decay estimate $V(t)=O(t^{-7/4})$ for initial data satisfying $M_{1}^{2}=M_{2}^{2}$ is optimal or not is left open. We remark, however, that we numerically observed that for the corresponding Cauchy problem, $u(0,t)$ actually decays as $t^{-7/4}$ for certain initial data satisfying
  \begin{equation}
    \int_{-\infty}^{\infty}(v_0-1)(x)\, dx=0, \quad \int_{-\infty}^{\infty}u_0(x)\, dx\neq 0.
  \end{equation}
  For this, we refer to~\cite[Section~4]{Koike21RIMS}.
\end{rem}

\begin{rem}
  In our previous work, we mentioned without a proof that~\eqref{cor:eq:V_lower} holds when $M_1 \neq 0$ and $M_2=0$~\cite[Remark~1.3]{Koike21}.\footnote{There is a typo in~\cite[Remark~1.3]{Koike21}: we wrote $|V(t)|\geq C^{-1}\delta(t+1)^{-3/2}$ but this should have been $|V(t)|\geq C^{-1}\delta^2(t+1)^{-3/2}$.}~Note that in this special case, we have $\theta_2=\xi_1=0$ and the differential equation for $\xi_2$ is not variable coefficient. We also mention that also in~\cite[Remark~2.7]{LZ97}, it was claimed without a proof that the decay estimate $(u_i-\theta_i)(0,t)=O(t^{-3/2})$ is optimal for the Cauchy problem. Our contribution in this paper is to give a rigorous proof and make clear --- give a simple necessary and sufficient condition --- when the decay rate $-3/2$ is the best possible.
\end{rem}

\begin{rem}
  In~\cite[Theorem~1.2]{Koike21} where we proved the pointwise estimates~\eqref{eq:pwe_old}, the required regularity of the initial data is $H^4(\mathbb{R}_*)$, whereas it is $H^6(\mathbb{R}_*)$ in Theorem~\ref{thm:main}. The reason is because, in order to obtain the refined pointwise estimates~\eqref{thm:main:ineq}, we need pointwise estimates of $\partial_x(u_i-\theta_i)$ (cf.~Theorem~\ref{thm:PWE_derivative}). Similarly to the case of the Cauchy problem~\cite[Remark~2.8]{LZ97}, we then need this stronger regularity requirement.
\end{rem}

\begin{rem}
  It is natural to ask whether we could remove the smallness assumptions in Theorem~\ref{thm:main}. In fact, for barotropic compressible Navier--Stokes equations and its heat conductive generalizations, there are many works on global-in-time existence of large solutions. One of the earliest works are those by Kanel'~\cite{Kanel68} and by Kazhikhov and Shelukhin~\cite{KS77}. Although these works consider systems without moving solids, subsequent works extended these to fluid--structure interaction problems~\cite{FMNT18,MTT17,Shelukhin77,Shelukhin78,Shelukhin82,Shelukhin83}. When the fluid domain is bounded, the long-time behavior of these large solutions are well understood, with or without moving solids: the system stabilizes to equilibrium exponentially fast; see~\cite{Lequeurre20} and Theorem~2.2 in the survey article~\cite{Nishida86}. However, when the fluid domain is unbounded, we only know that the system returns to equilibrium~\cite{Kanel68,LL16}; in order to obtain explicit convergence rates, we still need to restrict ourselves to small solutions as in~\cite{Hoff92,Kawashima86,LZ97} (this is also the case for the inviscid case~\cite{Liu78}). Considering this state-of-the-art, it seems quite difficult to remove the smallness assumptions in Theorem~\ref{thm:main}. We remark, however, that for a corresponding problem for viscous Burgers' equation, a very sharp result on the long-time behavior of a point mass is obtained in~\cite{VZ03} without smallness restrictions.
\end{rem}

\section{Proofs}
\label{sec:proof}
In this section, we prove Theorem~\ref{thm:main} and also Corollaries~\ref{cor:V_lower} and~\ref{cor:V_improved_decay}. After some preliminary considerations in Sections~\ref{sec:greens} to \ref{AppendixB}, Theorem~\ref{thm:main} is proved in Section~\ref{SectionIII:PointwiseEstimates}; Corollaries~\ref{cor:V_lower} and~\ref{cor:V_improved_decay} are proved in Section~\ref{sec:proofs_cors}.

\subsection{Green's functions and integral equations}
\label{sec:greens}
First, we review some properties related to Green's functions. Denote by $G=G(x,t)\in \mathbb{R}^{2\times 2}$ the fundamental solution to the linearization of~\eqref{eq:vector_form}, that is, the solution to
\begin{equation}
  \label{SectionIII:CNS_Lag_per_Fundamental_Solution}
  \begin{dcases}
    \partial_t G+
    \begin{pmatrix}
      0    & -1 \\
      -c^2 & 0
    \end{pmatrix}
    \partial_x G=
    \begin{pmatrix}
      0 & 0 \\
      0 & \nu
    \end{pmatrix}
    \partial_{x}^{2}G,   & x\in \mathbb{R},\, t>0, \\
    G(x,0)=\delta(x)I_2, & x\in \mathbb{R},
  \end{dcases}
\end{equation}
where $\delta(x)$ is the Dirac delta function and $I_2$ is the $2\times 2$ identity matrix. Next, let $G^*=G^*(x,t)\in \mathbb{R}^{2\times 2}$ be the modified fundamental solution defined by the following equations:
\begin{equation}
  \label{SectionIII:CNS_Lag_per_Fundamental_Solution_Modified}
  \begin{dcases}
    \partial_t G^*+
    \begin{pmatrix}
      0    & -1 \\
      -c^2 & 0
    \end{pmatrix}
    \partial_x G^* =\frac{\nu}{2}\partial_{x}^{2}G^*, & x\in \mathbb{R},\, t>0, \\
    G^*(x,0)=\delta(x)I_2,                            & x\in \mathbb{R}.
  \end{dcases}
\end{equation}
This function $G^*$ has the following explicit expression~\cite[p.~1060]{Zeng94}:
\begin{equation}
  \label{SectionIII:CNS_Lag_per_Fundamental_Solution_Modified_Explicit_Formula}
  G^*(x,t)=\frac{1}{2(2\pi \nu t)^{1/2}}e^{-\frac{(x-ct)^2}{2\nu t}}
  \begin{pmatrix}
    1  & -\frac{1}{c} \\
    -c & 1
  \end{pmatrix}
  +\frac{1}{2(2\pi \nu t)^{1/2}}e^{-\frac{(x+ct)^2}{2\nu t}}
  \begin{pmatrix}
    1 & \frac{1}{c} \\
    c & 1
  \end{pmatrix}.
\end{equation}
The difference $G-G^*$ between the fundamental solution $G$ and its modification $G^*$ satisfies the following pointwise estimates~\cite[Theorem~5.8]{LZ97}: for any integer $k\geq 0$,
\begin{equation}
  \label{SectionIII:Bound_of_G-G*}
  \left| \partial_{x}^{k}G(x,t)-\partial_{x}^{k}G^*(x,t)-e^{-\frac{c^2}{\nu}t}\sum_{j=0}^{k}\delta^{(k-j)}(x)Q_j(t) \right| \leq C(t+1)^{-\frac{1}{2}}t^{-\frac{k+1}{2}}\left( e^{-\frac{(x-ct)^2}{Ct}}+e^{-\frac{(x+ct)^2}{Ct}} \right),
\end{equation}
where $\delta^{(k)}(x)$ is the $k$-th derivative of the Dirac delta function and $Q_j=Q_j(t)$ is a $2\times 2$ polynomial matrix; in the inequality above and in what follows, the symbol $C$ denotes a sufficiently large constant. Additionally, we have
\begin{equation}
  \label{SectionIII:Q0Q1}
  Q_0=
  \begin{pmatrix}
    1 & 0 \\
    0 & 0
  \end{pmatrix}
  , \quad Q_1=
  \begin{pmatrix}
    0                & -\frac{1}{\nu} \\
    -\frac{c^2}{\nu} & 0
  \end{pmatrix}.
\end{equation}
Moreover, the following refined pointwise estimates are also known~\cite[Theorem~1.3]{LZ09}: for any integer $k\geq 0$,
\begin{align}
  \label{SectionIII:Bound_of_G-G*_refined}
  \begin{aligned}
    & \partial_{x}^{k}G(x,t)-\partial_{x}^{k}G^*(x,t)-\gamma_1 \partial_{x}^{k+1}\frac{e^{-\frac{(x-ct)^2}{2\nu t}}}{(2\pi \nu t)^{1/2}}
    \begin{pmatrix}
      -1 & 0 \\
      0 & 1
    \end{pmatrix}
    -\gamma_2 \partial_{x}^{k+1}\frac{e^{-\frac{(x+ct)^2}{2\nu t}}}{(2\pi \nu t)^{1/2}}
    \begin{pmatrix}
      -1 & 0 \\
      0 & 1
    \end{pmatrix}
    -e^{-\frac{c^2}{\nu}t}\sum_{j=0}^{k}\delta^{(k-j)}(x)Q_j(t) \\
    & =O(1)(t+1)^{-1/2}t^{-\frac{k+1}{2}}e^{-\frac{(x-ct)^2}{Ct}}
    \begin{pmatrix}
      1  & -\frac{1}{c} \\
      -c & 1
    \end{pmatrix}
    +O(1)(t+1)^{-1/2}t^{-\frac{k+1}{2}}e^{-\frac{(x+ct)^2}{Ct}}
    \begin{pmatrix}
      1 & \frac{1}{c} \\
      c & 1
    \end{pmatrix} \\
    & \quad +O(1)(t+1)^{-1/2}t^{-\frac{k+2}{2}}\left( e^{-\frac{(x-ct)^2}{Ct}}+e^{-\frac{(x+ct)^2}{Ct}} \right),
  \end{aligned}
\end{align}
where $\gamma_i=(-1)^i \nu/(4c)$. Here, $O(1)f(x,t)$ is a (scalar) function whose absolute value is bounded by $C|f(x,t)|$. In what follows, we mainly use~\eqref{SectionIII:Bound_of_G-G*_refined} but we also use~\eqref{SectionIII:Bound_of_G-G*} to treat the case of $t\leq 1$.

Next, let $\mathcal{L}$ be the Laplace transform in $t$ and $\mathcal{L}^{-1}$ its inverse. Denote by $s$ the Laplace variable. Then let
\begin{equation}
  \label{SectionIII:Gb}
  G_T(x,t)\coloneqq \mathcal{L}^{-1}\left[ \frac{2}{\lambda+2}\mathcal{L}[G] \right](x,t), \quad G_R(x,t)\coloneqq (G-G_T)(x,t)
  \begin{pmatrix}
    1 & 0 \\
    0 & -1
  \end{pmatrix},
\end{equation}
where $\lambda=s/\sqrt{\nu s+c^2}$ ($\Re s>0$).\footnote{We take the branch such that $\Re \lambda>0$ when $\Re s>0$.}~Here, the subscripts ``T'' and ``R'' stand for ``transmission'' and ``reflection'' since $G_T$ and $G_R$ can be understood as ``Green's functions'' describing transmission and reflection of waves at the point mass; see~\cite[Remark~3.1]{Koike21}. Transmissive Green's function $G_T$ is connected to $G$ in the following way~\cite[Eq.~(40)]{Koike21}:
\begin{equation}
  \label{eq:GT_der}
  \partial_x G_T(x,t)=-2G(x,t)+2G_T(x,t)
\end{equation}
for $x>0$; therefore, we have
\begin{equation}
  \label{SectionIII:GTFormula}
  G_T(x,t)=2\int_{-\infty}^{0}e^{2z}G(x-z,t)\, dz \quad (x>0).
\end{equation}
Similar relations can be obtained also for $x<0$. By~\eqref{eq:GT_der}, reflective Green's function $G_R$ is related to $G_T$ by
\begin{equation}
  \label{SectionIII:GbIntByParts}
  G_R(x,t)=-\frac{1}{2}\partial_x G_T(x,t)
  \begin{pmatrix}
    1 & 0 \\
    0 & -1
  \end{pmatrix}
  =-\frac{1}{2}\left( \partial_x G(x,t)+\frac{1}{2}\partial_{x}^{2}G_T(x,t) \right)
  \begin{pmatrix}
    1 & 0 \\
    0 & -1
  \end{pmatrix}
\end{equation}
for $x>0$; again, a similar formula holds for $x<0$. Now, using~\eqref{SectionIII:Bound_of_G-G*} and~\eqref{SectionIII:GTFormula}, we obtain the following pointwise estimates for $G_T$~\cite[Eq.~(43)]{Koike21}:
\begin{equation}
  \label{SectionIII:Bounds_of_GT}
  |\partial_{x}^{k}G_T(x,t)|\leq C(t+1)^{-1/2}t^{-k/2}\left( e^{-\frac{(x-ct)^2}{Ct}}+e^{-\frac{(x+ct)^2}{Ct}} \right)+Ce^{-\frac{|x|+t}{C}}
\end{equation}
for $(x,t)\in \mathbb{R}_*\times (0,\infty)$ and for any integer $k\geq 0$.

Next, we write down an integral equation satisfied by the solution $(v,u,V)$ to~\eqref{eq:FundamentalEquations}. First, we recall the definition of the nonlinear term $N$ defined in~\eqref{eq:vector_form_variables}:
\begin{equation}
  \label{def:N}
  N=-p(v)+p(1)-c^2(v-1)-\nu \frac{v-1}{v}u_x.
\end{equation}
Then, we have the following proposition~\cite[Proposition~3.1]{Koike21}.

\begin{prop}
  \label{SectionIII:Proposition:IntegralEquation}
  Let $(v,u,V)$ be the global-in-time solution to~\eqref{eq:FundamentalEquations} belonging to the function spaces described in Theorem~\ref{thm:global_existence} below. Then it satisfies the following integral equation:
  \begin{align}
    \label{SectionIII:Proposition:IntegralEquation:IntegralEquation}
    \begin{aligned}
      \begin{pmatrix}
        v-1 \\
        u
      \end{pmatrix}
      (x,t) & =\int_{0}^{\infty}G(x-y,t)
      \begin{pmatrix}
        v_0-1 \\
        u_0
      \end{pmatrix}
      (y)\, dy+\int_{0}^{\infty}G_R(x+y,t)
      \begin{pmatrix}
        v_0-1 \\
        u_0
      \end{pmatrix}
      (y)\, dy \\
      & \quad +\int_{-\infty}^{0}G_T(x-y,t)
      \begin{pmatrix}
        v_0-1 \\
        u_0
      \end{pmatrix}
      (y)\, dy+G_T(x,t)
      \begin{pmatrix}
        0 \\
        V_0
      \end{pmatrix} \\
      & \quad +\int_{0}^{t}\int_{0}^{\infty}G(x-y,t-s)
      \begin{pmatrix}
        0 \\
        N_x
      \end{pmatrix}
      (y,s)\, dyds \\
      & \quad +\int_{0}^{t}\int_{0}^{\infty}G_R(x+y,t-s)
      \begin{pmatrix}
        0 \\
        N_x
      \end{pmatrix}
      (y,s)\, dyds+\int_{0}^{t}\int_{-\infty}^{0}G_T(x-y,t-s)
      \begin{pmatrix}
        0 \\
        N_x
      \end{pmatrix}
      (y,s)\, dyds \\
      & \quad +\int_{0}^{t}G_T(x,t-s)
      \begin{pmatrix}
        0 \\
        \llbracket N \rrbracket
      \end{pmatrix}
      (s)\, ds
    \end{aligned}
  \end{align}
  for $x>0$; a similar formula holds for $x<0$.
\end{prop}

Next, let
\begin{equation}
  \label{def:vi}
  v_i \coloneqq u_i-\theta_i-\xi_i-\gamma_{i'}\partial_x \theta_{i'}.
\end{equation}
Using Proposition~\ref{SectionIII:Proposition:IntegralEquation}, we can write down an integral equation for $v_i$. To do so, we introduce:
\begin{equation}
  g_i \coloneqq l_i G
  \begin{pmatrix}
    r_1 & r_2
  \end{pmatrix}
  , \quad g_{i}^{*}\coloneqq l_i G^*
  \begin{pmatrix}
    r_1 & r_2
  \end{pmatrix}
  , \quad g_{T,i}\coloneqq l_i G_T
  \begin{pmatrix}
    r_1 & r_2
  \end{pmatrix}
  , \quad g_{R,i}\coloneqq l_i G_R
  \begin{pmatrix}
    r_1 & r_2
  \end{pmatrix},
\end{equation}
where $r_i$ and $l_i$ are given by~\eqref{right_ev} and~\eqref{def:li}. Note that
\begin{equation}
  \label{SectionIII:gi*_Explicit_Formula}
  g_{1}^{*}=\frac{1}{(2\pi \nu t)^{1/2}}e^{-\frac{(x-ct)^2}{2\nu t}}
  \begin{pmatrix}
    1 & 0
  \end{pmatrix}
  , \quad  g_{2}^{*}=\frac{1}{(2\pi \nu t)^{1/2}}e^{-\frac{(x+ct)^2}{2\nu t}}
  \begin{pmatrix}
    0 & 1
  \end{pmatrix}.
\end{equation}
First, multiplying~\eqref{SectionIII:Proposition:IntegralEquation:IntegralEquation} by $l_i$ from the left, we obtain, for $x>0$,
\begin{align}
  \label{SectionIII:uiIntegralEquation}
  \begin{aligned}
    u_i(x,t)
    & =\int_{0}^{\infty}g_i(x-y,t)
    \begin{pmatrix}
      u_{01} \\
      u_{02}
    \end{pmatrix}
    (y)\, dy+\int_{0}^{\infty}g_{R,i}(x+y,t)
    \begin{pmatrix}
      u_{01} \\
      u_{02}
    \end{pmatrix}
    (y)\, dy \\
    & \quad +\int_{-\infty}^{0}g_{T,i}(x-y,t)
    \begin{pmatrix}
      u_{01} \\
      u_{02}
    \end{pmatrix}
    (y)\, dy+m_V g_{T,i}(x,t)\bm{1} \\
    & \quad +\int_{0}^{t}\int_{0}^{\infty}g_i(x-y,t-s)
    \begin{pmatrix}
      N_1 \\
      N_2
    \end{pmatrix}_x
    (y,s)\, dyds \\
    & \quad +\int_{0}^{t}\int_{0}^{\infty}g_{R,i}(x+y,t-s)
    \begin{pmatrix}
      N_1 \\
      N_2
    \end{pmatrix}_x
    (y,s)\, dyds+\int_{0}^{t}\int_{-\infty}^{0}g_{T,i}(x-y,t-s)
    \begin{pmatrix}
      N_1 \\
      N_2
    \end{pmatrix}_x
    (y,s)\, dyds \\
    & \quad +\int_{0}^{t}g_{T,i}(x,t-s)
    \begin{pmatrix}
      \llbracket N_1 \rrbracket \\
      \llbracket N_2 \rrbracket
    \end{pmatrix}
    (s)\, ds,
  \end{aligned}
\end{align}
where $u_{0i}$ is defined by~\eqref{def:u0i}, $m_V=l_i(0\, V_0)^{T}$, $\bm{1}=(1\, 1)^{T}$, and $N_i$ is defined by~\eqref{def:Ni}; a similar formula holds for $x<0$. Next, integral equations satisfied by $\theta_i$ and $\xi_i$ are easily obtained by looking at~\eqref{def:thetai},~\eqref{def:xii}, and~\eqref{def:xii_init}:
\begin{equation}
  \label{SectionIII:thetaiIntegralEquation}
  \theta_i(x,t)=\int_{-\infty}^{\infty}g_{i}^{*}(x-y,t)
  \begin{pmatrix}
    \theta_1 \\
    \theta_2
  \end{pmatrix}
  (y,0)\, dy-\frac{1}{2}\int_{0}^{t}\int_{-\infty}^{\infty}g_{i}^{*}(x-y,t-s)
  \begin{pmatrix}
    \theta_{1}^{2} \\
    \theta_{2}^{2}
  \end{pmatrix}_x
  (y,s)\, dyds
\end{equation}
and
\begin{equation}
  \label{eq:xii_IE}
  \xi_i(x,t)=
  -\int_{0}^{t}\int_{-\infty}^{\infty}g_{i}^{*}(x-y,t-s)
  \begin{pmatrix}
    \theta_{2}^{2}/2+\theta_1 \xi_1 \\
    \theta_{1}^{2}/2+\theta_2 \xi_2
  \end{pmatrix}_x
  (y,s)\, dyds.
\end{equation}
Hence, subtracting~\eqref{SectionIII:thetaiIntegralEquation}, its spatial derivative, and~\eqref{eq:xii_IE} from~\eqref{SectionIII:uiIntegralEquation}, we finally obtain an integral equation for $v_i$ (cf.~\cite[Eq.~(58)]{Koike21}): for $x>0$,
\begin{align}
  \label{eq:vi_IE}
  \begin{aligned}
    & v_i(x,t) \\
    & =\int_{-\infty}^{\infty}g_i(x-y,t)
    \begin{pmatrix}
      u_{01} \\
      u_{02}
    \end{pmatrix}
    (y)\, dy-\int_{-\infty}^{\infty}g_{i}^{*}(x-y,t)
    \begin{pmatrix}
      \theta_1 \\
      \theta_2
    \end{pmatrix}
    (y,0)\, dy-\gamma_{i'}\int_{-\infty}^{\infty}\partial_x g_{i'}^{*}(x-y,t)
    \begin{pmatrix}
      \theta_1 \\
      \theta_2
    \end{pmatrix}
    (y,0)\, dy \\
    & \quad +\int_{0}^{\infty}g_{R,i}(x+y,t)
    \begin{pmatrix}
      u_{01} \\
      u_{02}
    \end{pmatrix}
    (y)\, dy+\int_{-\infty}^{0}g_{R,i}(x-y,t)
    \begin{pmatrix}
      u_{02} \\
      u_{01}
    \end{pmatrix}
    (y)\, dy+m_V g_{T,i}(x,t)\bm{1} \\
    & \quad +\int_{0}^{t}\int_{-\infty}^{\infty}g_{i}^{*}(x-y,t-s)
    \begin{pmatrix}
      N_1-N_{1}^{*} \\
      N_2-N_{2}^{*}
    \end{pmatrix}_x
    (y,s)\, dyds \\
    & \quad +\int_{0}^{t}\int_{-\infty}^{\infty}(g_i-g_{i}^{*})(x-y,t-s)
    \begin{pmatrix}
      N_1 \\
      N_2
    \end{pmatrix}_x
    (y,s)\, dyds+\gamma_{i'}\int_{0}^{t}\int_{-\infty}^{\infty}\partial_x g_{i'}^{*}(x-y,t-s)
    \begin{pmatrix}
      \theta_{1}^{2}/2 \\
      \theta_{2}^{2}/2
    \end{pmatrix}_x
    (y,s)\, dyds \\
    & \quad +\int_{0}^{t}\int_{0}^{\infty}g_{R,i}(x+y,t-s)
    \begin{pmatrix}
      N_1 \\
      N_2
    \end{pmatrix}_x
    (y,s)\, dyds+\int_{0}^{t}\int_{-\infty}^{0}g_{R,i}(x-y,t-s)
    \begin{pmatrix}
      N_1 \\
      N_2
    \end{pmatrix}_x
    (y,s)\, dyds \\
    & \quad +\int_{0}^{t}g_{T,i}(x,t-s)
    \begin{pmatrix}
      \llbracket N_1 \rrbracket \\
      \llbracket N_2 \rrbracket
    \end{pmatrix}
    (s)\, ds,
  \end{aligned}
\end{align}
where
\begin{equation}
  \label{def:N*}
  N_{i}^{*}=-\theta_{1}^{2}/2-\theta_{2}^{2}/2-\theta_i \xi_i.
\end{equation}
A similar formula holds for $x<0$.

%HERE
Next, we give some remarks on the structure of $g_i$, $g_{T,i}$, and $g_{R,i}$. First, corresponding to~\eqref{SectionIII:Bound_of_G-G*}, we have for any integer $k\geq 0$,
\begin{equation}
  \label{eq:PWE_gi}
  \left| \partial_{x}^{k}g_i(x,t)-\partial_{x}^{k}g_{i}^{*}(x,t)-e^{-\frac{c^2}{\nu}t}\sum_{j=0}^{k}\delta^{(k-j)}(x)q_{ij}(t) \right| \leq C(t+1)^{-1/2}t^{-\frac{k+1}{2}}\left( e^{-\frac{(x-ct)^2}{Ct}}+e^{-\frac{(x+ct)^2}{Ct}} \right),
\end{equation}
where
\begin{equation}
  q_{ij}(t)=l_i Q_j(t)
  \begin{pmatrix}
    r_1 & r_2
  \end{pmatrix}.
\end{equation}
Next, corresponding to~\eqref{SectionIII:Bound_of_G-G*_refined}, we have for any integer $k\geq 0$,
\begin{align}
  \label{eq:PWE_gi_refined}
  \begin{aligned}
    & \left| \partial_{x}^{k}g_i(x,t)-\partial_{x}^{k}g_{i}^{*}(x,t)-\gamma_{i'}\partial_{x}^{k+1}g_{i'}^{*}(x,t)-e^{-\frac{c^2}{\nu}t}\sum_{j=0}^{k}\delta^{(k-j)}(x)q_{ij}(t) \right| \\
    & \leq Ct^{-\frac{k+2}{2}}e^{-\frac{(x-\lambda_i t)^2}{Ct}}+C(t+1)^{-1/2}t^{-\frac{k+2}{2}}e^{-\frac{(x-\lambda_{i'}t)^2}{Ct}},
  \end{aligned}
\end{align}
where $i'=3-i$. Note also that by~\eqref{eq:GT_der} and~\eqref{SectionIII:GbIntByParts}, we have
\begin{equation}
  \label{eq:gT_der}
  g_{T,i}(x,t)=g_i(x,t)+\frac{1}{2}\partial_x g_{T,i}(x,t)=g_i(x,t)+\frac{1}{2}\partial_x g_i(x,t)+\frac{1}{4}\partial_{x}^{2}g_{T,i}(x,t)
\end{equation}
and
\begin{equation}
  \label{eq:gR_relation}
  g_{R,i}(x,t)=\frac{1}{2}\partial_x g_{T,i}(x,t)
  \begin{pmatrix}
    0 & 1 \\
    1 & 0
  \end{pmatrix}
  =(g_{T,i}-g_i)(x,t)
  \begin{pmatrix}
    0 & 1 \\
    1 & 0
  \end{pmatrix}
  =\frac{1}{2}\left( \partial_x g_i+\frac{1}{2}\partial_{x}^{2}g_{T,i} \right)(x,t)
  \begin{pmatrix}
    0 & 1 \\
    1 & 0
  \end{pmatrix}
\end{equation}
for $x>0$; similar formulae hold for $x<0$. Moreover, by~\eqref{SectionIII:Bounds_of_GT}, we have
\begin{equation}
  \label{eq:PWE_gTi}
  |\partial_{x}^{k}g_{T,i}(x,t)|\leq C(t+1)^{-1/2}t^{-k/2}\left( e^{-\frac{(x-ct)^2}{Ct}}+e^{-\frac{(x+ct)^2}{Ct}} \right)+Ce^{-\frac{|x|+t}{C}}
\end{equation}
for $(x,t)\in \mathbb{R}_*\times (0,\infty)$ and for any integer $k\geq 0$. Finally, by~\eqref{SectionIII:CNS_Lag_per_Fundamental_Solution} and~\eqref{SectionIII:CNS_Lag_per_Fundamental_Solution_Modified}, we have
\begin{equation}
  \label{eq:Li_gi-gi*}
  L_i(g_i-g_{i}^{*})=\frac{\nu}{2}\partial_{x}^{2}g_{i'},
\end{equation}
where $L_i=\partial_t+\lambda_i \partial_x -(\nu/2)\partial_{x}^{2}$.

\subsection{Structure of the nonlinear term}
\label{sec:nonlinear}
We next examine the structure of the nonlinear term $N$ defined by~\eqref{def:N}. First, note that by Taylor's theorem, we have
\begin{equation}
  \label{eq:N_structure}
  N=-\frac{p''(1)}{2}(v-1)^2-\nu(v-1)u_x+O(|v-1|^3)+O(|(v-1)^2 u_x|).
\end{equation}
Using~\eqref{eq:decomposition} and~\eqref{def:vi}, the first term on the right-hand side multiplied by $p''(1)/(4c^2)$ is expressed in terms of $v_i$, $\theta_i$, and $\xi_i$ as follows:
\begin{equation}
  -\frac{p''(1)^2}{8c^2}(v-1)^2=-\frac{1}{2}(-u_1+u_2)^2=-\frac{1}{2}(-\theta_1-\xi_1-\gamma_2 \partial_x \theta_2-v_1+\theta_2+\xi_2+\gamma_1 \partial_x \theta_1+v_2)^2.
\end{equation}
Then, by expanding the square, we obtain
\begin{align}
  \label{eq:tau_squared}
  \begin{aligned}
    & -\frac{p''(1)^2}{8c^2}(v-1)^2 \\
    & =-\frac{1}{2}\theta_{1}^{2}-\theta_1 \xi_1-\gamma_2 \theta_1 \partial_x \theta_2-\theta_1 v_1+\theta_1 \theta_2+\theta_1 \xi_2+\gamma_1 \theta_1 \partial_x \theta_1+\theta_1 v_2 \\
    & \quad -\frac{1}{2}\xi_{1}^{2}-\gamma_2 \xi_1 \partial_x \theta_2-\xi_1 v_1+\xi_1 \theta_2+\xi_1 \xi_2+\gamma_1 \xi_1 \partial_x \theta_1+\xi_1 v_2 \\
    & \quad -\frac{\gamma_{2}^{2}}{2}(\partial_x \theta_2)^2-\gamma_2(\partial_x \theta_2)v_1+\gamma_2(\partial_x \theta_2)\theta_2+\gamma_2(\partial_x \theta_2)\xi_2+\gamma_1 \gamma_2(\partial_x \theta_2)\partial_x \theta_1+\gamma_2(\partial_x \theta_2)v_2 \\
    & \quad -\frac{1}{2}v_{1}^{2}+v_1 \theta_2+v_1 \xi_2+\gamma_1 v_1 \partial_x \theta_1+v_1 v_2-\frac{1}{2}\theta_{2}^{2}-\theta_2 \xi_2-\gamma_1 \theta_2 \partial_x \theta_1-\theta_2 v_2 \\
    & \quad -\frac{1}{2}\xi_{2}^{2}-\gamma_1 \xi_2 \partial_x \theta_1-\xi_2 v_2-\frac{\gamma_{1}^{2}}{2}(\partial_x \theta_1)^2-\gamma_1(\partial_x \theta_1)v_2-\frac{1}{2}v_{2}^{2}.
  \end{aligned}
\end{align}

\subsection{Global-in-time existence of solutions and decay estimates of their derivatives}
\label{sec:global_existence}
We next state a theorem on the existence of the global-in-time solution to~\eqref{eq:FundamentalEquations}. Its proof is very similar to that of~\cite[Theorem~1.1]{Koike21}; hence we omit the proof. Note, however, that an additional compatibility condition is required since we consider here solutions with higher regularity. See~\eqref{def:C1C2} for the definitions of $\mathcal{C}_1(v,u)$ and $\mathcal{C}_2(v,u)$ used below.

\begin{thm}
  \label{thm:global_existence}
  Let $v_0-1,u\in H^6(\mathbb{R}_*)$ and $V_0 \in \mathbb{R}$. Assume that they satisfy the following compatibility conditions:
  \begin{equation}
    u_0(0_{\pm})=V_0, \quad \mathcal{C}_1(v_0,u_0)_x(0_{\pm})=\llbracket \mathcal{C}_1(v_0,u_0) \rrbracket, \quad \mathcal{C}_2(v_0,u_0)_x(0_{\pm})=\llbracket \mathcal{C}_2(v_0,u_0) \rrbracket.
  \end{equation}
  Then there exist $\varepsilon_0,C>0$ such that if
  \begin{equation}
    \label{smallness}
    \varepsilon \coloneqq ||v_0-1||_6+||u_0||_6 \leq \varepsilon_0,
  \end{equation}
  then~\eqref{eq:FundamentalEquations} has the unique classical solution
  \begin{align}
    v-1 & \in C([0,\infty);H^6(\mathbb{R}_*))\cap C^1([0,\infty);H^5(\mathbb{R}_*)), \\
    u   & \in C([0,\infty);H^6(\mathbb{R}_*))\cap C^1([0,\infty);H^4(\mathbb{R}_*)), \\
    u_x & \in L^2(0,\infty;H^6(\mathbb{R}_*)), \\
    V   & \in C^3([0,\infty))
  \end{align}
  satisfying
  \begin{equation}
    \label{solution_smallness}
    ||(v-1)(t)||_6+||u(t)||_6+\left( \int_{0}^{\infty}||u_x(s)||_{6}^{2}\, ds \right)^{1/2}+\sum_{k=0}^{3}|\partial_{t}^{k}V(t)|\leq C\varepsilon \quad (t\geq 0).
  \end{equation}
\end{thm}

Next, we state a theorem on pointwise estimates of $\partial_x(u_i-\theta_i)$. In the case of the Cauchy problem, a similar theorem is given in~\cite[Theorem~2.6]{LZ97} (see also~\cite[Remark~2.8]{LZ97}). The proof of Theorem~\ref{thm:PWE_derivative} below is basically just a combination of the proofs of~\cite[Theorem~2.6]{LZ97} and~\cite[Theorem~1.2]{Koike21}; hence we omit its proof for brevity.

\begin{thm}
  \label{thm:PWE_derivative}
  Let $v_0-1,u_0 \in H^6(\mathbb{R}_*)$ and $V_0 \in \mathbb{R}$. Assume that they satisfy the compatibility conditions stated in Theorem~\ref{thm:global_existence}. Then there exist $\delta_{0}',C>0$ such that if
  \begin{align}
    \label{ass:smallness_existence}
    \delta'
    & \coloneqq \sum_{i=1}^{2}\biggl\{ ||u_{0i}||_6+||u_{0i}^{-}||_{L^1(-\infty,0)}+||u_{0i}^{+}||_{L^1(0,\infty)}+\sup_{x\in \mathbb{R}_*}\left[ (|x|+1)^{3/2}|u_{0i}(x)| \right] \\
    & \phantom{\coloneqq \sum_{i}^{2}[} \quad +\sup_{x>0}\left[ (|x|+1)(|u_{0i}^{-}(-x)|+|u_{0i}^{+}(x)|) \right] \biggr\} \leq \delta_{0}',
  \end{align}
  where $u_{0i}^{\pm}$ are defined by~\eqref{def:u0ipm}, then
  \begin{equation}
    |\partial_x (u_i-\theta_i)(x,t)|\leq C\delta'(t+1)^{-1/2}\Phi_i(x,t) \quad (x\in \mathbb{R}_*,t\geq 0;i=1,2)
  \end{equation}
  and
  \begin{equation}
    |\partial_{x}^{2}u_i(\cdot,t)|_{\infty}\leq C\delta'(t+1)^{-3/2} \quad (t\geq 0;i=1,2),
  \end{equation}
  where $\Phi_i$ is defined by~\eqref{def:Phi}.
\end{thm}

\subsection{Properties of inter-diffusion waves}
\label{AppendixA}
In this section, we prove some properties of the inter-diffusion wave $\xi_i$ defined by~\eqref{def:xii} and~\eqref{def:xii_init}. Let us first introduce some notations. For $\lambda \in \mathbb{R}$ and $\alpha,\mu>0$, let
\begin{equation}
  \label{def:Theta}
  \Theta_{\alpha}(x,t;\lambda,\mu)\coloneqq (t+1)^{-\alpha/2}e^{-\frac{(x-\lambda(t+1))^2}{\mu(t+1)}}
\end{equation}
and
\begin{equation}
  \psi_{\alpha}(x,t;\lambda)\coloneqq [(x-\lambda(t+1))^2+(t+1)]^{-\alpha/2}.
\end{equation}
Note that
\begin{equation}
  \label{eq:Theta_psi}
  \Theta_{\alpha}(x,t;\lambda,\mu)\leq C\psi_{\alpha}(x,t;\lambda).
\end{equation}
Moreover, the diffusion wave $\theta_i$ satisfies
\begin{equation}
  \label{eq:theta_Theta}
  |\partial_{x}^{k}\theta_i(x,t)|\leq C\delta \Theta_{1+k}(x,t;\lambda_i,\nu^*)
\end{equation}
for any integer $k\geq0$ when $\delta$ defined in~\eqref{ass:smallness} satisfies $\delta \leq 1$; see~\eqref{eq:theta_explicit} (and note that $|M_i|\leq C\delta$). \textit{Here and in what follows, the symbols $C$ and $\nu^*$ denote generic large constants.}

The following lemma gives accurate pointwise estimates of $\xi_i$.

\begin{lem}
  \label{lem:xi}
  Suppose that $\delta$ defined in~\eqref{ass:smallness} is sufficiently small. Then we have
  \begin{equation}
    \label{lem:xi:eq}
    |\partial_{x}^{k}\xi_i(x,t)-(-1)^i (4c)^{-1}\partial_{x}^{k}\theta_{i'}^{2}(x,t)|\leq C\delta^2 (t+1)^{-k/2}\psi_{3/2}(x,t;\lambda_i)
  \end{equation}
  for any integer $k\geq 0$, where $i'=3-i$. In particular, we have
  \begin{equation}
    |\partial_{x}^{k}\xi_i(x,t)|\leq C\delta^2 (t+1)^{-k/2}[\psi_{3/2}(x,t;\lambda_i)+\Theta_2(x,t;\lambda_{i'},\nu^*)].
  \end{equation}
\end{lem}

\begin{proof}
  We consider the case of $i=1$ and $k=0$ since other cases can be treated similarly. We also assume that $t\geq 4$ since otherwise the lemma is easy to prove. Recall that $\xi_i(x,t)$ satisfies the integral equation~\eqref{eq:xii_IE}. Taking this into mind, let $\xi_1(x,t)=\zeta_1(x,t)+\eta_1(x,t)$, where
  \begin{equation}
    \label{eq:zetai_IE}
    \zeta_1(x,t)=
    -\int_{0}^{t}\int_{-\infty}^{\infty}g_{1}^{*}(x-y,t-s)
    \begin{pmatrix}
      \theta_{2}^{2}/2 \\
      \theta_{1}^{2}/2
    \end{pmatrix}_x
    (y,s)\, dyds
  \end{equation}
  and
  \begin{equation}
    \label{eq:etai_IE}
    \eta_1(x,t)=
    -\int_{0}^{t}\int_{-\infty}^{\infty}g_{1}^{*}(x-y,t-s)
    \begin{pmatrix}
      \theta_1 \xi_1 \\
      \theta_2 \xi_2
    \end{pmatrix}_x
    (y,s)\, dyds.
  \end{equation}
  We remind the reader that we have an explicit formula~\eqref{SectionIII:gi*_Explicit_Formula} for $g_{1}^{*}$.

  Let us first consider $\zeta_1(x,t)$. Set $I(x,t)=-(2\sqrt{2\pi \nu})\zeta_1(x,t)$. By Lemma~\ref{lem:3.4_pre} in the appendix, we have $I(x,t)=(2c)^{-1}\sqrt{2\pi \nu}\theta_{2}^{2}(x,t)+I_1(x,t)+I_2(x,t)$, where
  \begin{equation}
    I_1(x,t)=\int_{0}^{t^{1/2}}\int_{-\infty}^{\infty}\partial_x \left\{ (t-s)^{-1/2}e^{-\frac{(x-y-c(t-s))^2}{2\nu(t-s)}} \right\} \theta_{2}^{2}(y,s)\, dyds,
  \end{equation}
  \begin{align}
    I_2(x,t)
    & =-(2c)^{-1}\int_{-\infty}^{\infty}(t-t^{1/2})^{-1/2}e^{-\frac{(x-y-c(t-\sqrt{t}))^2}{2\nu(t-\sqrt{t})}}\theta_{2}^{2}(y,t^{1/2})\, dy \\
    & \quad -(2c)^{-1}\int_{t^{1/2}}^{t}\int_{-\infty}^{\infty}(t-s)^{-1/2}e^{-\frac{(x-y-c(t-s))^2}{2\nu(t-s)}}L_2 \theta_{2}^{2}(y,s)\, dyds \\
    & \eqqcolon I_{21}(x,t)+I_{22}(x,t),
  \end{align}
  and $L_2=\partial_t-c\partial_x-(\nu/2)\partial_{x}^{2}$. For $I_1(x,t)$, we can show that (see the analysis of $I_1$ in~\cite[p.~22]{LZ97})
  \begin{equation}
    |I_1(x,t)|\leq C\delta^2 \Theta_{3/2}(x,t;c,\nu^*)\leq C\delta^2 \psi_{3/2}(x,t;c).
  \end{equation}
  We can also show that (see the analysis of $I_{21}$ in~\cite[p.~23]{LZ97})
  \begin{equation}
    |I_{21}(x,t)|\leq C\delta^2 \Theta_{3/2}(x,t;c,\nu^*)\leq C\delta^2 \psi_{3/2}(x,t;c).
  \end{equation}
  Moreover, since 
  \begin{equation}
    L_2 \theta_{2}^{2}=-2\partial_x(\theta_{2}^{3}/3)-\nu(\partial_x \theta_2)^2
  \end{equation}
  by~\eqref{def:thetai}, we have $|(L_2 \theta_{2}^{2})(x,t)|\leq C\delta^2 \Theta_4(x,t;-c,\nu^*)$. Hence, applying Lemma~\ref{lem:I221} (with $\alpha=4$), we obtain
  \begin{equation}
    |I_{22}(x,t)|\leq C\delta^2 \psi_{3/2}(x,t;c).
  \end{equation}
  Thus, we have proved
  \begin{equation}
    \label{lem:xi:proof:eq1}
    |\zeta_1(x,t)+(4c)^{-1}\theta_2^{2}(x,t)|\leq C\delta^2 \psi_{3/2}(x,t;c).
  \end{equation}

  Let us next consider $\eta_1(x,t)$. Note that $\eta_1$ is the solution to
  \begin{equation}
    \partial_t \eta_1+c\partial_x \eta_1+\partial_x(\theta_1 \eta_1)=\frac{\nu}{2}\partial_{x}^{2}\eta_1-\partial_x(\theta_1 \zeta_1), \quad x\in \mathbb{R},t>0
  \end{equation}
  and
  \begin{equation}
    \eta_1(x,0)=0, \quad x\in \mathbb{R}.
  \end{equation}
  Hence, $\eta_1$ has a series representation:
  \begin{equation}
    \eta_1=\sum_{n=1}^{\infty}\eta_{1;n},
  \end{equation}
  where $\eta_{1;1}$ is the solution to
  \begin{equation}
    \begin{dcases}
      \partial_t \eta_{1;1}+c\partial_x \eta_{1;1}=\frac{\nu}{2}\partial_{x}^{2}\eta_{1;1}-\partial_x (\theta_1 \zeta_1), & x\in \mathbb{R},t>0, \\
      \eta_{1;1}(x,0)=0, & x\in \mathbb{R}
    \end{dcases}
  \end{equation}
  and $\eta_{1;n+1}$ ($n\geq 1$) are the solutions to
  \begin{equation}
    \begin{dcases}
      \partial_t \eta_{1;n+1}+c\partial_x \eta_{1;n+1}=\frac{\nu}{2}\partial_{x}^{2}\eta_{1;n+1}-\partial_x (\theta_1 \eta_{1;n}), & x\in \mathbb{R},t>0, \\
      \eta_{1;n+1}(x,0)=0, & x\in \mathbb{R}.
    \end{dcases}
  \end{equation}
  Note that
  \begin{equation}
    \eta_{1,1}(x,t)=
    -\int_{0}^{t}\int_{-\infty}^{\infty}\partial_x g_{1}^{*}(x-y,t-s)
    \begin{pmatrix}
      \theta_1 \zeta_1 \\
      \theta_2 \zeta_2
    \end{pmatrix}
    (y,s)\, dyds
  \end{equation}
  and that~\eqref{lem:xi:proof:eq1} implies
  \begin{equation}
    |(\theta_1 \zeta_1)(x,t)|\leq A_1 \delta^3 \Theta_{5/2}(x,t;c,2\nu')
  \end{equation}
  for some $A_1>0$ and $\nu'(>\nu)$. Hence, by~\cite[Lemma~3.2]{LZ97} (with $\alpha=0$ and $\beta=5/2$), we obtain
  \begin{equation}
    |\eta_{1;1}(x,t)|\leq M A_1 \delta^3 \Theta_{3/2}(x,t;c,2\nu')
  \end{equation}
  for some $M>0$. This means that, for $k=1$, we have
  \begin{equation}
    \label{zeta:induction}
    |\eta_{1;k}(x,t)|\leq MA_1 (MA_0)^{k-1}\delta^{k+2}\Theta_{3/2}(x,t;c,2\nu').
  \end{equation}
  Here, we define $A_0>0$ as a constant such that
  \begin{equation}
    |\theta_1(x,t)\Theta_{3/2}(x,t;c,2\nu')|\leq A_0 \delta \Theta_{5/2}(x,t;c,2\nu').
  \end{equation}
  We next show that~\eqref{zeta:induction} holds for $k=n+1$ assuming that it holds for $k=n$. Under this induction hypothesis, we have
  \begin{equation}
    |(\theta_1 \eta_{1;n})(x,t)|\leq A_1 (MA_0)^{n}\delta^{n+3}\Theta_{5/2}(x,t;c,2\nu').
  \end{equation}
  Now, using~\cite[Lemma~3.2]{LZ97} again, this time to the integral representation
  \begin{equation}
    \eta_{1,n+1}(x,t)=
    -\int_{0}^{t}\int_{-\infty}^{\infty}\partial_x g_{1}^{*}(x-y,t-s)
    \begin{pmatrix}
      \theta_1 \eta_{1;n} \\
      \theta_2 \eta_{1;n}
    \end{pmatrix}
    (y,s)\, dyds,
  \end{equation}
  we obtain
  \begin{equation}
    |\eta_{1;n+1}(x,t)|\leq MA_1 (MA_0)^n \delta^{n+3}\Theta_{3/2}(x,t;c,2\nu').
  \end{equation}
  This shows that \eqref{zeta:induction} holds for all $k\geq 1$. Hence, by taking $\delta$ sufficiently small, we have
  \begin{equation}
    |\eta_1(x,t)|=\left| \sum_{n=1}^{\infty}\eta_{1;n}(x,t) \right|\leq C\delta^3 \Theta_{3/2}(x,t;c,2\nu')\leq C\delta^3 \psi_{3/2}(x,t;c).
  \end{equation}
  Combining this with~\eqref{lem:xi:proof:eq1}, we obtain~\eqref{lem:xi:eq}.
\end{proof}

Next, we prove two lemmas that we shall use when analyzing the long-time behavior of $(\xi_1+\xi_2)(0,t)$. For $t\geq 1$, let
\begin{equation}
  \mathcal{V}_i(t)\coloneqq (-1)^i \frac{\sqrt{\nu}}{4c\sqrt{2\pi}}\int_{t^{1/2}}^{t}\int_{-\infty}^{\infty}(t-s)^{-1/2}e^{-\frac{(y+\lambda_i(t-s))^2}{2\nu(t-s)}}(\partial_x \theta_{i'})^2(y,s)\, dyds
\end{equation}
and
\begin{equation}
  \label{def:V}
  \mathcal{V}(t)\coloneqq \frac{2c^2}{p''(1)}(\mathcal{V}_1+\mathcal{V}_2)(t).
\end{equation}
Here, we remind the reader that $i'=3-i$. We also define
\begin{equation}
  \label{def:W}
  \mathcal{W}(t)\coloneqq \frac{c(M_{1}^{2}-M_{2}^{2})}{4\pi p''(1)\sqrt{2\pi \nu}}\int_{t^{1/2}}^{t}\int_{-\infty}^{\infty}(t-s)^{-1/2}e^{-\frac{(y+c(t-s))^2}{2\nu(t-s)}}[\partial_x \Theta_1(y,s;-c,2\nu)]^2\, dyds
\end{equation}
for $t\geq 1$. First, we prove the following lemma.

\begin{lem}
  \label{lem:xi_lower}
  Suppose that $\delta$ defined in~\eqref{ass:smallness} is sufficiently small and $t\geq 4$. Then we have
  \begin{equation}
    \label{lem:eq1:xi_lower}
    |\xi_i(0,t)-\mathcal{V}_i(t)|\leq C\delta^2(t+1)^{-2}.
  \end{equation}
  In particular, we have
  \begin{equation}
    \label{lem:eq4:xi_lower}
    \left| \frac{2c^2}{p''(1)}(\xi_1+\xi_2)(0,t)-\mathcal{V}(t) \right| \leq C\delta^2(t+1)^{-2}.
  \end{equation}
  Moreover, we have\footnote{We set $\sgn(0)=0$.}
  \begin{equation}
    \label{lem:eq2:xi_lower}
    C^{-1}|M_{1}^{2}-M_{2}^{2}|(t+1)^{-3/2}\leq (\sgn(M_{1}^{2}-M_{2}^{2}))\mathcal{W}(t)
  \end{equation}
  and
  \begin{equation}
    \label{lem:eq3:xi_lower}
    |\mathcal{V}(t)-\mathcal{W}(t)|\leq C\delta^3(t+1)^{-3/2}.
  \end{equation}
\end{lem}

\begin{proof}
  For~\eqref{lem:eq1:xi_lower}, let us only consider the case of $i=1$ since the other case is similar. By looking at the proof of Lemma~\ref{lem:xi}, we see that
  \begin{equation}
    \label{eq:xi1_I22}
    |\xi_1(0,t)+(8\pi \nu)^{-1/2}I_{22}(0,t)|\leq C\delta^2 e^{-\frac{t}{C}},
  \end{equation}
  where
  \begin{equation}
    I_{22}(0,t)=-(2c)^{-1}\int_{t^{1/2}}^{t}\int_{-\infty}^{\infty}(t-s)^{-1/2}e^{-\frac{(y+c(t-s))^2}{2\nu(t-s)}}L_2 \theta_{2}^{2}(y,s)\, dyds.
  \end{equation}
  Since
  \begin{equation}
    L_2 \theta_{2}^{2}=-2\partial_x(\theta_{2}^{3}/3)-\nu(\partial_x \theta_2)^2,
  \end{equation}
  we have
  \begin{align}
    I_{22}(0,t)
    & =\frac{\nu}{2c}\int_{t^{1/2}}^{t}\int_{-\infty}^{\infty}(t-s)^{-1/2}e^{-\frac{(y+c(t-s))^2}{2\nu(t-s)}}(\partial_x \theta_{2})^2(y,s)\, dyds \\
    & \quad +\frac{1}{3c}\restr{\int_{t^{1/2}}^{t}\int_{-\infty}^{\infty}\partial_x \left\{ (t-s)^{-1/2}e^{-\frac{(x-y-c(t-s))^2}{2\nu(t-s)}} \right\} \theta_{2}^{3}(y,s)\, dyds}{x=0} \\
    & \eqqcolon J_1(t)+J_2(t).
  \end{align}
  For $J_2(t)$, by mimicking the estimate of $I(x,t)$ in the proof of Lemma~\ref{lem:xi}, we obtain
  \begin{equation}
    \label{eq:J2}
    |J_2(t)|\leq C\delta^3(t+1)^{-2}.
  \end{equation}
  Now, since $\mathcal{V}_1(t)=-(8\pi \nu)^{-1/2}J_1(t)$, we obtain~\eqref{lem:eq1:xi_lower} by~\eqref{eq:xi1_I22} and~\eqref{eq:J2}.

  The bound~\eqref{lem:eq2:xi_lower} is easily proved by observing that the integrand in~\eqref{def:W} is non-negative (cf.~the proof of Lemma~\ref{lem:I221}). We next note that by~\eqref{eq:theta_explicit}, if $\delta$ and hence $M_i$ is sufficiently small, we have
  \begin{equation}
    \left| \partial_x \theta_i(x,t)-\frac{M_i}{\sqrt{2\pi \nu}}\partial_x \Theta_1(x,t;\lambda_i,2\nu) \right| \leq C\delta^2 \Theta_2(x,t;\lambda_i,\nu^*).
  \end{equation}
  Hence, by Lemma~\ref{lem:I221} (with $\alpha=4$), we obtain
  \begin{equation}
    |\mathcal{V}(t)-\mathcal{W}(t)|\leq C\delta^3 \int_{t^{1/2}}^{t}\int_{-\infty}^{\infty}(t-s)^{-1/2}e^{-\frac{(y+c(t-s))^2}{2\nu(t-s)}}\Theta_4(y,s;-c,\nu^*)\, dyds\leq C\delta^3(t+1)^{-3/2}.
  \end{equation}
  This proves~\eqref{lem:eq3:xi_lower}.
\end{proof}

Next, we prove a lemma related to a possible cancellation in the sum $(\xi_1+\xi_2)(0,t)$.

\begin{lem}
  \label{lem:xi_cancellation}
  Suppose that $\delta$ defined in~\eqref{ass:smallness} satisfies $\delta \leq 1$. Then for $t\geq 1$, we have
  \begin{equation}
    \label{lem:eq:xi_cancellation}
    \begin{dcases}
      \mathcal{V}(t)=0 & (M_1+M_2=0), \\
      |\mathcal{V}(t)|\leq C\delta^2(t+1)^{-2} & (M_1-M_2=0).
    \end{dcases}
  \end{equation}
\end{lem}

\begin{proof}
  The case when $M_1+M_2=0$ is easy. Let $M\coloneqq M_1$. Then by~\eqref{eq:theta_explicit},
  \begin{align}
    \theta_1(-x,t)
    & =\frac{\sqrt{\nu}}{\sqrt{2(t+1)}}\left( e^{\frac{M}{\nu}}-1 \right) e^{-\frac{(-x-c(t+1))^2}{2\nu(t+1)}}\left[ \sqrt{\pi}+\left( e^{\frac{M}{\nu}}-1 \right) \int_{\frac{-x-c(t+1)}{\sqrt{2\nu(t+1)}}}^{\infty}e^{-y^2}\, dy \right]^{-1} \\
    & =\frac{\sqrt{\nu}}{\sqrt{2(t+1)}}\left( e^{\frac{M}{\nu}}-1 \right) e^{-\frac{(x+c(t+1))^2}{2\nu(t+1)}}\left[ \sqrt{\pi}+\left( e^{\frac{M}{\nu}}-1 \right) \left( \sqrt{\pi}-\int_{\frac{x+c(t+1)}{\sqrt{2\nu(t+1)}}}^{\infty}e^{-y^2}\, dy \right) \right]^{-1} \\
    & =-\frac{\sqrt{\nu}}{\sqrt{2(t+1)}}\left( e^{-\frac{M}{\nu}}-1 \right) e^{-\frac{(x+c(t+1))^2}{2\nu(t+1)}}\left[ \sqrt{\pi}+\left( e^{-\frac{M}{\nu}}-1 \right) \int_{\frac{x+c(t+1)}{\sqrt{2\nu(t+1)}}}^{\infty}e^{-y^2}\, dy \right]^{-1} \\
    & =-\theta_2(x,t).
  \end{align}
  Hence, it follows that $\mathcal{V}_1(t)=-\mathcal{V}_2(t)$, and we have $\mathcal{V}(t)=0$.

  The case when $M\coloneqq M_1=M_2$ is more subtle. By the change of variable $y=-z$, we obtain
  \begin{equation}
    \mathcal{V}(t)=\frac{c\sqrt{\nu}}{2p''(1)\sqrt{2\pi}}\int_{t^{1/2}}^{t}\int_{-\infty}^{\infty}(t-s)^{-1/2}e^{-\frac{(y+c(t-s))^2}{2\nu(t-s)}}\left[ (\partial_x \theta_{1})^{2}(-y,s)-(\partial_x \theta_{2})^{2}(y,s) \right] \, dyds.
  \end{equation}
  Note that by~\eqref{eq:theta_explicit}, we have
  \begin{equation}
    \theta_1(y+2c(s+1),s)=\frac{\sqrt{\nu}}{\sqrt{2(s+1)}}\left( e^{\frac{M}{\nu}}-1 \right) e^{-\frac{(y+c(s+1))^2}{2\nu(s+1)}}\left[ \sqrt{\pi}+\left( e^{\frac{M}{\nu}}-1 \right) \int_{\frac{y+c(s+1)}{\sqrt{2\nu(s+1)}}}^{\infty}e^{-z^2}\, dz \right]^{-1}=\theta_2(y,s).
  \end{equation}
  Hence, we obtain the relation
  \begin{equation}
    \int_{-\infty}^{\infty}\left[ (\partial_x \theta_{1})^{2}(-y,s)-(\partial_x \theta_{2})^{2}(y,s) \right] \, dy=0.
  \end{equation}
  Taking this into consideration, we define $F=F(y,s)$ by
  \begin{equation}
    \label{eq:antiderivative}
      F(y,s)\coloneqq \int_{-\infty}^{y}\left[ (\partial_x \theta_{1})^{2}(-z,s)-(\partial_x \theta_{2})^{2}(z,s) \right] \, dz=-\int_{y}^{\infty}\left[ (\partial_x \theta_{1})^{2}(-z,s)-(\partial_x \theta_{2})^{2}(z,s) \right] \, dz.
  \end{equation}
  Then, we have
  \begin{equation}
    \mathcal{V}(t)=\frac{c\sqrt{\nu}}{2p''(1)\sqrt{2\pi}}\int_{t^{1/2}}^{t}\int_{-\infty}^{\infty}(t-s)^{-1/2}e^{-\frac{(y+c(t-s))^2}{2\nu(t-s)}}\partial_y F(y,s) \, dyds.
  \end{equation}
  By mimicking the estimate of $I(x,t)$ in the proof of Lemma~\ref{lem:xi} (cf.~the estimate of $J_2(t)$ in the proof of Lemma~\ref{lem:xi_lower}), we obtain
  \begin{equation}
    \label{eq:approx_barV}
    |\mathcal{V}(t)-\bar{\mathcal{V}}(t)|\leq C\delta^2 e^{-\frac{t}{C}},
  \end{equation}
  where
  \begin{equation}
    \label{def:Vbar}
    \bar{\mathcal{V}}(t)=-\frac{\sqrt{\nu}}{4p''(1)\sqrt{2\pi}}\int_{t^{1/2}}^{t}\int_{-\infty}^{\infty}(t-s)^{-1/2}e^{-\frac{(y+c(t-s))^2}{2\nu(t-s)}}L_2 F(y,s) \, dyds
  \end{equation}
  and $L_2=\partial_s-c\partial_y-(\nu/2)\partial_{y}^{2}$. Now, note that
  \begin{align}
    \label{eq:L2F}
    \begin{aligned}
      L_2 F(y,s)
      & =2\int_{-\infty}^{y}\{ [(\partial_x \theta_1)(\partial_x \partial_t \theta_1)](-z,s)-[(\partial_x \theta_2)(\partial_x \partial_t \theta_2)](z,s) \} \, dz \\
      & \quad -c(\partial_x \theta_1)^2(-y,s)+c(\partial_x \theta_2)^2(y,s)+\nu[(\partial_x \theta_1)(\partial_{x}^{2}\theta_1)](-y,s)+\nu[(\partial_x \theta_2)(\partial_{x}^{2}\theta_2)](y,s) \\
      & =2\int_{-\infty}^{y}\left\{ \left[ (\partial_x \theta_1)\left( -c\partial_{x}^{2}\theta_1+\frac{\nu}{2}\partial_{x}^{3}\theta_1-\partial_{x}^{2}\left( \frac{\theta_{1}^{2}}{2} \right) \right) \right](-z,s) \right. \\
      & \phantom{2\int_{-\infty}^{y}\biggl\{ \biggl[} \quad - \left. \left[ (\partial_x \theta_2)\left( c\partial_{x}^{2}\theta_2+\frac{\nu}{2}\partial_{x}^{3}\theta_2-\partial_{x}^{2}\left( \frac{\theta_{2}^{2}}{2} \right) \right) \right](z,s) \right\}\, dz \\
      & \quad -c(\partial_x \theta_1)^2(-y,s)+c(\partial_x \theta_2)^2(y,s)+\nu[(\partial_x \theta_1)(\partial_{x}^{2}\theta_1)](-y,s)+\nu[(\partial_x \theta_2)(\partial_{x}^{2}\theta_2)](y,s) \\
      & =2\int_{-\infty}^{y}\left\{ \left[ (\partial_x \theta_1)\left( \frac{\nu}{2}\partial_{x}^{3}\theta_1-\partial_{x}^{2}\left( \frac{\theta_{1}^{2}}{2} \right) \right) \right](-z,s)-\left[ (\partial_x \theta_2)\left( \frac{\nu}{2}\partial_{x}^{3}\theta_2-\partial_{x}^{2}\left( \frac{\theta_{2}^{2}}{2} \right) \right) \right](z,s) \right\}\, dz \\
      & \quad +\nu[(\partial_x \theta_1)(\partial_{x}^{2}\theta_1)](-y,s)+\nu[(\partial_x \theta_2)(\partial_{x}^{2}\theta_2)](y,s).
    \end{aligned}
  \end{align}
  By a similar computation, we also have
  \begin{align}
    \label{eq:L2F_negative}
    \begin{aligned}
      L_2 F(y,s)
      & =-2\int_{y}^{\infty}\left\{ \left[ (\partial_x \theta_1)\left( \frac{\nu}{2}\partial_{x}^{3}\theta_1-\partial_{x}^{2}\left( \frac{\theta_{1}^{2}}{2} \right) \right) \right](-z,s)-\left[ (\partial_x \theta_2)\left( \frac{\nu}{2}\partial_{x}^{3}\theta_2-\partial_{x}^{2}\left( \frac{\theta_{2}^{2}}{2} \right) \right) \right](z,s) \right\}\, dz \\
      & \quad +\nu[(\partial_x \theta_1)(\partial_{x}^{2}\theta_1)](-y,s)+\nu[(\partial_x \theta_2)(\partial_{x}^{2}\theta_2)](y,s).
    \end{aligned}
  \end{align}
  If we show that
  \begin{equation}
    \label{eq:antiderivative_bound}
    |L_2 F(y,s)|\leq C\delta^2 \Theta_5(y,s;-c,\nu^*),
  \end{equation}
  then applying Lemma~\ref{lem:I221} (with $\alpha=5$) to~\eqref{def:Vbar} implies
  \begin{equation}
    \label{eq:barV_bound}
    |\bar{\mathcal{V}}(t)|\leq C\delta^2(t+1)^{-2}.
  \end{equation}
  This, together with \eqref{eq:approx_barV}, proves the lemma; therefore, let us prove~\eqref{eq:antiderivative_bound}. We first consider the case of $y<-c(s+1)$. Using~\eqref{eq:L2F}, we obtain~\eqref{eq:antiderivative_bound} as follows:
  \begin{align}
    |L_2 F(y,s)|
    & \leq C\delta^2(s+1)^{-3}\int_{-\infty}^{y}e^{-\frac{(z+c(s+1))^2}{C(s+1)}}\, dz+C\delta^2 \Theta_5(y,s;-c,\nu^*) \\
    & \leq C\delta^2(s+1)^{-3}e^{-\frac{(y+c(s+1))^2}{2C(s+1)}}\int_{-\infty}^{y}e^{-\frac{(z+c(s+1))^2}{2C(s+1)}}\, dz+C\delta^2 \Theta_5(y,s;-c,\nu^*) \\
    & \leq C\delta^2 \Theta_5(y,s;-c,\nu^*).
  \end{align}
  The case of $y\geq -c(s+1)$ can be handled similarly using~\eqref{eq:L2F_negative}.
\end{proof}

\subsection{Pointwise estimates of certain convolutions}
\label{AppendixB}
In this section, we consider pointwise estimates of certain integrals that shall appear in the proof of Theorem~\ref{thm:main}; these, in fact, constitute a core part of the proof.

For a function $f=f(x,t)$ defined for $x\in \mathbb{R}_*$ and $t>0$, let
\begin{equation}
  \mathcal{I}_i[f](x,t)\coloneqq \int_{0}^{t}\int_{-\infty}^{\infty}\partial_x \left\{ (t-s)^{-1/2}e^{-\frac{(x-y-\lambda_i(t-s))^2}{2\nu(t-s)}} \right\} f(y,s)\, dyds.
\end{equation}
Note that in what follows, as in other places, the symbols $C$ and $\nu^*$ denote generic large constants. We also remind the reader that $\delta$ is the quantity defined in~\eqref{ass:smallness}; in this section, we assume without further mention that $\delta \leq 1$.

\begin{lem}
  \label{lem:xii_thetaj}
  We have
  \begin{equation}
    |\mathcal{I}_i[\xi_i \theta_{i'}](x,t)|\leq C\delta^3 \Psi_i(x,t),
  \end{equation}
  where $i'=3-i$.
\end{lem}

\begin{proof}
  We only consider the case of $i=1$ since the other case is similar. By Lemma~\ref{lem:xi} (see also Lemma~\ref{lem:product}), we have
  \begin{equation}
    |(\xi_1 \theta_2)(x,t)+(4c)^{-1}\theta_{2}^{3}(x,t)|\leq C\delta^3 \Theta_4(x,t;-c,\nu^*)\leq C\delta^3(t+1)^{-9/8}\psi_{7/4}(x,t;-c).
  \end{equation}
  Similar calculations leading to the bound of $\zeta_1$ in the proof of Lemma~\ref{lem:xi} show that
  \begin{equation}
    |\mathcal{I}_1[\theta_{2}^{3}](x,t)|\leq C\delta^3[\psi_2(x,t;c)+\log(t+2)\Theta_2(x,t;c,\nu^*)+\Theta_3(x,t;-c,\nu^*)].
  \end{equation}
  We note that the logarithmic term $\log(t+2)\Theta_2(x,t;c,\nu^*)$ comes from the bound of the term corresponding to $I_1(x,t)$ in the proof of Lemma~\ref{lem:xi}. This and Lemma~\ref{lem:LZ98_lemma3.6_modified} (with $\alpha=0$ and $\beta=9/4$) give us
  \begin{equation}
    |\mathcal{I}_1[\xi_1 \theta_2](x,t)|\leq C\delta^3 \Psi_1(x,t).
  \end{equation}
\end{proof}

\begin{lem}
  \label{lem:thetaj_block}
  Suppose that $v_i$ defined by~\eqref{def:vi} satisfies
  \begin{equation}
    \label{eq:vi_ansatz}
    |v_i(x,s)|\leq P(t)\Psi_i(x,s) \quad (0\leq s\leq t)
  \end{equation}
  for some function $P=P(t)\geq 0$.\footnote{Cf.~Eq.~\eqref{def:P} in Section~\ref{SectionIII:PointwiseEstimates}.}~Then for $f\in \{ \theta_{i'}\xi_{i'},\theta_{i'}v_{i'},\xi_{i'}^{2},\xi_{i'}v_{i'} \}$, where $i'=3-i$, we have
  \begin{equation}
    |\mathcal{I}_i[f](x,t)|\leq C(\delta+P(t))^2 \Psi_i(x,t).
  \end{equation}
  Moreover, for $g\in \{ \partial_x(\theta_{i'}^{2}),\partial_x(\theta_{i'}\xi_{i'}),\partial_x(\theta_{i'}v_{i'}) \}$, we have
  \begin{equation}
    |\mathcal{I}_i[g](x,t)|\leq C\delta^2 [\bar{\psi}(x,t;c)+\bar{\psi}(x,t;-c)]\leq C\delta^2 \Psi_j(x,t) \quad (j=1,2).
  \end{equation}
\end{lem}

\begin{proof}
  We assume that $t\geq 4$ since otherwise the lemma is easy to prove, and we only treat $\mathcal{I}_i[\xi_{i'}v_{i'}](x,t)$ since this term is technically the most demanding.\footnote{However, this term is not the most severe one in terms of pointwise estimates; in fact, we can show that it decays faster than $\Psi_i(x,t)$. In terms of pointwise estimates, one of the most important terms is $\mathcal{I}_i[\theta_{i'}\xi_{i'}](x,t)$; its analysis, however, is simpler than the one we describe below.}~Also, we restrict our attention to the case of $i=1$ since the other case can be treated similarly.

  Let $L_2=\partial_t-c\partial_x-(\nu/2)\partial_{x}^{2}$. By Lemma~\ref{lem:3.4_pre}, we have
  \begin{align}
    \mathcal{I}_1[\xi_2 v_2](x,t)
    & =(2c)^{-1}\sqrt{2\pi \nu}(\xi_2 v_2)(x,t)+I_1(x,t)+I_{21}(x,t)+I_{22}(x,t)+I_b(x,t),
  \end{align}
  where
  \begin{equation}
    I_1(x,t)=\int_{0}^{t^{1/2}}\int_{-\infty}^{\infty}\partial_x \left\{ (t-s)^{-1/2}e^{-\frac{(x-y-c(t-s))^2}{2\nu(t-s)}} \right\} (\xi_2 v_2)(y,s)\, dyds,
  \end{equation}
  \begin{equation}
    I_{21}(x,t)=-(2c)^{-1}\int_{-\infty}^{\infty}(t-t^{1/2})^{-1/2}e^{-\frac{(x-y-c(t-\sqrt{t}))^2}{2\nu(t-\sqrt{t})}}(\xi_2 v_2)(y,t^{1/2})\, dy,
  \end{equation}
  \begin{equation}
    I_{22}(x,t)=-(2c)^{-1}\int_{t^{1/2}}^{t}\int_{-\infty}^{\infty}(t-s)^{-1/2}e^{-\frac{(x-y-c(t-s))^2}{2\nu(t-s)}}L_2(\xi_2 v_2)(y,s)\, dyds,
  \end{equation}
  and
  \begin{align}
    I_b(x,t)
    & =2^{-1}\int_{t^{1/2}}^{t}(t-s)^{-1/2}e^{-\frac{(x-c(t-s))^2}{2\nu(t-s)}}\llbracket \xi_2 v_2 \rrbracket(s)\, ds \\
    & \quad -(\nu/2)(2c)^{-1}\int_{t^{1/2}}^{t}(t-s)^{-1/2}e^{-\frac{(x-c(t-s))^2}{2\nu(t-s)}}\llbracket \partial_x(\xi_2 v_2) \rrbracket(s)\, ds \\
    & \quad -(\nu/2)(2c)^{-1}\int_{t^{1/2}}^{t}\partial_x \left\{ (t-s)^{-1/2}e^{-\frac{(x-c(t-s))^2}{2\nu(t-s)}} \right\} \llbracket \xi_2 v_2 \rrbracket(s)\, ds.
  \end{align}
  By~\eqref{def:thetai},~\eqref{def:xii},~\eqref{eq:pde_ui}, and~\eqref{def:N*}, we obtain
  \begin{equation}
    L_2 \xi_2=-\partial_x(\theta_2 \xi_2+\theta_{1}^{2}/2)
  \end{equation}
  and
  \begin{align}
    L_2 v_2
    & =L_2(u_2-\theta_2-\xi_2-\gamma_1 \partial_x \theta_1) \\
    & =\left( \frac{\nu}{2}\partial_{x}^{2}u_1+\partial_x N_2 \right)+\partial_x(\theta_{2}^{2}/2)+\partial_x(\theta_2 \xi_2+\theta_{1}^{2}/2)-\frac{\nu}{4c}\partial_{x}^{2}(\theta_{1}^{2}/2+2c\theta_1) \\
    & =\frac{\nu}{2}\partial_{x}^{2}(u_1-\theta_1)+\partial_x(N_2-N_{2}^{*})-\frac{\nu}{4c}\partial_{x}^{2}(\theta_{1}^{2}/2).
  \end{align}
  Hence, we have
  \begin{align}
    \label{eq:L2_xi2_v2}
    \begin{aligned}
      L_2[\xi_2 v_2]
      & =-v_2 \partial_x(\theta_2 \xi_2+\theta_{1}^{2}/2)+\partial_x \left\{ \xi_2 \left[ \frac{\nu}{2}\partial_x(u_1-\theta_1)+(N_2-N_{2}^{*})-\frac{\nu}{4c}\partial_x(\theta_{1}^{2}/2) \right]-\nu v_2 \partial_x \xi_2 \right\} \\
      & \quad -(\partial_x \xi_2)\left[ \frac{\nu}{2}\partial_x(u_1-\theta_1)+(N_2-N_{2}^{*})-\frac{\nu}{4c}\partial_x(\theta_{1}^{2}/2) \right]+\nu v_2 \partial_{x}^{2}\xi_2.
    \end{aligned}
  \end{align}

  Let us first consider $I_1(x,t)$. By Lemma~\ref{lem:xi} and~\eqref{eq:vi_ansatz} (see also~\ref{lem:product}), we have
  \begin{equation}
    |(\xi_2 v_2)(x,t)|\leq C\delta(\delta+P(t))^2[(t+1)^{-11/8}\psi_{7/4}(x,t;c)+(t+1)^{-3/4}\psi_{7/4}(x,t;-c)].
  \end{equation}
  Then, applying Lemmas~\ref{lem:sqrt1} (with $\alpha=0$ and $3/4\leq \beta<5/4$) and~\ref{lem:LZ98_lemma3.5_modified} (with $\alpha=0$ and $\beta=11/4$), we obtain
  \begin{equation}
    |I_1(x,t)|\leq C\delta(\delta+P(t))^2 \Psi_1(x,t).
  \end{equation}
  For $I_{21}(x,t)$, we apply Lemma~\ref{lem:sqrt3} (with $\alpha=0$ and $\beta=3/2$) to obtain
  \begin{equation}
    |I_{21}(x,t)|\leq C\delta(\delta+P(t))^2 \Psi_1(x,t).
  \end{equation}
  For $I_{22}(x,t)$, we have
  \begin{align}
    |L_2[\xi_2 v_2](x,t)-\partial_x F(x,t)|
    & \leq C\delta(\delta+P(t))^2[(t+1)^{-15/8}\psi_{7/4}(x,t;c)+(t+1)^{-7/4}\psi_{7/4}(x,t;-c)]
  \end{align}
  with
  \begin{equation}
    F(x,t)=\xi_2 \left[ \frac{\nu}{2}\partial_x(u_1-\theta_1)+(N_2-N_{2}^{*})-\frac{\nu}{4c}\partial_x(\theta_{1}^{2}/2) \right]-\nu v_2 \partial_x \xi_2,
  \end{equation}
  and $F$ satisfies
  \begin{equation}
    |F(x,t)|\leq C\delta(\delta+P(t))^2[(t+1)^{-11/8}\psi_{7/4}(x,t;c)+(t+1)^{-5/4}\psi_{7/4}(x,t;-c)].
  \end{equation}
  To see this, we use~\eqref{def:Ni},~\eqref{def:N*},~\eqref{eq:N_structure},~\eqref{eq:tau_squared},~\eqref{eq:vi_ansatz},~\eqref{eq:L2_xi2_v2}, Theorem~\ref{thm:PWE_derivative}, and Lemma~\ref{lem:xi} (see also Lemma~\ref{lem:product}); the calculations are tedious but straightforward. Then, using integration by parts, we obtain
  \begin{align}
    & |I_{22}(x,t)-I_{22,b}(x,t)| \\
    & \leq C\delta(\delta+P(t))^2\int_{t^{1/2}}^{t/2}\int_{-\infty}^{\infty}(t-s)^{-1/2}e^{-\frac{(x-y-c(t-s))^2}{2\nu(t-s)}}[(s+1)^{-15/8}\psi_{7/4}(y,s;c)+(s+1)^{-7/4}\psi_{7/4}(y,s;-c)]\, dyds \\
    & \quad +C\delta(\delta+P(t))^2\int_{t^{1/2}}^{t}\int_{-\infty}^{\infty}(t-s)^{-1}e^{-\frac{(x-y-c(t-s))^2}{C(t-s)}}[(s+1)^{-11/8}\psi_{7/4}(y,s;c)+(s+1)^{-5/4}\psi_{7/4}(y,s;-c)]\, dyds \\
    & \leq C\delta(\delta+P(t))^2\int_{t^{1/2}}^{t/2}\int_{-\infty}^{\infty}(t-s)^{-1/2}e^{-\frac{(x-y-c(t-s))^2}{2\nu(t-s)}}(s+1)^{-7/4}\psi_{7/4}(y,s;-c)\, dyds \\
    & \quad +C\delta(\delta+P(t))^2\int_{t^{1/2}}^{t}\int_{-\infty}^{\infty}(t-s)^{-1}e^{-\frac{(x-y-c(t-s))^2}{C(t-s)}}[(s+1)^{-7/8}\psi_{7/4}(y,s;c)+(s+1)^{-5/4}\psi_{7/4}(y,s;-c)]\, dyds,
  \end{align}
  where
  \begin{equation}
    I_{22,b}(x,t)=(2c)^{-1}\int_{t^{1/2}}^{t}(t-s)^{-1/2}e^{-\frac{(x-c(t-s))^2}{2\nu(t-s)}}\llbracket F \rrbracket(s)\, ds.
  \end{equation}
  By applying Lemmas~\ref{lem:sqrt_ij} (with $\alpha=0$),~\ref{lem:LZ98_lemma3.5_modified} (with $\alpha=0$ and $\beta=7/4$), and~\ref{lem:LZ98_lemma3.6_modified} (with $\alpha=0$ and $\beta=5/2$) (see also Lemma~\ref{lem:indicator}), we obtain
  \begin{equation}
    |I_{22}(x,t)-I_{22,b}(x,t)|\leq C\delta(\delta+P(t))^2 \Psi_1(x,t).
  \end{equation}
  Finally, by applying Lemmas~\ref{lem:boundary1} and~\ref{lem:boundary2}, we obtain
  \begin{equation}
    |I_b(x,t)|+|I_{22,b}(x,t)|\leq C\delta(\delta+P(t))^2 \Psi_1(x,t).
  \end{equation}
  This ends the proof.
\end{proof}

\subsection{Proof of Theorem~\ref{thm:main}}
\label{SectionIII:PointwiseEstimates}
Recall that $v_i$ and $\Psi_i$ are defined by~\eqref{def:vi} and~\eqref{def:Psi}. Define $P(t)$ by
\begin{equation}
  \label{def:P}
  P(t)\coloneqq \sum_{i=1}^{2}\sup_{0\leq s\leq t}|v_i(\cdot,s)\Psi_i(\cdot,s)^{-1}|_{\infty}.
\end{equation}
Here and in what follows, we denote by $|\cdot |_{\infty}$ the Lebesgue $L^{\infty}(\mathbb{R}_*)$-norm. It should be noted that we do not know a priori that $P(t)$ is finite. In what follows, as in previous works (e.g.~\cite{DW18,Koike21,LZ97}), we shall tacitly assume that $P(t)$ is already known to be finite. We mention that, for a related system of equations, this problem was handled in~\cite[p.~296]{IK02} by first deriving estimates for suitably weighted versions of $v_i$ (for which the corresponding $P(t)$ is trivially finite) and taking the limit to the original $v_i$ afterwards.

In order to prove Theorem~\ref{thm:main}, it suffices to prove that there exists a positive constant $C$ such that $P(t)\leq C\delta$ for all $t\geq 0$. To show this, we prove instead that
\begin{equation}
  \label{SectionIII:GoalInequality}
  P(t)\leq C\delta +C(\delta+P(t))^2 \quad (t\geq 0).
\end{equation}
Then, by taking $\delta$ sufficiently small, we can conclude that $P(t)\leq C\delta$ for all $t\geq 0$ (see the argument at the end of~\cite[Section~3.3]{Koike21}).

In what follows, to show~\eqref{SectionIII:GoalInequality}, we evaluate each and every term on the right-hand side of~\eqref{eq:vi_IE}. We start with the terms related to the initial data. For $x>0$, let
\begin{align}
  \label{def:Ii}
  \begin{aligned}
    \mathcal{I}_i(x,t)
    & =\int_{-\infty}^{\infty}g_i(x-y,t)
    \begin{pmatrix}
      u_{01} \\
      u_{02}
    \end{pmatrix}
    (y)\, dy \\
    & \quad -\int_{-\infty}^{\infty}g_{i}^{*}(x-y,t)
    \begin{pmatrix}
      \theta_1 \\
      \theta_2
    \end{pmatrix}
    (y,0)\, dy-\gamma_{i'}\int_{-\infty}^{\infty}\partial_x g_{i'}^{*}(x-y,t)
    \begin{pmatrix}
      \theta_1 \\
      \theta_2
    \end{pmatrix}
    (y,0)\, dy \\
    & \quad +\int_{0}^{\infty}g_{R,i}(x+y,t)
    \begin{pmatrix}
      u_{01} \\
      u_{02}
    \end{pmatrix}
    (y)\, dy+\int_{-\infty}^{0}g_{R,i}(x-y,t)
    \begin{pmatrix}
      u_{02} \\
      u_{01}
    \end{pmatrix}
    (y)\, dy+m_V g_{T,i}(x,t)\bm{1}.
  \end{aligned}
\end{align}

\begin{lem}
  \label{SectionIII:Lemma:Bound_of_Ii}
  Under the assumptions of Theorem~\ref{thm:main}, if~\eqref{ass:smallness} holds with $\delta_0>0$ sufficiently small, then there exists a positive constant $C$ such that
  \begin{equation}
    \label{SectionIII:Lemma:Bound_of_Ii:Bound_of_Ii}
    |\mathcal{I}_i(x,t)|\leq C\delta \Psi_i(x,t)
  \end{equation}
  for $(x,t)\in \mathbb{R}_* \times (0,\infty)$.
\end{lem}

\begin{proof}
  Since the case when $t<1$ can be handled easily, we assume that $t\geq 1$ in the following. Also, we only consider the case of $x>0$ since the case of $x<0$ is similar. First, we rewrite~\eqref{def:Ii} as follows:
  \begin{align}
    \mathcal{I}_i(x,t)
    & =\int_{-\infty}^{\infty}g_{i}^{*}(x-y,t)
    \begin{pmatrix}
      u_1-\theta_1 \\
      u_2-\theta_2
    \end{pmatrix}
    (y,0)\, dy+m_V g_{i}^{*}(x,t)\bm{1} \\
    & \quad +\int_{-\infty}^{\infty}(g_i-g_{i}^{*})(x-y,t)
    \begin{pmatrix}
      u_{01} \\
      u_{02}
    \end{pmatrix}
    (y)\, dy-\gamma_{i'}\int_{-\infty}^{\infty}\partial_x g_{i'}^{*}(x-y,t)
    \begin{pmatrix}
      \theta_1 \\
      \theta_2
    \end{pmatrix}
    (y,0)\, dy+m_V \gamma_{i'}\partial_x g_{i'}^{*}(x,t)\bm{1} \\
    & \quad +\int_{0}^{\infty}g_{R,i}(x+y,t)
    \begin{pmatrix}
      u_{01} \\
      u_{02}
    \end{pmatrix}
    (y)\, dy+\int_{-\infty}^{0}g_{R,i}(x-y,t)
    \begin{pmatrix}
      u_{02} \\
      u_{01}
    \end{pmatrix}
    (y)\, dy \\
    & \quad +m_V (g_{T,i}-g_{i}^{*}-\gamma_{i'}\partial_x g_{i'}^{*})(x,t)\bm{1}.
  \end{align}
  Let
  \begin{align}
    \mathcal{I}_{i,1}(x,t) & \coloneqq \int_{-\infty}^{\infty}g_{i}^{*}(x-y,t)
    \begin{pmatrix}
      u_1-\theta_1 \\
      u_2-\theta_2
    \end{pmatrix}
    (y,0)\, dy+m_V g_{i}^{*}(x,t)\bm{1}, \\
    \mathcal{I}_{i,2}(x,t) & \coloneqq \int_{-\infty}^{\infty}(g_i-g_{i}^{*})(x-y,t)
    \begin{pmatrix}
      u_{01} \\
      u_{02}
    \end{pmatrix}
    (y)\, dy-\gamma_{i'}\int_{-\infty}^{\infty}\partial_x g_{i'}^{*}(x-y,t)
    \begin{pmatrix}
      \theta_1 \\
      \theta_2
    \end{pmatrix}
    (y,0)\, dy+m_V \gamma_{i'}\partial_x g_{i'}^{*}(x,t)\bm{1}, \\
    \mathcal{I}_{i,3}(x,t) & \coloneqq \int_{0}^{\infty}g_{R,i}(x+y,t)
    \begin{pmatrix}
      u_{01} \\
      u_{02}
    \end{pmatrix}
    (y)\, dy+\int_{-\infty}^{0}g_{R,i}(x-y,t)
    \begin{pmatrix}
      u_{02} \\
      u_{01}
    \end{pmatrix}
    (y)\, dy, \\
    \mathcal{I}_{i,4}(x,t) & \coloneqq m_V (g_{T,i}-g_{i}^{*}-\gamma_{i'}\partial_x g_{i'}^{*})(x,t)\bm{1}.
  \end{align}

  We first show that
  \begin{equation}
    \label{SectionIII:Lemma:Bound_of_Ii:Proof:Bound_of_Ii1}
    |\mathcal{I}_{i,1}(x,t)|\leq C\delta \Psi_i(x,t).
  \end{equation}
  For this purpose, define $\eta_j$ by
  \begin{equation}
    \eta_j(x)\coloneqq \int_{-\infty}^{x}(u_j-\theta_j)(y,0)\, dy+m_V H(x),
  \end{equation}
  where $H(x)$ is the Heaviside function. Let $\eta=(\eta_1 \, \eta_2)^{T}$. Then since $\partial_x H(x)=\delta(x)$, we have
  \begin{equation}
    \mathcal{I}_{i,1}(x,t)=\int_{-\infty}^{\infty}g_{i}^{*}(x-y,t)\partial_x \eta(y)\, dy.
  \end{equation}
  Note that by~\eqref{eq:mass}--\eqref{def:thetai_init}, we have\footnote{We remind the reader that $m_V=l_i(0\, V_0)^{T}$. Also, since $V_0=u_0(0_{\pm})$ by one of the compatibility conditions~\eqref{eq:compatibility}, we have $|m_V|\leq C\delta$.}
  \begin{equation}
    \eta_j(x)=
    \begin{dcases}
      -\int_{x}^{\infty}(u_j-\theta_j)(y,0)\, dy & (x>0), \\
      \int_{-\infty}^{x}(u_j-\theta_j)(y,0)\, dy & (x<0).
    \end{dcases}
  \end{equation}
  Hence, by~\eqref{def:u0ipm} and~\eqref{ass:smallness}, we have
  \begin{equation}
    \label{SectionIII:Lemma:Bound_of_Ii:Proof:Bound_of_etaj}
    |\eta_j(x)|\leq C\delta(|x|+1)^{-5/4}.
  \end{equation}
  Let us first show~\eqref{SectionIII:Lemma:Bound_of_Ii:Proof:Bound_of_Ii1} in the case of (i) $|x-\lambda_i t|\leq (t+1)^{1/2}$. By integration by parts, we have
  \begin{equation}
    |\mathcal{I}_{i,1}(x,t)|=\left| \int_{-\infty}^{\infty}\partial_x g_{i}^{*}(x-y,t)\eta(y)\, dy \right|\leq C(t+1)^{-1}\int_{-\infty}^{\infty}|\eta_i(x)|\, dx\leq C\delta (t+1)^{-1}\leq C\delta \Psi_i(x,t).
  \end{equation}
  We next consider the case of (ii) $(t+1)^{1/2}<|x-\lambda_i t|<t+1$. Suppose that $x-\lambda_i t>0$; the case when $x-\lambda_i t<0$ can be treated in a similar manner. Again, by integration by parts, we have
  \begin{align}
    |\mathcal{I}_{i,1}(x,t)|
    & \leq C(t+1)^{-1}\int_{-\infty}^{(x-\lambda_i t)/2}e^{-\frac{(x-\lambda_i t)^2}{Ct}}|\eta_i(y)|\, dy+C\delta (t+1)^{-1}\int_{(x-\lambda_i t)/2}^{\infty}e^{-\frac{(x-y-\lambda_i t)^2}{Ct}}(y+1)^{-5/4}dy \\
    & \leq C\delta(t+1)^{-1}e^{-\frac{(x-\lambda_i t)^2}{Ct}}+C\delta (|x-\lambda_i t|+1)^{-5/4}(t+1)^{-1/2} \\
    & \leq C\delta(t+1)^{-1}e^{-\frac{(x-\lambda_i t)^2}{Ct}}+C\delta (|x-\lambda_i t|+1)^{-7/4}\leq C\delta \Psi_i(x,t).
  \end{align}
  For the last inequality, we used~\eqref{eq:Theta_psi}. We finally consider the case of (iii) $|x-\lambda_i t|\geq t+1$. Again, let us only consider the case when $x-\lambda_i t>0$. By~\eqref{ass:smallness}, we have
  \begin{align}
    |\mathcal{I}_{i,1}(x,t)|
    & \leq C(t+1)^{-1/2}\int_{-\infty}^{(x-\lambda_i t)/2}e^{-\frac{(x-\lambda_i t)^2}{Ct}}|(u_i-\theta_i)(y,0)|\, dy \\
    & \quad +C\delta (t+1)^{-1/2}\int_{(x-\lambda_i t)/2}^{\infty}e^{-\frac{(x-y-\lambda_i t)^2}{2\nu t}}(y+1)^{-7/4}dy+C\delta (t+1)^{-1/2}e^{-\frac{(x-\lambda_i t)^2}{2\nu t}} \\
    & \leq C\delta e^{-\frac{t}{C}}e^{-\frac{(x-\lambda_i t)^2}{Ct}}+C\delta (|x-\lambda_i t|+1)^{-7/4}\leq C\delta \Psi_i(x,t).
  \end{align}
  Thus~\eqref{SectionIII:Lemma:Bound_of_Ii:Proof:Bound_of_Ii1} is proved.

  We next show that
  \begin{equation}
    \label{SectionIII:Lemma:Bound_of_Ii:Proof:Bound_of_Ii2}
    |\mathcal{I}_{i,2}(x,t)|\leq C\delta \Psi_i(x,t).
  \end{equation}
  By~\eqref{eq:PWE_gi_refined}, we have
  \begin{align}
    \mathcal{I}_{i,2}(x,t)
    & =\int_{-\infty}^{\infty}(g_i-g_{i}^{*}-\gamma_{i'}\partial_x g_{i'}^{*})(x-y,t)
    \begin{pmatrix}
      u_{01} \\
      u_{02}
    \end{pmatrix}
    (y)\, dy \\
    & \quad +\gamma_{i'}\left\{ \int_{-\infty}^{\infty}\partial_x g_{i'}^{*}(x-y,t)
    \begin{pmatrix}
      u_1-\theta_1 \\
      u_2-\theta_2
    \end{pmatrix}
    (y,0)\, dy+m_V \partial_x g_{i'}^{*}(x,t)\bm{1} \right\} \\
    & =\gamma_{i'}\partial_x \left\{ \int_{-\infty}^{\infty}g_{i'}^{*}(x-y,t)
    \begin{pmatrix}
      u_1-\theta_1 \\
      u_2-\theta_2
    \end{pmatrix}
    (y,0)\, dy+m_V g_{i'}^{*}(x,t)\bm{1} \right\} \\
    & \quad +O(1)(t+1)^{-1}\int_{-\infty}^{\infty}e^{-\frac{(x-y-\lambda_i t)^2}{Ct}}\left|
    \begin{pmatrix}
      u_{01} \\
      u_{02}
    \end{pmatrix}
    \right| (y)\, dy \\
    & \quad +O(1)(t+1)^{-3/2}\int_{-\infty}^{\infty}e^{-\frac{(x-y-\lambda_{i'}t)^2}{Ct}}\left|
    \begin{pmatrix}
      u_{01} \\
      u_{02}
    \end{pmatrix}
    \right| (y)\, dy \\
    & \quad +O(1)e^{-\frac{c^2}{\nu}t}\left|
    \begin{pmatrix}
      u_{01} \\
      u_{02}
    \end{pmatrix}
    \right| (x) \\
    & \eqqcolon \mathcal{I}_{i,21}(x,t)+\mathcal{I}_{i,22}(x,t)+\mathcal{I}_{i,23}(x,t)+\mathcal{I}_{i,24}(x,t).
  \end{align}
  Here, $O(1)f(x,t)$ is a function whose absolute value is bounded by $C|f(x,t)|$. First, by~\eqref{ass:smallness}, we have
  \begin{equation}
    |\mathcal{I}_{i,24}(x,t)|\leq C\delta e^{-\frac{c^2}{\nu}t}(|x|+1)^{-7/4}\leq C\delta \Psi_i(x,t).
  \end{equation}
  Next, since $\mathcal{I}_{i,21}(x,t)=\gamma_{i'}\partial_x \mathcal{I}_{i',1}(x,t)$, we can treat $\mathcal{I}_{i,21}(x,t)$ similarly to $\mathcal{I}_{i,1}(x,t)$ and obtain
  \begin{equation}
    |\mathcal{I}_{i,21}(x,t)|\leq C\delta (t+1)^{-1/2}\Psi_{i'}(x,t)\leq C\delta \Psi_i(x,t).
  \end{equation}
  For $\mathcal{I}_{i,22}(x,t)$, we first consider the case of (i) $|x-\lambda_i t|\leq (t+1)^{1/2}$. In this case, by~\eqref{ass:smallness}, we have
  \begin{equation}
    |\mathcal{I}_{i,22}(x,t)|\leq C\delta (t+1)^{-1} \leq C\delta \Psi_i(x,t).
  \end{equation}
  We next consider the case of (ii) $|x-\lambda_i t|>(t+1)^{1/2}$. Suppose that $x-\lambda_i t>0$; the case when $x-\lambda_i t<0$ can be treated in a similar manner. Then, by~\eqref{ass:smallness}, we have
  \begin{align}
    |\mathcal{I}_{i,22}(x,t)|
    & \leq C(t+1)^{-1}\int_{-\infty}^{(x-\lambda_i t)/2}e^{-\frac{(x-\lambda_i t)^2}{Ct}}\left|
    \begin{pmatrix}
      u_{01} \\
      u_{02}
    \end{pmatrix}
    \right| (y)\, dy+C\delta (t+1)^{-1}\int_{(x-\lambda_i t)/2}^{\infty}e^{-\frac{(x-y-\lambda_i t)^2}{Ct}}(y+1)^{-7/4}\, dy \\
    & \leq C\delta (t+1)^{-1}e^{-\frac{(x-\lambda_i t)^2}{Ct}}+C\delta (|x-\lambda_i t|+1)^{-7/4}(t+1)^{-1/2}\leq C\delta \Psi_i(x,t).
  \end{align}
  For $\mathcal{I}_{i,23}(x,t)$, we first consider the case of (i)' $|x-\lambda_{i'}t|\leq (t+1)^{1/2}$. In this case, by~\eqref{ass:smallness}, we have
  \begin{equation}
    |\mathcal{I}_{i,23}(x,t)|\leq C\delta (t+1)^{-3/2}\leq C\delta \bar{\psi}(x,t;\lambda_{i'})\leq C\delta \Psi_i(x,t).
  \end{equation}
  We next consider the case of (ii)' $|x-\lambda_{i'}t|>(t+1)^{1/2}$. Suppose that $x-\lambda_{i'}t>0$; the case when $x-\lambda_{i'}t<0$ can be treated similarly. Then, by~\eqref{ass:smallness}, we have
  \begin{align}
    |\mathcal{I}_{i,23}(x,t)|
    & \leq C(t+1)^{-3/2}\int_{-\infty}^{(x-\lambda_{i'}t)/2}e^{-\frac{(x-\lambda_{i'}t)^2}{Ct}}\left|
    \begin{pmatrix}
      u_{01} \\
      u_{02}
    \end{pmatrix}
    \right| (y)\, dy \\
    & \quad +C\delta (t+1)^{-3/2}\int_{(x-\lambda_{i'}t)/2}^{\infty}e^{-\frac{(x-y-\lambda_{i'}t)^2}{Ct}}(y+1)^{-7/4}\, dy \\
    & \leq C\delta (t+1)^{-3/2}e^{-\frac{(x-\lambda_{i'}t)^2}{Ct}}+C\delta (|x-\lambda_{i'}t|+1)^{-7/4}(t+1)^{-1}\leq C\delta \Psi_i(x,t).
  \end{align}
  Thus~\eqref{SectionIII:Lemma:Bound_of_Ii:Proof:Bound_of_Ii2} is proved.

  We next show that
  \begin{equation}
    \label{SectionIII:Lemma:Bound_of_Ii:Proof:Bound_of_Ii3}
    |\mathcal{I}_{i,3}(x,t)|\leq C\delta \Psi_i(x,t).
  \end{equation}
  By~\eqref{eq:PWE_gi}, \eqref{eq:gR_relation}, and~\eqref{eq:PWE_gTi}, we have\footnote{Note that the delta functions in~\eqref{eq:PWE_gi} can be ignored since the spatial arguments of $g_{R,i}$ are positive due to $x>0$.}
  \begin{align}
    \mathcal{I}_{i,3}(x,t)
    & =\frac{1}{2}\int_{0}^{\infty}\partial_x g_{i}^{*}(x+y,t)
    \begin{pmatrix}
      u_{02} \\
      u_{01}
    \end{pmatrix}
    (y)\, dy+\frac{1}{2}\int_{-\infty}^{0}\partial_x g_{i}^{*}(x-y,t)
    \begin{pmatrix}
      u_{01} \\
      u_{02}
    \end{pmatrix}
    (y)\, dy \\
    & \quad +O(1)\sum_{j=1}^{2}(t+1)^{-3/2}\int_{-\infty}^{\infty}e^{-\frac{(x-y-\lambda_j t)^2}{Ct}}\left|
    \begin{pmatrix}
      u_{01} \\
      u_{02}
    \end{pmatrix}
    \right| (y)\, dy \\
    & \quad +O(1)\int_{-\infty}^{\infty}e^{-\frac{|x|+|y|+t}{C}}\left|
    \begin{pmatrix}
      u_{01} \\
      u_{02}
    \end{pmatrix}
    \right| (y)\, dy.
  \end{align}
  The first two terms can be treated similarly to $\mathcal{I}_{i,22}(x,t)$ and the third term similarly to $\mathcal{I}_{i,23}(x,t)$. The last term is bounded by $C\delta e^{-(|x|+t)/C}\leq C\delta \Psi_i(x,t)$. Thus~\eqref{SectionIII:Lemma:Bound_of_Ii:Proof:Bound_of_Ii3} is proved.

  Finally, we need to show that
  \begin{equation}
    \label{SectionIII:Lemma:Bound_of_Ii:Proof:Bound_of_Ii4}
    |\mathcal{I}_{i,4}(x,t)|\leq C\delta \Psi_i(x,t).
  \end{equation}
  This can be easily proved using~\eqref{eq:PWE_gi_refined},~\eqref{eq:gT_der},~and~\eqref{eq:PWE_gTi}:
  \begin{equation}
    |\mathcal{I}_{i,4}(x,t)|\leq C\delta (t+1)^{-1}e^{-\frac{(x-\lambda_i t)^2}{Ct}}+C\delta (t+1)^{-3/2}e^{-\frac{(x-\lambda_{i'}t)^2}{Ct}}+C\delta e^{-\frac{|x|+t}{C}}\leq C\delta \Psi_i(x,t).
  \end{equation}
  This ends the proof of the lemma.
\end{proof}

Next, we study the nonlinear terms in the integral equation~\eqref{eq:vi_IE}: for $x>0$, let
\begin{align}
  & \mathcal{N}_i(x,t) \\
  & =v_i(x,t)-\mathcal{I}_i(x,t) \\
  & =\int_{0}^{t}\int_{-\infty}^{\infty}g_{i}^{*}(x-y,t-s)
  \begin{pmatrix}
    N_1-N_{1}^{*} \\
    N_2-N_{2}^{*}
  \end{pmatrix}_x
  (y,s)\, dyds \\
  & \quad +\int_{0}^{t}\int_{-\infty}^{\infty}(g_i-g_{i}^{*})(x-y,t-s)
  \begin{pmatrix}
    N_1 \\
    N_2
  \end{pmatrix}_x
  (y,s)\, dyds+\gamma_{i'}\int_{0}^{t}\int_{-\infty}^{\infty}\partial_x g_{i'}^{*}(x-y,t-s)
  \begin{pmatrix}
    \theta_{1}^{2}/2 \\
    \theta_{2}^{2}/2
  \end{pmatrix}_x
  (y,s)\, dyds \\
  & \quad +\int_{0}^{t}\int_{0}^{\infty}g_{R,i}(x+y,t-s)
  \begin{pmatrix}
    N_1 \\
    N_2
  \end{pmatrix}_x
  (y,s)\, dyds+\int_{0}^{t}\int_{-\infty}^{0}g_{R,i}(x-y,t-s)
  \begin{pmatrix}
    N_1 \\
    N_2
  \end{pmatrix}_x
  (y,s)\, dyds \\
  & \quad +\int_{0}^{t}g_{T,i}(x,t-s)
  \begin{pmatrix}
    \llbracket N_1 \rrbracket \\
    \llbracket N_2 \rrbracket
  \end{pmatrix}
  (s)\, ds.
\end{align}
We then prove the following lemma.

\begin{lem}
  \label{SectionIII:Lemma:Bound_of_Ni}
  Under the assumptions of Theorem~\ref{thm:main}, if~\eqref{ass:smallness} holds with $\delta_0>0$ sufficiently small, then there exists a positive constant $C$ such that
  \begin{equation}
    \label{SectionIII:Lemma:Bound_of_Ii:Bound_of_Ni}
    |\mathcal{N}_i(x,t)|\leq C(\delta+P(t))^2 \Psi_i(x,t)
  \end{equation}
  for $(x,t)\in \mathbb{R}_* \times (0,\infty)$.
\end{lem}

\begin{proof}
  Since the case when $t<4$ can be handled easily, we assume that $t\geq 4$ in the following. Also, we only consider the case of $x>0$ since the case of $x<0$ is similar. Let
  \begin{equation}
    N_{i,a}=-\frac{p''(1)^2}{8c^2}(v-1)^2, \quad N_{i,b}=-\frac{\nu p''(1)}{4c^2}(v-1)u_x, \quad N_{i,c}=N_i-N_{i,a}-N_{i,b}.
  \end{equation}
  By~\eqref{eq:N_structure}, we can see that $N_{i,a}$ is the lowest order term among the three. By definition, we have
  \begin{equation}
    N_i=N_{i,a}+N_{i,b}+N_{i,c}.
  \end{equation}
  Corresponding to this decomposition, let $\mathcal{N}_{i}(x,t)=\mathcal{N}_{i,a}(x,t)+\mathcal{N}_{i,b}(x,t)+\mathcal{N}_{i,c}(x,t)$, where
  \begin{align}
    & \mathcal{N}_{i,a}(x,t) \\
    & =\int_{0}^{t}\int_{-\infty}^{\infty}g_{i}^{*}(x-y,t-s)
    \begin{pmatrix}
      N_{1,a}-N_{1}^{*} \\
      N_{2,a}-N_{2}^{*}
    \end{pmatrix}_x
    (y,s)\, dyds \\
    & \quad +\int_{0}^{t}\int_{-\infty}^{\infty}(g_i-g_{i}^{*})(x-y,t-s)
    \begin{pmatrix}
      N_{1,a} \\
      N_{2,a}
    \end{pmatrix}_x
    (y,s)\, dyds+\gamma_{i'}\int_{0}^{t}\int_{-\infty}^{\infty}\partial_x g_{i'}^{*}(x-y,t-s)
    \begin{pmatrix}
      \theta_{1}^{2}/2 \\
      \theta_{2}^{2}/2
    \end{pmatrix}_x
    (y,s)\, dyds \\
    & \quad +\int_{0}^{t}\int_{0}^{\infty}g_{R,i}(x+y,t-s)
    \begin{pmatrix}
      N_{1,a} \\
      N_{2,a}
    \end{pmatrix}_x
    (y,s)\, dyds+\int_{0}^{t}\int_{-\infty}^{0}g_{R,i}(x-y,t-s)
    \begin{pmatrix}
      N_{1,a} \\
      N_{2,a}
    \end{pmatrix}_x
    (y,s)\, dyds \\
    & \quad +\int_{0}^{t}g_{T,i}(x,t-s)
    \begin{pmatrix}
      \llbracket N_{1,a} \rrbracket \\
      \llbracket N_{2,a} \rrbracket
    \end{pmatrix}
    (s)\, ds,
  \end{align}
  \begin{align}
    \mathcal{N}_{i,b}(x,t)
    & =\int_{0}^{t}\int_{-\infty}^{\infty}g_{i}^{*}(x-y,t-s)
    \begin{pmatrix}
      N_{1,b} \\
      N_{2,b}
    \end{pmatrix}_x
    (y,s)\, dyds \\
    & \quad +\int_{0}^{t}\int_{-\infty}^{\infty}(g_i-g_{i}^{*})(x-y,t-s)
    \begin{pmatrix}
      N_{1,b} \\
      N_{2,b}
    \end{pmatrix}_x
    (y,s)\, dyds \\
    & \quad +\int_{0}^{t}\int_{0}^{\infty}g_{R,i}(x+y,t-s)
    \begin{pmatrix}
      N_{1,b} \\
      N_{2,b}
    \end{pmatrix}_x
    (y,s)\, dyds+\int_{0}^{t}\int_{-\infty}^{0}g_{R,i}(x-y,t-s)
    \begin{pmatrix}
      N_{1,b} \\
      N_{2,b}
    \end{pmatrix}_x
    (y,s)\, dyds \\
    & \quad +\int_{0}^{t}g_{T,i}(x,t-s)
    \begin{pmatrix}
      \llbracket N_{1,b} \rrbracket \\
      \llbracket N_{2,b} \rrbracket
    \end{pmatrix}
    (s)\, ds,
  \end{align}
  and
  \begin{align}
    \mathcal{N}_{i,c}(x,t)
    & =\int_{0}^{t}\int_{-\infty}^{\infty}g_{i}^{*}(x-y,t-s)
    \begin{pmatrix}
      N_{1,c} \\
      N_{2,c}
    \end{pmatrix}_x
    (y,s)\, dyds \\
    & \quad +\int_{0}^{t}\int_{-\infty}^{\infty}(g_i-g_{i}^{*})(x-y,t-s)
    \begin{pmatrix}
      N_{1,c} \\
      N_{2,c}
    \end{pmatrix}_x
    (y,s)\, dyds \\
    & \quad +\int_{0}^{t}\int_{0}^{\infty}g_{R,i}(x+y,t-s)
    \begin{pmatrix}
      N_{1,c} \\
      N_{2,c}
    \end{pmatrix}_x
    (y,s)\, dyds+\int_{0}^{t}\int_{-\infty}^{0}g_{R,i}(x-y,t-s)
    \begin{pmatrix}
      N_{1,c} \\
      N_{2,c}
    \end{pmatrix}_x
    (y,s)\, dyds \\
    & \quad +\int_{0}^{t}g_{T,i}(x,t-s)
    \begin{pmatrix}
      \llbracket N_{1,c} \rrbracket \\
      \llbracket N_{2,c} \rrbracket
    \end{pmatrix}
    (s)\, ds.
  \end{align}
  In what follows, we only consider the case of $i=1$ since the other case is similar.

  We then show in the following that
  \begin{equation}
    \label{SectionIII:Lemma:Bound_of_Ii:Bound_of_Nia}
    |\mathcal{N}_{1,a}(x,t)|\leq C(\delta+P(t))^2 \Psi_1(x,t).
  \end{equation}
  The other terms $\mathcal{N}_{1,b}(x,t)$ and $\mathcal{N}_{1,c}(x,t)$ are basically easier to treat, and we can also show that
  \begin{equation}
    \label{SectionIII:Lemma:Bound_of_Ii:Bound_of_Nibc}
    |\mathcal{N}_{1,b}(x,t)|+|\mathcal{N}_{1,c}(x,t)|\leq C(\delta+P(t))^2 \Psi_1(x,t).
  \end{equation}
  In the following, we only prove~\eqref{SectionIII:Lemma:Bound_of_Ii:Bound_of_Nia} and omit the proof of~\eqref{SectionIII:Lemma:Bound_of_Ii:Bound_of_Nibc} for brevity. To begin with, let
  \begin{equation}
    R_{i,a}=N_{i,a}-N_{i}^{*}.
  \end{equation}
  Note that since $N_{i}^{*}$ is continuous in $x$, we have $\llbracket R_{i,a} \rrbracket(t)=\llbracket N_{i,a} \rrbracket(t)$. Then, using~\eqref{eq:gR_relation} and noting that $N_{1,a}=N_{2,a}$, integration by parts gives
  \begin{align}
    \label{eq:Nia}
    & \mathcal{N}_{1,a}(x,t) \\
    & =\int_{0}^{t}\int_{-\infty}^{\infty}\partial_x g_{1}^{*}(x-y,t-s)
    \begin{pmatrix}
      R_{1,a} \\
      R_{2,a}
    \end{pmatrix}
    (y,s)\, dyds \\
    & \quad +\int_{0}^{t}\int_{-\infty}^{\infty}\partial_x (g_1-g_{1}^{*})(x-y,t-s)
    \begin{pmatrix}
      N_{1,a} \\
      N_{2,a}
    \end{pmatrix}
    (y,s)\, dyds+\gamma_2\int_{0}^{t}\int_{-\infty}^{\infty}\partial_{x}^{2}g_{2}^{*}(x-y,t-s)
    \begin{pmatrix}
      \theta_{1}^{2}/2 \\
      \theta_{2}^{2}/2
    \end{pmatrix}
    (y,s)\, dyds \\
    & \quad -\int_{0}^{t}\int_{0}^{\infty}\partial_x g_{R,1}(x+y,t-s)
    \begin{pmatrix}
      N_{1,a} \\
      N_{2,a}
    \end{pmatrix}
    (y,s)\, dyds+\int_{0}^{t}\int_{-\infty}^{0}\partial_x g_{R,1}(x-y,t-s)
    \begin{pmatrix}
      N_{1,a} \\
      N_{2,a}
    \end{pmatrix}
    (y,s)\, dyds \\
    & \eqqcolon \mathcal{J}_1(x,t)+\mathcal{J}_2(x,t)+\mathcal{J}_3(x,t).
  \end{align}
  Here, the terms $\mathcal{J}_1(x,t)$, $\mathcal{J}_2(x,t)$, and $\mathcal{J}_3(x,t)$ correspond to the first, the second, and the third row, respectively, of the right-hand side of the first equality.

  We first show that
  \begin{equation}
    \label{eq:J1_bound}
    |\mathcal{J}_1(x,t)|\leq C(\delta+P(t))^2 \Psi_1(x,t).
  \end{equation}
  For this purpose, it is important to note that
  \begin{align}
    \label{eq:R1a_bound}
    \begin{aligned}
      & \left| R_{1,a}(x,t)-\left[ \xi_1 \theta_2+\frac{\gamma_2}{2}\partial_x(\theta_{2}^{2})-\theta_2 \xi_2-\theta_2 v_2-\xi_{2}^{2}/2-\xi_2 v_2 \right](x,t) \right| \\
      & \leq C(\delta+P(t))^2[(t+1)^{-1/2}\psi_{7/4}(x,t;c)+(t+1)^{-7/8}\psi_{7/4}(x,t;-c)].
    \end{aligned}
  \end{align}
  To show this, we need three ingredients: (a) structure of $R_{1,a}$; (b) pointwise estimates of $\theta_i$, $\xi_i$, and $v_i$; and (c) pointwise estimates of products of certain functions, for example, $\psi_{7/4}(x,t;\lambda_i)^2$. First, (a) is provided by~\eqref{eq:tau_squared} (note that the left-hand side of~\eqref{eq:tau_squared} is $N_{i,a}$); secondly, (b) is provided by~\eqref{eq:theta_Theta},~Lemma~\ref{lem:xi}, and~\eqref{def:P}, which implies
  \begin{equation}
    |v_i(x,s)|\leq P(t)\Psi_i(x,s) \quad (0\leq s\leq t).
  \end{equation}
  Finally, (c) is provided by Lemma~\ref{lem:product}. The required calculations are straightforward but inevitably long. Now, using~\eqref{eq:R1a_bound}, we quickly obtain~\eqref{eq:J1_bound} by the use of Lemmas~\ref{lem:xii_thetaj},~\ref{lem:thetaj_block},~\ref{lem:LZ98_lemma3.5_modified} (with $\alpha=0$ and $\beta=1$), and~\ref{lem:LZ98_lemma3.6_modified} (with $\alpha=0$ and $\beta=7/4$).

  Next, we show that
  \begin{equation}
    \label{eq:J2_bound}
    |\mathcal{J}_2(x,t)|\leq C(\delta+P(t))^2 \Psi_1(x,t).
  \end{equation}
  Let
  \begin{align}
    & \mathcal{J}_2(x,t) \\
    & =\int_{0}^{t}\int_{-\infty}^{\infty}\partial_x (g_1-g_{1}^{*})(x-y,t-s)
    \begin{pmatrix}
      N_{1,a}+\theta_{2}^{2}/2 \\
      N_{2,a}+\theta_{2}^{2}/2
    \end{pmatrix}
    (y,s)\, dyds \\
    & \quad -\int_{0}^{t}\int_{-\infty}^{\infty}\partial_x (g_1-g_{1}^{*})(x-y,t-s)
    \begin{pmatrix}
      \theta_{2}^{2}/2 \\
      \theta_{2}^{2}/2
    \end{pmatrix}
    (y,s)\, dyds+\gamma_2 \int_{0}^{t}\int_{-\infty}^{\infty}\partial_{x}^{2}g_{2}^{*}(x-y,t-s)
    \begin{pmatrix}
      \theta_{1}^{2}/2 \\
      \theta_{2}^{2}/2
    \end{pmatrix}
    (y,s)\, dyds \\
    & \eqqcolon \mathcal{J}_{21}(x,t)+\mathcal{J}_{22}(x,t).
  \end{align}
  Here, $\mathcal{J}_{21}(x,t)$ and $\mathcal{J}_{22}(x,t)$ correspond to the first and the second row, respectively, of the right-hand side of the first equality.

  For $\mathcal{J}_{21}(x,t)$, we further divide it into two parts:
  \begin{align}
    \mathcal{J}_{21}(x,t)
    & =\int_{0}^{t/2}\int_{-\infty}^{\infty}\partial_x (g_1-g_{1}^{*})(x-y,t-s)
    \begin{pmatrix}
      N_{1,a}+\theta_{2}^{2}/2 \\
      N_{2,a}+\theta_{2}^{2}/2
    \end{pmatrix}
    (y,s)\, dyds \\
    & \quad +\int_{t/2}^{t}\int_{-\infty}^{\infty}\partial_x (g_1-g_{1}^{*})(x-y,t-s)
    \begin{pmatrix}
      N_{1,a}+\theta_{2}^{2}/2 \\
      N_{2,a}+\theta_{2}^{2}/2
    \end{pmatrix}
    (y,s)\, dyds \\
    & \eqqcolon \mathcal{J}_{211}(x,t)+\mathcal{J}_{212}(x,t).
  \end{align}

  Let us first consider $\mathcal{J}_{211}(x,t)$. By~\eqref{eq:PWE_gi_refined}, we have
  \begin{equation}
    \label{eq:gi-gi*_der}
    \left| \partial_x g_1(x,t)-\partial_x g_{1}^{*}(x,t)-\gamma_2 \partial_{x}^{2}g_{2}^{*}(x,t)-e^{-\frac{c^2}{\nu}t}\sum_{j=0}^{1}\delta^{(1-j)}(x)q_{1j}(t) \right|\leq Ct^{-\frac{3}{2}}e^{-\frac{(x-ct)^2}{Ct}}+Ct^{-2}e^{-\frac{(x+ct)^2}{Ct}}.
  \end{equation}
  Hence, to show that $|J_{211}(x,t)|$ is bounded by the right-hand side of~\eqref{eq:J2_bound}, it suffices to show the same thing for the following terms:
  \begin{equation}
    \label{eq:J211_1}
    \int_{0}^{t/2}\int_{-\infty}^{\infty}(t-s)^{-3/2}e^{-\frac{(x-y-c(t-s))^2}{C(t-s)}}\left|
    \begin{pmatrix}
      N_{1,a}+\theta_{2}^{2}/2 \\
      N_{2,a}+\theta_{2}^{2}/2
    \end{pmatrix}
    \right| (y,s) \, dyds,
  \end{equation}
  \begin{equation}
    \label{eq:J211_2}
    \int_{0}^{t/2}\int_{-\infty}^{\infty}(t-s)^{-2}e^{-\frac{(x-y+c(t-s))^2}{C(t-s)}}\left|
    \begin{pmatrix}
      N_{1,a}+\theta_{2}^{2}/2 \\
      N_{2,a}+\theta_{2}^{2}/2
    \end{pmatrix}
    \right| (y,s)\, dyds,
  \end{equation}
  \begin{equation}
    \label{eq:J211_3}
    \left| \int_{0}^{t/2}\int_{-\infty}^{\infty}\partial_{x}^{2}g_{22}^{*}(x-y,t-s)(N_{2,a}+\theta_{2}^{2}/2)(y,s)\, dyds \right|,
  \end{equation}
  and
  \begin{equation}
    \label{eq:J211_4}
    \int_{0}^{t/2}e^{-\frac{c^2}{\nu}(t-s)}\left|
    \begin{pmatrix}
      N_{1,a}+\theta_{2}^{2}/2 \\
      N_{2,a}+\theta_{2}^{2}/2
    \end{pmatrix}
    \right| (x,s)\, ds.
  \end{equation}
  Here, $g_{2}^{*}=(0\, g_{22}^{*})$, and we used the relation
  \begin{equation}
    \label{eq:delta1_vanishing}
    q_{10}(t)
    \begin{pmatrix}
      N_{1,a}+\theta_{2}^{2}/2 \\
      N_{2,a}+\theta_{2}^{2}/2
    \end{pmatrix}
    =l_1 Q_0
    \begin{pmatrix}
      r_1 & r_2
    \end{pmatrix}
    \begin{pmatrix}
      N_{1,a}+\theta_{2}^{2}/2 \\
      N_{2,a}+\theta_{2}^{2}/2
    \end{pmatrix}
    =\frac{1}{2}
    \begin{pmatrix}
      1 & -1
    \end{pmatrix}
    \begin{pmatrix}
      N_{1,a}+\theta_{2}^{2}/2 \\
      N_{2,a}+\theta_{2}^{2}/2
    \end{pmatrix}
    =0,
  \end{equation}
  which follows from~\eqref{SectionIII:Q0Q1} and $N_{1,a}=N_{2,a}$. Now, we first treat~\eqref{eq:J211_1} and~\eqref{eq:J211_2}. By calculations similar to those leading to~\eqref{eq:R1a_bound}, we can show that
  \begin{equation}
    \label{eq:Nja+half_theta22}
    |(N_{j,a}+\theta_{2}^{2}/2)(x,t)|\leq C(\delta+P(t))^2 [(t+1)^{-1/8}\psi_{7/4}(x,t;c)+(t+1)^{-3/8}\psi_{7/4}(x,t;-c)].
  \end{equation}
  Using this inequality together with Lemmas~\ref{lem:LZ98_lemma3.5_modified} (with $\alpha=1$ and $\beta=1/4$; $\alpha=2$ and $\beta=3/4$) and~\ref{lem:LZ98_lemma3.6_modified} (with $\alpha=1$ and $\beta=3/4$; $\alpha=2$ and $\beta=1/4$), we see that~\eqref{eq:J211_1} and~\eqref{eq:J211_2} are bounded by the right-hand side of~\eqref{eq:J2_bound}. Next, to treat~\eqref{eq:J211_3}, we first rewrite it as
  \begin{align}
    & \int_{0}^{t/2}\int_{-\infty}^{\infty}\partial_{x}^{2}g_{22}^{*}(x-y,t-s)(N_{2,a}+\theta_{2}^{2}/2)(y,s)\, dyds \\
    & =\int_{0}^{t/2}\int_{-\infty}^{\infty}\partial_{x}^{2}g_{22}^{*}(x-y,t-s)(N_{2,a}+\theta_{1}^{2}/2+\theta_{2}^{2}/2)(y,s)\, dyds \\
    & \quad -\frac{1}{2}\int_{0}^{t/2}\int_{-\infty}^{\infty}\partial_{x}^{2}g_{22}^{*}(x-y,t-s)\theta_{1}^{2}(y,s)\, dyds.
  \end{align}
  Again, by calculations similar to those leading to~\eqref{eq:R1a_bound}, we can show that
  \begin{equation}
    \label{eq:N2a+half_theta11_half_theta22}
    |(N_{2,a}+\theta_{1}^{2}/2+\theta_{2}^{2}/2)(x,t)|\leq C(\delta+P(t))^2(t+1)^{-3/8}[\psi_{7/4}(x,t;c)+\psi_{7/4}(x,t;-c)].
  \end{equation}
  Using this inequality together with Lemmas~\ref{lem:thetaj_block},~\ref{lem:LZ98_lemma3.5_modified} (with $\alpha=1$ and $\beta=3/4$), and~\ref{lem:LZ98_lemma3.6_modified} (with $\alpha=1$ and $\beta=3/4$), we see that~\eqref{eq:J211_3} is bounded by the right-hand side of~\eqref{eq:J2_bound}. Finally, by Lemma~\ref{lem:delta}, we easily see that~\eqref{eq:J211_4} is bounded by the right-hand side of~\eqref{eq:J2_bound}. Hence, the same bound also holds for $|\mathcal{J}_{211}(x,t)|$.

  For $\mathcal{J}_{212}(x,t)$, we first apply integration by parts:
  \begin{align}
    \label{eq:J212}
    \begin{aligned}
      \mathcal{J}_{212}(x,t)
      & =\int_{t/2}^{t}\int_{-\infty}^{\infty}(g_1-g_{1}^{*})(x-y,t-s)
      \begin{pmatrix}
        N_{1,a}+\theta_{2}^{2}/2 \\
        N_{2,a}+\theta_{2}^{2}/2
      \end{pmatrix}_x
      (y,s)\, dyds \\
      & \quad +\int_{t/2}^{t}(g_1-g_{1}^{*})(x,t-s)
      \begin{pmatrix}
        \llbracket N_{1,a} \rrbracket \\
        \llbracket N_{2,a} \rrbracket
      \end{pmatrix}
      (s)\, ds.
    \end{aligned}
  \end{align}
  The second term on the right-hand side is shown to be bounded by the right-hand side of~\eqref{eq:J2_bound} by~Lemma~\ref{lem:boundary1}; use here~\eqref{eq:PWE_gi} instead of~\eqref{eq:PWE_gi_refined}. Next, by~\eqref{eq:PWE_gi_refined}, we have
  \begin{equation}
    \label{eq:gi-gi*}
    \left| g_1(x,t)-g_{1}^{*}(x,t)-\gamma_2 \partial_x g_{2}^{*}(x,t)-e^{-\frac{c^2}{\nu}t}\delta(x)q_{10}(t) \right|\leq Ct^{-1}e^{-\frac{(x-ct)^2}{Ct}}+C(t+1)^{-1/2}t^{-1}e^{-\frac{(x+ct)^2}{Ct}}.
  \end{equation}
  Hence, to show that $|J_{212}(x,t)|$ is bounded by the right-hand side of~\eqref{eq:J2_bound}, it suffices to show the same thing for the following terms:
  \begin{equation}
    \label{eq:J212_1}
    \int_{t/2}^{t}\int_{-\infty}^{\infty}(t-s)^{-1}e^{-\frac{(x-y-c(t-s))^2}{C(t-s)}}\left|
    \begin{pmatrix}
      N_{1,a}+\theta_{2}^{2}/2 \\
      N_{2,a}+\theta_{2}^{2}/2
    \end{pmatrix}_x
    \right| (y,s)\, dyds,
  \end{equation}
  \begin{equation}
    \label{eq:J212_2}
    \int_{t/2}^{t}\int_{-\infty}^{\infty}(t-s)^{-1}(t-s+1)^{-1/2}e^{-\frac{(x-y+c(t-s))^2}{C(t-s)}}\left|
    \begin{pmatrix}
      N_{1,a}+\theta_{2}^{2}/2 \\
      N_{2,a}+\theta_{2}^{2}/2
    \end{pmatrix}_x
    \right| (y,s)\, dyds,
  \end{equation}
  and
  \begin{equation}
    \label{eq:J212_3}
    \left| \int_{t/2}^{t}\int_{-\infty}^{\infty}\partial_x g_{22}^{*}(x-y,t-s)\partial_x(N_{2,a}+\theta_{2}^{2}/2)(y,s)\, dyds \right|.
  \end{equation}
  Note here the use of~\eqref{eq:delta1_vanishing}. Now, we first treat~\eqref{eq:J212_1} and~\eqref{eq:J212_2}. By calculations similar to those leading to~\eqref{eq:R1a_bound}, we can show that (we use Theorem~\ref{thm:PWE_derivative} here)
  \begin{equation}
    \label{eq:Nja+half_theta22_der}
    |\partial_x(N_{j,a}+\theta_{2}^{2}/2)(x,t)|\leq C(\delta+P(t))^2 [(t+1)^{-5/8}\psi_{7/4}(x,t;c)+(t+1)^{-7/8}\psi_{7/4}(x,t;-c)].
  \end{equation}
  Using this inequality together with Lemmas~\ref{lem:LZ98_lemma3.5_modified} (with $\alpha=0$ and $\beta=5/4$; $\alpha=1$ and $\beta=7/4$) and~\ref{lem:LZ98_lemma3.6_modified} (with $\alpha=0$ and $\beta=7/4$; $\alpha=1$ and $\beta=5/4$), we see that~\eqref{eq:J212_1} and~\eqref{eq:J212_2} are bounded by the right-hand side of~\eqref{eq:J2_bound}. Finally, to treat~\eqref{eq:J212_3}, we first rewrite it as
  \begin{align}
    \label{eq:g22_interaction}
    \begin{aligned}
      & \int_{t/2}^{t}\int_{-\infty}^{\infty}\partial_x g_{22}^{*}(x-y,t-s)\partial_x(N_{2,a}+\theta_{2}^{2}/2)(y,s)\, dyds \\
      & =\int_{t/2}^{t}\int_{-\infty}^{\infty}\partial_x g_{22}^{*}(x-y,t-s)\partial_x (N_{2,a}+\theta_{1}^{2}/2+\theta_{2}^{2}/2+\theta_1 \xi_1+\theta_1 v_1)(y,s)\, dyds \\
      & \quad -\int_{t/2}^{t}\int_{-\infty}^{\infty}\partial_x g_{22}^{*}(x-y,t-s)\partial_x(\theta_{1}^{2}/2+\theta_1 \xi_1+\theta_1 v_1)(y,s)\, dyds.
    \end{aligned}
  \end{align}
  By calculations similar to those leading to~\eqref{eq:R1a_bound}, we can show that (we use Theorem~\ref{thm:PWE_derivative} here)
  \begin{equation}
    |\partial_x (N_{2,a}+\theta_{1}^{2}/2+\theta_{2}^{2}/2+\theta_1 \xi_1+\theta_1 v_1)(x,t)|\leq C(\delta+P(t))^2[(t+1)^{-9/8}\psi_{7/4}(x,t;c)+(t+1)^{-7/8}\psi_{7/4}(x,t;-c)].
  \end{equation}
  From these inequalities together with Lemmas~\ref{lem:thetaj_block},~\ref{lem:LZ98_lemma3.5_modified} (with $\alpha=0$ and $\beta=7/4$), and~\ref{lem:LZ98_lemma3.6_modified} (with $\alpha=0$ and $\beta=9/4$), we see that~\eqref{eq:J212_3} is bounded by the right-hand side of~\eqref{eq:J2_bound}. Hence, the same bound also holds for $|\mathcal{J}_{212}(x,t)|$.

  We next consider $\mathcal{J}_{22}(x,t)$. By calculations similar to those leading to Lemma~\ref{lem:3.4_pre}, we have
  \begin{align}
    & \mathcal{J}_{22}(x,t) \\
    & =-\frac{1}{2}\int_{0}^{t^{1/2}}\int_{-\infty}^{\infty}\partial_x(g_1-g_{1}^{*})(x-y,t-s)
    \begin{pmatrix}
      \theta_{2}^{2} \\
      \theta_{2}^{2}
    \end{pmatrix}
    (y,s)\, dyds+\frac{\gamma_2}{2}\int_{0}^{t^{1/2}}\int_{-\infty}^{\infty}\partial_{x}^{2}g_{2}^{*}(x-y,t-s)
    \begin{pmatrix}
      \theta_{2}^{2} \\
      \theta_{2}^{2}
    \end{pmatrix}
    (y,s)\, dyds \\
    & \quad -\frac{1}{4c}\int_{t^{1/2}}^{t}\int_{-\infty}^{\infty}L_1(g_1-g_{1}^{*})(x-y,t-s)
    \begin{pmatrix}
      \theta_{2}^{2} \\
      \theta_{2}^{2}
    \end{pmatrix}
    (y,s)\, dyds+\frac{\nu}{8c}\int_{t^{1/2}}^{t}\int_{-\infty}^{\infty}\partial_{x}^{2}g_{2}^{*}(x-y,t-s)
    \begin{pmatrix}
      \theta_{2}^{2} \\
      \theta_{2}^{2}
    \end{pmatrix}
    (y,s)\, dyds \\
    & \quad +\frac{1}{4c}\int_{t^{1/2}}^{t}\int_{-\infty}^{\infty}(g_1-g_{1}^{*})(x-y,t-s)
    \begin{pmatrix}
      L_2 \theta_{2}^{2} \\
      L_2 \theta_{2}^{2}
    \end{pmatrix}
    (y,s)\, dyds \\
    & \quad +\frac{1}{4c}\int_{-\infty}^{\infty}(g_1-g_{1}^{*})(x-y,t-t^{1/2})
    \begin{pmatrix}
      \theta_{2}^{2} \\
      \theta_{2}^{2}
    \end{pmatrix}
    (y,t^{1/2})\, dy \\
    & \eqqcolon \mathcal{J}_{221}(x,t)+\mathcal{K}(x,t)+\mathcal{J}_{223}(x,t)+\mathcal{J}_{224}(x,t).
  \end{align}
  Here, $L_2=\partial_t-c\partial_x-(\nu/2)\partial_{x}^{2}$ and the terms $\mathcal{J}_{221}(x,t)$, $\mathcal{K}(x,t)$, $\mathcal{J}_{223}(x,t)$, and $\mathcal{J}_{224}(x,t)$ correspond to the first, the second, the third, and the fourth row, respectively, of the right-hand side of the first equality. In the derivation, we used $g_{2}^{*}=(0\, g_{22}^{*})$, $\gamma_2=\nu/(4c)$, and $(g_1-g_{1}^{*})(x,0)=0$. For $\mathcal{J}_{221}(x,t)$, we use~\eqref{eq:gi-gi*_der},~Lemmas~\ref{lem:sqrt1} (with $\alpha=1/2$ and $\beta=1/4$),~\ref{lem:LZ98_lemma3.5_modified} (with $\alpha=2$ and $\beta=1/4$), and~\ref{lem:delta}; this shows that $|\mathcal{J}_{221}(x,t)|$ is bounded by the right-hand side of~\eqref{eq:J2_bound}. For $\mathcal{J}_{223}(x,t)$, we use~\eqref{eq:PWE_gi}, an obvious analogue of~\eqref{eq:delta1_vanishing},
  \begin{equation}
    L_2 \theta_{2}^{2}=-2\partial_x(\theta_{2}^{3}/3)+\nu(\partial_x \theta_2)^2,
  \end{equation}
  Lemmas~\ref{lem:lem3.2_improved} (with $\alpha=4$),~\ref{lem:LZ98_lemma3.5_modified} (with $\alpha=0$ and $\beta=9/4$), and~\ref{lem:LZ98_lemma3.6_modified} (with $\alpha=0$ and $\beta=9/4$); this shows that $|\mathcal{J}_{223}(x,t)|$ is bounded by the right-hand side of~\eqref{eq:J2_bound}. For $\mathcal{J}_{224}(x,t)$, we use~\eqref{eq:PWE_gi}, an obvious analogue of~\eqref{eq:delta1_vanishing}, and Lemma~\ref{lem:sqrt3} (with $\alpha=1/2$ and $\beta=1/4$); this shows that $|\mathcal{J}_{224}(x,t)|$ is bounded by the right-hand side of~\eqref{eq:J2_bound}. Finally, for $\mathcal{K}(x,t)$, we first divide it into two parts:
  \begin{align}
    & \mathcal{K}(x,t) \\
    & =-\frac{1}{4c}\int_{t^{1/2}}^{t/2}\int_{-\infty}^{\infty}L_1(g_1-g_{1}^{*})(x-y,t-s)
    \begin{pmatrix}
      \theta_{2}^{2} \\
      \theta_{2}^{2}
    \end{pmatrix}
    (y,s)\, dyds+\frac{\nu}{8c}\int_{t^{1/2}}^{t/2}\int_{-\infty}^{\infty}\partial_{x}^{2}g_{2}^{*}(x-y,t-s)
    \begin{pmatrix}
      \theta_{2}^{2} \\
      \theta_{2}^{2}
    \end{pmatrix}
    (y,s)\, dyds \\
    & \quad -\frac{1}{4c}\int_{t/2}^{t}\int_{-\infty}^{\infty}L_1(g_1-g_{1}^{*})(x-y,t-s)
    \begin{pmatrix}
      \theta_{2}^{2} \\
      \theta_{2}^{2}
    \end{pmatrix}
    (y,s)\, dyds+\frac{\nu}{8c}\int_{t/2}^{t}\int_{-\infty}^{\infty}\partial_{x}^{2}g_{2}^{*}(x-y,t-s)
    \begin{pmatrix}
      \theta_{2}^{2} \\
      \theta_{2}^{2}
    \end{pmatrix}
    (y,s)\, dyds \\
    & \eqqcolon \mathcal{K}_1(x,t)+\mathcal{K}_2(x,t).
  \end{align}
  Here, the terms $\mathcal{K}_1(x,t)$ and $\mathcal{K}_2(x,t)$ correspond to the first and the second row, respectively, of the right-hand side of the first equality. Let us first consider $\mathcal{K}_1(x,t)$. By~\eqref{eq:PWE_gi} and~\eqref{eq:Li_gi-gi*}, we have
  \begin{equation}
    \left| L_1(g_1-g_{1}^{*})(x,t)-\frac{\nu}{2}\partial_{x}^{2}g_{2}^{*}(x,t)-\frac{\nu}{2}e^{-\frac{c^2}{\nu}t}\sum_{j=0}^{2}\delta^{(2-j)}(x)q_{2j}(t) \right|\leq Ct^{-2}\left( e^{-\frac{(x-ct)^2}{Ct}}+e^{-\frac{(x+ct)^2}{Ct}} \right).
  \end{equation}
  Thus, we have
  \begin{align}
    |\mathcal{K}_1(x,t)|
    & \leq C\delta^2\int_{t^{1/2}}^{t/2}\int_{-\infty}^{\infty}(t-s)^{-1}(t-s+1)^{-1}e^{-\frac{(x-y-c(t-s))^2}{C(t-s)}}(s+1)^{-1/8}\psi_{7/4}(y,s;-c)\, dyds \\
    & \quad +C\delta^2 \int_{t^{1/2}}^{t/2}\int_{-\infty}^{\infty}(t-s)^{-1}(t-s+1)^{-1}e^{-\frac{(x-y+c(t-s))^2}{C(t-s)}}(s+1)^{-1/8}\psi_{7/4}(y,s;-c)\, dyds \\
    & \quad +C\delta^2 \int_{t^{1/2}}^{t/2}e^{-\frac{t-s}{C}}(s+1)^{-1/8}\psi(x,s;-c)\, ds.
  \end{align}
  Hence, by Lemmas~\ref{lem:LZ98_lemma3.5_modified} (with $\alpha=2$ and $\beta=1/4$),~\ref{lem:LZ98_lemma3.6_modified} (with $\alpha=2$ and $\beta=1/4$), and~\ref{lem:delta}, we see that $|\mathcal{K}_1(x,t)|$ is bounded by the right-hand side of~\eqref{eq:J2_bound}. For $\mathcal{K}_2(x,t)$, we use~\eqref{eq:Li_gi-gi*} and integration by parts to obtain
  \begin{align}
    \mathcal{K}_2(x,t)
    & =-\frac{\nu}{8c}\int_{t/2}^{t}\int_{-\infty}^{\infty}g_2(x-y,t-s)
    \begin{pmatrix}
      \theta_{2}^{2} \\
      \theta_{2}^{2}
    \end{pmatrix}_{xx}
    (y,s)\, dyds \\
    & \quad +\frac{\nu}{8c}\int_{t/2}^{t}\int_{-\infty}^{\infty}g_{2}^{*}(x-y,t-s)
    \begin{pmatrix}
      \theta_{2}^{2} \\
      \theta_{2}^{2}
    \end{pmatrix}_{xx}
    (y,s)\, dyds \\
    & =-\frac{\nu}{8c}\int_{t/2}^{t}\int_{-\infty}^{\infty}(g_2-g_{2}^{*})(x-y,t-s)
    \begin{pmatrix}
      \theta_{2}^{2} \\
      \theta_{2}^{2}
    \end{pmatrix}_{xx}
    (y,s)\, dyds.
  \end{align}
  Then, by~\eqref{eq:PWE_gi}, we have
  \begin{align}
    |\mathcal{K}_2(x,t)|
    & \leq C\delta^2\int_{t/2}^{t}\int_{-\infty}^{\infty}(t-s)^{-1}e^{-\frac{(x-y-c(t-s))^2}{C(t-s)}}(s+1)^{-9/8}\psi_{7/4}(y,s;-c)\, dyds \\
    & \quad +C\delta^2 \int_{t/2}^{t}\int_{-\infty}^{\infty}(t-s)^{-1}e^{-\frac{(x-y+c(t-s))^2}{C(t-s)}}(s+1)^{-9/8}\psi_{7/4}(y,s;-c)\, dyds \\
    & \quad +C\delta^2 \int_{t/2}^{t}e^{-\frac{t-s}{C}}(s+1)^{-9/8}\psi(x,s;-c)\, ds.
  \end{align}
  Hence, by Lemmas~\ref{lem:LZ98_lemma3.5_modified} (with $\alpha=0$ and $\beta=9/4$),~\ref{lem:LZ98_lemma3.6_modified} (with $\alpha=0$ and $\beta=9/4$), and~\ref{lem:delta} (with $\alpha=9/8$), we see that $|\mathcal{K}_2(x,t)|$ is bounded by the right-hand side of~\eqref{eq:J2_bound}.

  These calculations show~\eqref{eq:J2_bound}. Since the analysis of $\mathcal{J}_{3}(x,t)$ are similar and simpler, we omit this. We have thus showed~\eqref{SectionIII:Lemma:Bound_of_Ii:Bound_of_Nia}. This ends the proof.
\end{proof}

\begin{proof}[Proof of Theorem~\ref{thm:main}]
  By Lemmas~\ref{SectionIII:Lemma:Bound_of_Ii} and~\ref{SectionIII:Lemma:Bound_of_Ii:Bound_of_Ni}, we conclude that~\eqref{SectionIII:GoalInequality} holds. From this, we immediately obtain Theorem~\ref{thm:main} (cf.~the argument at the end of~\cite[Section~3.3]{Koike21}).
\end{proof}

\subsection{Proofs of Corollaries~\ref{cor:V_lower} and~\ref{cor:V_improved_decay}}
\label{sec:proofs_cors}
In this section, we prove Corollaries~\ref{cor:V_lower} and~\ref{cor:V_improved_decay}.

\begin{proof}[Proof of Corollary~\ref{cor:V_lower}]
  By~\eqref{right_ev} and~\eqref{eq:decomposition}, we have
  \begin{equation}
    u(x,t)=\frac{2c^2}{p''(1)}(u_1+u_2)(x,t).
  \end{equation}
  Also, by~\eqref{eq:theta_explicit} and~\eqref{def:Psi}, we have
  \begin{equation}
    |\theta_i(0,t)|+|\partial_x \theta_i(0,t)|\leq C\delta e^{-\frac{t}{C}}, \quad |\Psi_i(0,t)|\leq C\delta(t+1)^{-7/4}.
  \end{equation}
  Therefore, by Theorem~\ref{thm:main}, we obtain
  \begin{equation}
    \label{cor:eq:V_lower:proof:eq1}
    \left| u(0_{\pm},t)-\frac{2c^2}{p''(1)}(\xi_1+\xi_2)(0,t) \right| \leq C\delta(t+1)^{-7/4}.
  \end{equation}
  Using this inequality and Lemma~\ref{lem:xi_lower}, \eqref{lem:eq4:xi_lower} and~\eqref{lem:eq3:xi_lower} in particular, we obtain
  \begin{equation}
    \left| u(0_{\pm},t)-\mathcal{W}(t) \right| \leq C[\delta(t+1)^{-7/4}+\delta^3(t+1)^{-3/2}].
  \end{equation}
  Here, $\mathcal{W}(t)$ is the function defined by~\eqref{def:W}. Now, note that $V(t)=u(0_{\pm},t)$. Hence, when $\delta$ is sufficiently small, using~\eqref{lem:eq2:xi_lower} in Lemma~\ref{lem:xi_lower}, we conclude that
  \begin{equation}
    C^{-1}|M_{1}^{2}-M_{2}^{2}|(t+1)^{-3/2}-C\delta(t+1)^{-7/4}\leq (\sgn(M_{1}^{2}-M_{2}^{2}))V(t)
  \end{equation}
  for some $C>1$. Finally, by taking $T=T(\delta)>0$ sufficiently large, we obtain~\eqref{cor:eq:V_lower}. This ends the proof of Corollary~\ref{cor:V_lower}.
\end{proof}

\begin{proof}[Proof of Corollary~\ref{cor:V_improved_decay}]
  First, by Lemma~\ref{lem:xi_cancellation}, we have $|\mathcal{V}(t)|\leq C\delta^2(t+1)^{-2}$ for $t\geq 1$; here, $\mathcal{V}(t)$ is the function define by~\eqref{def:V}. From this inequality,~\eqref{lem:eq4:xi_lower} in Lemma~\ref{lem:xi_lower}, and~\eqref{cor:eq:V_lower:proof:eq1}, we conclude that
  \begin{equation}
    |u(0_{\pm},t)|\leq C\delta(t+1)^{-7/4}
  \end{equation}
  for some $C>0$. Since $V(t)=u(0_{\pm},t)$, we obtain~\eqref{cor:eq:V_improved_decay}. This ends the proof of Corollary~\ref{cor:V_improved_decay}.
\end{proof}

\subsubsection*{Acknowledgements}
First, I thank Yasunori Maekawa for his comment: he pointed out that the inter-diffusion wave $\xi_i$ appearing in this paper is a solution to a single PDE~\eqref{def:xii}; previously, $\xi_i$ was represented as an infinite sum of solutions to certain PDEs (this infinite series representation is nevertheless important and is used in the proof of Lemma~\ref{lem:xi}). His comment made the presentation much clearer and shorter. I also thank Shih-Hsien Yu for informing me about the structure of Green's function when I visited National University of Singapore in 2019, which was financially supported by Grant-in-Aid for JSPS Research Fellow (Grant Number 18J20574); this, although indirectly, helped me find a useful reference~\cite{LZ09}. I also express my gratitude to Masanari Hattori and Shigeru Takata for supplying me with the computational resources and advices for the relevant numerical simulations presented in~\cite{Koike21RIMS}; these simulations helped me understand the behavior of the functions appearing in this paper. Finally, this work was financially supported by Grant-in-Aid for JSPS Research Fellow (Grant Number 20J00882).

\appendix
\renewcommand*{\thesection}{\Alph{section}}
\section{General lemmas on pointwise estimates of convolutions}
\label{AppendixC}
In this appendix, we gather some general lemmas that are used in this paper. For ease of reference, we first state the lemmas without proofs (unless the proof is short), and the proofs are presented afterwards. Note that in what follows, as in other places, the symbols $C$ and $\nu^*$ denote generic large constants.

The following lemma --- for smooth functions $f=f(x,t)$ on $\mathbb{R}\times (0,\infty)$ --- is proved inside the proof of~\cite[Lemma~3.4]{LZ97}. Here, we allow functions to have a discontinuity at $x=0$.

\begin{lem}
  \label{lem:3.4_pre}
  Suppose that $f=f(x,t)$ is a function on $\mathbb{R}_* \times (0,\infty)$ and that $\partial_{t}^{k}\partial_{x}^{l}f$ is bounded and continuous on $\mathbb{R}_* \times (0,\infty)$ for $0\leq k\leq 1$ and $0\leq l\leq 2$. Let $\lambda \neq \lambda'$, $\nu>0$, $t\geq 4$, and $L_{\lambda'}=\partial_t+\lambda' \partial_x-(\nu/2)\partial_{x}^{2}$. Then the function $I(x,t)$ defined by
  \begin{equation}
    I(x,t)=\int_{0}^{t}\int_{-\infty}^{\infty}\partial_x \left\{ (t-s)^{-1/2}e^{-\frac{(x-y-\lambda(t-s))^2}{2\nu(t-s)}} \right\} f(y,s)\, dyds
  \end{equation}
  can be written as
  \begin{equation}
    I(x,t)=(\lambda-\lambda')^{-1}\sqrt{2\pi \nu}f(x,t)+I_1(x,t)+I_{21}(x,t)+I_{22}(x,t)+I_b(x,t),
  \end{equation}
  where
  \begin{equation}
    I_1(x,t)=\int_{0}^{t^{1/2}}\int_{-\infty}^{\infty}\partial_x \left\{ (t-s)^{-1/2}e^{-\frac{(x-y-\lambda(t-s))^2}{2\nu(t-s)}} \right\} f(y,s)\, dyds,
  \end{equation}
  \begin{equation}
    I_{21}(x,t)=-(\lambda-\lambda')^{-1}\int_{-\infty}^{\infty}(t-t^{1/2})^{-1/2}e^{-\frac{(x-y-\lambda(t-\sqrt{t}))^2}{2\nu(t-\sqrt{t})}}f(y,t^{1/2})\, dy,
  \end{equation}
  \begin{equation}
    I_{22}(x,t)=-(\lambda-\lambda')^{-1}\int_{t^{1/2}}^{t}\int_{-\infty}^{\infty}(t-s)^{-1/2}e^{-\frac{(x-y-\lambda(t-s))^2}{2\nu(t-s)}}L_{\lambda'}f(y,s)\, dyds,
  \end{equation}
  and
  \begin{align}
    I_b(x,t)
    & =-\lambda'(\lambda-\lambda')^{-1}\int_{t^{1/2}}^{t}(t-s)^{-1/2}e^{-\frac{(x-\lambda(t-s))^2}{2\nu(t-s)}}\llbracket f \rrbracket(s)\, ds \\
    & \quad -(\nu/2)(\lambda-\lambda')^{-1}\int_{t^{1/2}}^{t}(t-s)^{-1/2}e^{-\frac{(x-\lambda(t-s))^2}{2\nu(t-s)}}\llbracket \partial_x f \rrbracket(s)\, ds \\
    & \quad -(\nu/2)(\lambda-\lambda')^{-1}\int_{t^{1/2}}^{t}\partial_x \left\{ (t-s)^{-1/2}e^{-\frac{(x-\lambda(t-s))^2}{2\nu(t-s)}} \right\} \llbracket f \rrbracket(s)\, ds.
  \end{align}
\end{lem}

\begin{proof}
  This lemma is proved inside the proof of~\cite[Lemma~B.3]{Koike21}.
\end{proof}

For $\lambda \in \mathbb{R}$ and $\mu>0$, we set
\begin{equation}
  \Theta_{\alpha}(x,t;\lambda,\mu)\coloneqq (t+1)^{-\alpha/2}e^{-\frac{(x-\lambda(t+1))^2}{\mu(t+1)}}
\end{equation}
and
\begin{equation}
  \psi_{7/4}(x,t;\lambda)\coloneqq [(x-\lambda(t+1))^2+(t+1)]^{-7/8}, \quad \bar{\psi}(x,t;\lambda)\coloneqq [|x-\lambda(t+1)|^7+(t+1)^5]^{-1/4}.
\end{equation}
Note that $\Theta_{\alpha}(x,t;\lambda,\mu)$ is the one already defined in~\eqref{def:Theta} and the definitions of $\psi_{7/4}(x,t;\lambda)$ and $\bar{\psi}(x,t;\lambda)$ are consistent with~\eqref{eq:psi74} and~\eqref{eq:psibar}.

\begin{lem}
  \label{lem:sqrt1}
  Let $\lambda \neq \lambda'$, $\mu>0$, $\alpha \geq 0$, and $0\leq \beta<5/4$. Then
  \begin{equation}
    \int_{0}^{t^{1/2}}\int_{-\infty}^{\infty}(t-s)^{-1-\alpha}e^{-\frac{(x-y-\lambda(t-s))^2}{\mu(t-s)}}(s+1)^{-\beta/2}\psi_{7/4}(y,s;\lambda')\, dyds\leq C(t+1)^{-\alpha-(\beta-3/4)/4}\psi_{7/4}(x,t;\lambda)
  \end{equation}
  for $t\geq 4$.
\end{lem}

\begin{lem}
  \label{lem:sqrt3}
  Let $\lambda,\lambda' \in \mathbb{R}$, $\mu>0$, and $\alpha,\beta \geq 0$ (not necessarily $\lambda \neq \lambda'$). Suppose that
  \begin{equation}
    |f(x,t)|\leq C(t+1)^{-\beta/2}\psi_{7/4}(x,t;\lambda').
  \end{equation}
  Then
  \begin{equation}
    \int_{-\infty}^{\infty}(t-t^{1/2})^{-1/2-\alpha}e^{-\frac{(x-y-\lambda(t-\sqrt{t}))^2}{\mu(t-\sqrt{t})}}|f(y,t^{1/2})|\, dy\leq C(t+1)^{-\alpha-(\beta-3/4)/4}\psi_{7/4}(x,t;\lambda)
  \end{equation}
  for $t\geq 4$.
\end{lem}

\begin{lem}
  \label{lem:I221}
  Let $\lambda \neq \lambda'$, $\mu>0$, and $\alpha>0$. Then
  \begin{equation}
    \int_{t^{1/2}}^{t}\int_{-\infty}^{\infty}(t-s)^{-1/2}e^{-\frac{(x-y-\lambda(t-s))^2}{\mu(t-s)}}\Theta_{\alpha}(y,s;\lambda',\mu)\, dyds\leq C\psi_{(\alpha-1)/2}(x,t;\lambda)
  \end{equation}
  for $t\geq 4$.
\end{lem}

\begin{proof}
  See the analysis of $I_{22}^{(1)}$ in~\cite[p.~24--25]{LZ97}.
\end{proof}

\begin{lem}
  \label{lem:lem3.2_improved}
  Let $\lambda \in \mathbb{R}$, $\mu>0$, and $\alpha>3$. Then
  \begin{equation}
    \int_{t^{1/2}}^{t/2}\int_{-\infty}^{\infty}(t-s)^{-1}e^{-\frac{(x-y-\lambda(t-s))^2}{\mu(t-s)}}\Theta_{\alpha}(y,s;\lambda,\mu)\, dyds\leq C\Theta_{(\alpha+1)/2}(x,t;\lambda,\mu)
  \end{equation}
  for $t\geq 4$.
\end{lem}

\begin{lem}
  \label{lem:sqrt_ij}
  Let $\lambda \neq \lambda'$, $\mu>0$, and $\alpha \geq 0$. Then
  \begin{equation}
    \int_{t^{1/2}}^{t/2}\int_{-\infty}^{\infty}(t-s)^{-1/2-\alpha}e^{-\frac{(x-y-\lambda(t-s))^2}{\mu(t-s)}}(s+1)^{-11/8}\psi_{7/4}(y,s;\lambda')\, dyds\leq C(t+1)^{-\alpha}\psi_{7/4}(x,t;\lambda)
  \end{equation}
  for $t\geq 4$.
\end{lem}

\begin{lem}
  \label{lem:LZ98_lemma3.5_modified}
  Let $\lambda \in \mathbb{R}$, $\mu>0$, and $\alpha,\beta \geq 0$. Then we have
  \begin{align}
    \label{lem:LZ98_lemma3.5_modified:eq1}
    \begin{aligned}
      & \int_{0}^{t/2}\int_{-\infty}^{\infty}(t-s)^{-1}(t+1-s)^{-\alpha/2}e^{-\frac{(x-y-\lambda(t-s))^2}{\mu(t-s)}}(s+1)^{-\beta/2}\psi_{7/4}(y,s;\lambda)\, dyds \\
      & \leq
      \begin{dcases}
        C(t+1)^{-\gamma_1/2}\psi_{7/4}(x,t;\lambda) & \text{if $\beta \neq 5/4$}, \\
        C\log(t+2)(t+1)^{-\gamma_1/2}\psi_{7/4}(x,t;\lambda) & \text{if $\beta=5/4$},
      \end{dcases}
    \end{aligned}
  \end{align}
  where $\gamma_1=\alpha+\min(\beta,5/4)-1$. We also have
  \begin{align}
    \label{lem:LZ98_lemma3.5_modified:eq2}
    \begin{aligned}
      & \int_{t/2}^{t}\int_{-\infty}^{\infty}(t-s)^{-1}(t+1-s)^{-\alpha/2}e^{-\frac{(x-y-\lambda(t-s))^2}{\mu(t-s)}}(s+1)^{-\beta/2}\psi_{7/4}(y,s;\lambda)\, dyds \\
      & \leq
      \begin{dcases}
        C(t+1)^{-\gamma_2/2}\psi_{7/4}(x,t;\lambda) & \text{if $\alpha \neq 1$}, \\
        C\log(t+2)(t+1)^{-\gamma_2/2}\psi_{7/4}(x,t;\lambda) & \text{if $\alpha=1$},
      \end{dcases}
    \end{aligned}
  \end{align}
  where $\gamma_2=\min(\alpha,1)+\beta-1$.
\end{lem}

Let
\begin{equation}
  \label{def:chiK}
  \chi_K(x,t;\lambda,\lambda')\coloneqq \mathrm{char}\left\{ \min(\lambda,\lambda')(t+1)+K(t+1)^{1/2}\leq x\leq \max(\lambda,\lambda')(t+1)-K(t+1)^{1/2} \right\},
\end{equation}
where $K>0$ and $\mathrm{char}\{ S \}$ is the indicator function of a set $S$.

\begin{lem}
  \label{lem:LZ98_lemma3.6_modified}
  Let $\lambda \neq \lambda'$, $\mu>0$, $\alpha \geq 0$, and $0\leq \beta \leq 7/2$ ($\beta \neq 2$). Then for $K>0$ large enough, we have
  \begin{align}
    \label{lem:LZ98_lemma3.6_modified:eq1}
    \begin{aligned}
      & \int_{0}^{t/2}\int_{-\infty}^{\infty}(t-s)^{-1}(t+1-s)^{-\alpha/2}e^{-\frac{(x-y-\lambda(t-s))^2}{\mu(t-s)}}(s+1)^{-\beta/2}\psi_{7/4}(y,s;\lambda')\, dyds \\
      & \leq
      \begin{dcases}
        C[(t+1)^{-\gamma_1/2}\psi_{7/4}(x,t;\lambda)+(t+1)^{-\gamma_{1}'/2}\psi_{7/4}(x,t;\lambda')] & \text{if $\beta \neq 5/4$} \\
        C[\log(t+2)(t+1)^{-\gamma_1/2}\psi_{7/4}(x,t;\lambda)+(t+1)^{-\gamma_{1}'/2}\psi_{7/4}(x,t;\lambda')] & \text{if $\beta=5/4$}
      \end{dcases} \\
      & \quad +C|x-\lambda(t+1)|^{-\min(\beta,11/4)/2-3/8}|x-\lambda'(t+1)|^{-\alpha/2-1/2}\chi_K(x,t;\lambda,\lambda'),
    \end{aligned}
  \end{align}
  where $\gamma_1=\alpha+\min(\beta,5/4)-1$ and $\gamma_{1}'=\alpha+\min(\beta,2)-1$. We also have
  \begin{align}
    \label{lem:LZ98_lemma3.6_modified:eq2}
    \begin{aligned}
      & \int_{t/2}^{t}\int_{-\infty}^{\infty}(t-s)^{-1}(t+1-s)^{-\alpha/2}e^{-\frac{(x-y-\lambda(t-s))^2}{\mu(t-s)}}(s+1)^{-\beta/2}\psi_{7/4}(y,s;\lambda')\, dyds \\
      & \leq
      \begin{dcases}
        C(t+1)^{-\gamma_2/2}[\psi_{7/4}(x,t;\lambda)+\psi_{7/4}(x,t;\lambda')] & \text{if $\alpha \neq 1$} \\
        C\log(t+2)(t+1)^{-\gamma_2/2}[\psi_{7/4}(x,t;\lambda)+\psi_{7/4}(x,t;\lambda')] & \text{if $\alpha=1$}
      \end{dcases} \\
      & \quad +C|x-\lambda(t+1)|^{-\beta/2-3/8}|x-\lambda'(t+1)|^{-\min(\alpha,1)/2-1/2}\chi_K(x,t;\lambda,\lambda'),
    \end{aligned}
  \end{align}
  where $\gamma_2=\min(\alpha,1)+\beta-1$.
\end{lem}

\begin{lem}
  \label{lem:boundary1}
  Let $\lambda \in \mathbb{R}$ and $\mu>0$. Then
  \begin{equation}
    \int_{t^{1/2}}^{t}(t-s)^{-1/2}e^{-\frac{(x-\lambda(t-s))^2}{\mu(t-s)}}(s+1)^{-21/8}\, ds\leq C\bar{\psi}(x,t;\lambda)
  \end{equation}
  for $t\geq 4$.
\end{lem}

\begin{lem}
  \label{lem:boundary2}
  Let $\lambda \neq 0$ and $\mu>0$. Suppose that $|f(t)|\leq C(t+1)^{-9/4}$ and $|\partial_t f(t)|\leq C(t+1)^{-7/4}$. Then
  \begin{equation}
    \left| \int_{t^{1/2}}^{t}\partial_x \left\{ (t-s)^{-1/2}e^{-\frac{(x-\lambda(t-s))^2}{\mu(t-s)}} \right\} f(s)\, ds \right| \leq C\bar{\psi}(x,t;\lambda)
  \end{equation}
  for $x\neq 0$ and $t\geq 4$.
\end{lem}

\begin{lem}
  \label{lem:delta}
  Let $\lambda \in \mathbb{R}$, $\mu>0$, and $\alpha \geq 0$. Then
  \begin{equation}
    \int_{0}^{t}e^{-\frac{t-s}{\mu}}(s+1)^{-\alpha}\psi_{7/4}(x,s;\lambda)\, ds\leq C(t+1)^{-\alpha}\psi_{7/4}(x,t;\lambda).
  \end{equation}
  We also have
  \begin{equation}
    \int_{0}^{t/2}e^{-\frac{t-s}{\mu}}\psi_{7/4}(x,s;\lambda)\, ds\leq Ce^{-\frac{t}{C}}\psi_{7/4}(x,t;\lambda).
  \end{equation}
\end{lem}

\begin{proof}
  These can be proved by slightly modifying the proof of~\cite[Lemma~3.9]{LZ97}.
\end{proof}

\begin{proof}[Proof of Lemma~\ref{lem:sqrt1}]
  We only consider the case of $(\lambda,\lambda',\alpha)=(0,1,0)$ since the other cases are similar. Denote by $I(x,t)$ the integral appearing in the statement of the lemma. Then
  \begin{align}
    I(x,t)
    & \leq C(t+1)^{-1}\int_{0}^{t^{1/2}}\int_{-\infty}^{s+1-(s+1)^{1/2}}(s+1)^{-\beta/2}(s+1-y)^{-7/4}e^{-\frac{(y-x)^2}{\mu(t-s)}}\, dyds \\
    & \quad +C(t+1)^{-1}\int_{0}^{t^{1/2}}\int_{s+1-(s+1)^{1/2}}^{s+1+(s+1)^{1/2}}(s+1)^{-(\beta+7/4)/2}e^{-\frac{(y-x)^2}{\mu(t-s)}}\, dyds \\
    & \quad +C(t+1)^{-1}\int_{0}^{t^{1/2}}\int_{s+1+(s+1)^{1/2}}^{\infty}(s+1)^{-\beta/2}(y-s-1)^{-7/4}e^{-\frac{(y-x)^2}{\mu(t-s)}}\, dyds \\
    & \eqqcolon I_1(x,t)+I_2(x,t)+I_3(x,t).
  \end{align}
  For $I_2(x,t)$, using
  \begin{align}
    \label{eq:LZ97_39}
    \begin{aligned}
      \frac{(x-y-\lambda(t-s))^2}{t-s}+\frac{(y-\lambda'(s+1))^2}{s+1}
      & =\frac{t+1}{(t-s)(s+1)}\left[ y-\frac{(s+1)(x-(\lambda-\lambda')(t-s))}{t+1} \right]^2 \\
      & \quad +\frac{(x-\lambda(t-s)-\lambda'(s+1))^2}{t+1},
    \end{aligned}
  \end{align}
  we obtain
  \begin{align}
    I_2(x,t)
    & \leq C(t+1)^{-1}\int_{0}^{t^{1/2}}\int_{s+1-(s+1)^{1/2}}^{s+1+(s+1)^{1/2}}(s+1)^{-(\beta+7/4)/2}e^{-\frac{(y-x)^2}{\mu(t-s)}}e^{-\frac{(y-(s+1))^2}{\mu(s+1)}}\, dyds \\
    & \leq C(t+1)^{-1}\int_{0}^{t^{1/2}}(s+1)^{-(\beta+3/4)/2}e^{-\frac{(x-(s+1))^2}{\mu(t+1)}}\, ds \\
    & \leq 
    \begin{dcases}
      C(t+1)^{-(\beta+11/4)/4} & \text{if $|x|\leq 2(t+1)^{1/2}$} \\
      C(t+1)^{-(\beta+11/4)/4}e^{-\frac{x^2}{C(t+1)}} & \text{if $|x|>2(t+1)^{1/2}$}
    \end{dcases} \\
    & \leq C(t+1)^{-(\beta-3/4)/4}\Theta_{7/4}(x,t;0,\nu^*)\leq C(t+1)^{-(\beta-3/4)/4}\psi_{7/4}(x,t;0).
  \end{align}
  We next give estimates for $I_1(x,t)$ and $I_3(x,t)$. Let $K$ be a sufficiently large positive number. Consider first the case of (i) $x<-K(t+1)^{1/2}$. In this case, we have $I_3(x,t)\leq CI_1(x,t)$ and
  \begin{align}
    I_1(x,t)
    & \leq C(t+1)^{-1}\left( \int_{0}^{t^{1/2}}\int_{x/2}^{s+1-(s+1)^{1/2}}+\int_{0}^{t^{1/2}}\int_{-\infty}^{x/2} \right) (s+1)^{-\beta/2}(s+1-y)^{-7/4}e^{-\frac{(y-x)^2}{\mu(t+1)}}\, dyds \\
    & \leq C(t+1)^{-1}e^{-\frac{x^2}{C(t+1)}}\int_{0}^{t^{1/2}}(s+1)^{-(\beta+3/4)/2}\, ds+C(t+1)^{-1/2}|x|^{-7/4}\int_{0}^{t^{1/2}}(s+1)^{-\beta/2}\, ds \\
    & \leq C(t+1)^{-(\beta+11/4)/4}e^{-\frac{x^2}{C(t+1)}}+C(t+1)^{-\beta/4}|x|^{-7/4}\leq C(t+1)^{-(\beta-3/4)/4}\psi_{7/4}(x,t;0).
  \end{align}
  Next, let us consider the case of (ii) $|x|\leq K(t+1)^{1/2}$. In this case, we have
  \begin{equation}
    I_1(x,t)+I_3(x,t)\leq C(t+1)^{-1}\int_{0}^{t^{1/2}}(s+1)^{-(\beta+3/4)/2}\, ds\leq C(t+1)^{-(\beta+11/4)/4}\leq C(t+1)^{-(\beta-3/4)/4}\psi_{7/4}(x,t;0).
  \end{equation}
  Finally, for the case of (iii) $x>K(t+1)^{1/2}$, we have $I_1(x,t)\leq CI_3(x,t)$ and
  \begin{align}
    I_3(x,t)
    & \leq C(t+1)^{-1}\left( \int_{0}^{t^{1/2}}\int_{s+1+(s+1)^{1/2}}^{x/2}+\int_{0}^{t^{1/2}}\int_{x/2}^{\infty} \right)(s+1)^{-\beta/2}(y-s-1)^{-7/4}e^{-\frac{(y-x)^2}{\mu(t-s)}}\, dyds \\
    & \leq C(t+1)^{-1}e^{-\frac{x^2}{C(t+1)}}\int_{0}^{t^{1/2}}(s+1)^{-(\beta+3/4)/2}\, ds+C(t+1)^{-1/2}|x|^{-7/4}\int_{0}^{t^{1/2}}(s+1)^{-\beta/2}\, ds \\
    & \leq C(t+1)^{-(\beta+11/4)/4}e^{-\frac{x^2}{C(t+1)}}+C(t+1)^{-\beta/4}|x|^{-7/4}\leq C(t+1)^{-(\beta-3/4)/4}\psi_{7/4}(x,t;0).
  \end{align}
  This ends the proof.
\end{proof}

\begin{proof}[Proof of Lemma~\ref{lem:sqrt3}]
  We only consider the case of $(\lambda,\lambda',\alpha,\beta)=(0,1,0,3/4)$ since the other cases are similar. Denote by $I(x,t)$ the integral appearing in the statement of the lemma. Let $K$ be a sufficiently large positive number. First, let us consider the case of (i) $|x|\leq K(t+1)^{1/2}$. In this case, we have
  \begin{align}
    I(x,t)
    & \leq C(t+1)^{-9/8}\int_{|y-(\sqrt{t}+1)|\leq (\sqrt{t}+1)^{1/2}}e^{-\frac{(x-y)^2}{\mu(t-\sqrt{t})}}\, dy \\
    & \quad +C(t+1)^{-11/16}\int_{|y-(\sqrt{t}+1)|>(\sqrt{t}+1)^{1/2}}|y-(t^{1/2}+1)|^{-7/4}\, dy\leq C(t+1)^{-7/8}\leq C\psi_{7/4}(x,t;0).
  \end{align}
  Next, let us consider the case of (ii) $|x|>K(t+1)^{1/2}$. Suppose that $x>0$; the case of $x<0$ is similar. Then we have
  \begin{align}
    I(x,t)
    & \leq C\left( \int_{-\infty}^{x/2}+\int_{x/2}^{\infty} \right) (t-t^{1/2})^{-1/2}e^{-\frac{(x-y)^2}{\mu(t-\sqrt{t})}}|f(y,t^{1/2})|\, dy \\
    & \leq Ce^{-\frac{x^2}{C(t+1)}}\int_{-\infty}^{\infty}(t-t^{1/2})^{-1/2}e^{-\frac{(x-y)^2}{\mu(t-\sqrt{t})}}|f(y,t^{1/2})|\, dy+C(t+1)^{-11/16}|x|^{-7/4}\int_{-\infty}^{\infty}e^{-\frac{(x-y)^2}{\mu(t-\sqrt{t})}}\, dy \\
    & \leq C(t+1)^{-7/8}e^{-\frac{x^2}{C(t+1)}}+C(t+1)^{-3/16}|x|^{-7/4}\leq C\psi_{7/4}(x,t;0).
  \end{align}
  For the third inequality, we repeated the argument in Case (i). This ends the proof.
\end{proof}

\begin{proof}[Proof of Lemma~\ref{lem:lem3.2_improved}]
  Using~\eqref{eq:LZ97_39}, we see that the integral in the statement of the lemma is bounded by
  \begin{equation}
    C(t+1)^{-1/2}\int_{t^{1/2}}^{t/2}(t-s)^{-1/2}(s+1)^{-(\alpha-1)/2}e^{-\frac{(x-\lambda(t+1))^2}{\mu(t+1)}}\, ds.
  \end{equation}
  This is then bounded by
  \begin{equation}
    C\Theta_2(x,t;\lambda,\mu)\int_{t^{1/2}}^{t/2}(s+1)^{-(\alpha-1)/2}\, ds\leq C\Theta_{(\alpha+1)/2}(x,t;\lambda,\mu).
  \end{equation}
  This ends the proof.
\end{proof}

\begin{proof}[Proof of Lemma~\ref{lem:sqrt_ij}]
  We only consider the case of $(\lambda,\lambda',\alpha)=(0,1,0)$ since the other cases are similar. Denote by $I(x,t)$ the integral appearing in the statement of the lemma. Let
  \begin{align}
    I(x,t)
    & =\int_{t^{1/2}}^{t/2}\int_{|y-(s+1)|\leq (s+1)^{1/2}}(t-s)^{-1/2}e^{-\frac{(x-y)^2}{\mu(t-s)}}(s+1)^{-11/8}\psi_{7/4}(y,s;1)\, dyds \\
    & \quad +\int_{t^{1/2}}^{t/2}\int_{|y-(s+1)|>(s+1)^{1/2}}(t-s)^{-1/2}e^{-\frac{(x-y)^2}{\mu(t-s)}}(s+1)^{-11/8}\psi_{7/4}(y,s;1)\, dyds \\
    & \eqqcolon I_1(x,t)+I_2(x,t).
  \end{align}

  For $I_1(x,t)$, applying Lemma~\ref{lem:I221} (with $\alpha=9/2$), we obtain
  \begin{equation}
    I_1(x,t)\leq C\int_{t^{1/2}}^{t}\int_{|y-(s+1)|\leq (s+1)^{1/2}}(t-s)^{-1/2}e^{-\frac{(x-y)^2}{\mu(t-s)}}(s+1)^{-9/4}e^{-\frac{(y-(s+1))^2}{\mu(s+1)}}\, dyds\leq C\psi_{7/4}(x,t;0).
  \end{equation}

  For $I_2(x,t)$, we consider three cases separately. Let $K$ be a sufficiently large positive number. For the case of (i) $|x|\leq K(t+1)^{1/2}$, we have
  \begin{equation}
    I_2(x,t)\leq C(t+1)^{-1/2}\int_{t^{1/2}}^{t/2}(s+1)^{-7/4}\, ds\leq C(t+1)^{-7/8}\leq C\psi_{7/4}(x,t;0).
  \end{equation}
  For the case of (ii) $x<-K(t+1)^{1/2}$, we have
  \begin{align}
    I_2(x,t)
    & =\int_{t^{1/2}}^{t/2}\left( \int_{-\infty}^{x/2}+\int_{x/2}^{s+1-(s+1)^{1/2}}+\int_{s+1+(s+1)^{1/2}}^{\infty} \right) (t-s)^{-1/2}e^{-\frac{(x-y)^2}{\mu(t-s)}}(s+1)^{-11/8}\psi_{7/4}(y,s;1)\, dyds \\
    & \leq C|x|^{-7/4}\int_{t^{1/2}}^{t/2}(s+1)^{-11/8}\, ds+C(t+1)^{-1/2}e^{-\frac{x^2}{C(t+1)}}\int_{t^{1/2}}^{t/2}(s+1)^{-7/4}\, ds \\
    & \leq C(t+1)^{-3/16}|x|^{-7/4}+C\Theta_{7/4}(x,t;0,\nu^*)\leq C\psi_{7/4}(x,t;0).
  \end{align}
  Next, let us consider the case of (iii) $x\geq t+1-K(t+1)^{1/2}$. In what follows, we assume that $t$ is sufficiently large; the statement of the lemma is otherwise easy to check. We have
  \begin{align}
    I_2(x,t)
    & =\int_{t^{1/2}}^{t/2}\left( \int_{-\infty}^{s+1-(s+1)^{1/2}}+\int_{s+1+(s+1)^{1/2}}^{2x/3}+\int_{2x/3}^{\infty} \right) (t-s)^{-1/2}e^{-\frac{(x-y)^2}{\mu(t-s)}}(s+1)^{-11/8}\psi_{7/4}(y,s;1)\, dyds \\
    & \leq C(t+1)^{-1/2}e^{-\frac{x^2}{C(t+1)}}\int_{t^{1/2}}^{t/2}(s+1)^{-7/4}\, ds+C|x|^{-7/4}\int_{t^{1/2}}^{t/2}(s+1)^{-11/8}\, ds \\
    & \leq C\Theta_{7/4}(x,t;0,\nu^*)+C(t+1)^{-3/16}|x|^{-7/4}\leq C\psi_{7/4}(x,t;0).
  \end{align}
  Finally, let us consider the case of (iv) $K(t+1)^{1/2}<x<t+1-K(t+1)^{1/2}$. Let
  \begin{align}
    I_2(x,t)
    & =\int_{t^{1/2}}^{t/2}\int_{y<s+1-(s+1)^{1/2}}(t-s)^{-1/2}e^{-\frac{(x-y)^2}{\mu(t-s)}}(s+1)^{-11/8}\psi_{7/4}(y,s;1)\, dyds \\
    & \quad +\int_{t^{1/2}}^{t/2}\int_{y>s+1+(s+1)^{1/2}}(t-s)^{-1/2}e^{-\frac{(x-y)^2}{\mu(t-s)}}(s+1)^{-11/8}\psi_{7/4}(y,s;1)\, dyds \\
    & \eqqcolon J_2(x,t)+K_2(x,t).
  \end{align}
  For $J_2(x,t)$, we further divide this into four terms:
  \begin{align}
    J_2(x,t)
    & \leq C(t+1)^{-1/2}\int_{t^{1/2}}^{x-x^{1/2}}e^{-\frac{(x-(s+1))^2}{\mu(t-s)}}(s+1)^{-7/4}\, ds \\
    & \quad +C\int_{x-x^{1/2}}^{x+4x^{1/2}}(s+1)^{-9/4}\, ds \\
    & \quad +C\int_{x+4x^{1/2}}^{t}\int_{y\leq (s+1+x)/2}(t-s)^{-1/2}e^{-\frac{(x-y)^2}{\mu(t-s)}}(s+1)^{-11/8}(s+1-x)^{-7/4}\, dyds \\
    & \quad +C(t+1)^{-1/2}\int_{x+4x^{1/2}}^{t}\int_{(s+1+x)/2}^{s+1-(s+1)^{1/2}}e^{-\frac{(s+1-x)^2}{C(t-s)}}(s+1)^{-11/8}(s+1-y)^{-7/4}\, dyds \\
    & \eqqcolon J_{21}(x,t)+J_{22}(x,t)+J_{23}(x,t)+J_{24}(x,t).
  \end{align}
  The first term $J_{21}(x,t)$ is estimated as follows:
  \begin{align}
    J_{21}(x,t)
    & \leq C(t+1)^{-1/2}e^{-\frac{x^2}{C(t+1)}}\int_{t^{1/2}}^{x/2-1}(s+1)^{-7/4}\, ds+C(t+1)^{-1/2}|x|^{-7/4}\int_{x/2-1}^{t}e^{-\frac{(x-(s+1))^2}{\mu(t+1)}}\, ds \\
    & \leq C\Theta_{7/4}(x,t;0,\nu^*)+C|x|^{-7/4}\leq C\psi_{7/4}(x,t;0).
  \end{align}
  For the second term $J_{22}(x,t)$, we have
  \begin{equation}
    J_{22}(x,t)\leq C\int_{x-x^{1/2}}^{x+4x^{1/2}}(x-x^{1/2}+1)^{-9/4}\, ds\leq C|x|^{-7/4}\leq C\psi_{7/4}(x,t;0).
  \end{equation}
  For the third term $J_{23}(x,t)$, we have
  \begin{equation}
    J_{23}(x,t)\leq C|x|^{-11/8}\int_{x+4x^{1/2}}^{t}(s+1-x)^{-7/4}\, ds\leq C|x|^{-7/4}\leq C\psi_{7/4}(x,t;0).
  \end{equation}
  For the fourth term $J_{24}(x,t)$, we have
  \begin{equation}
    J_{24}(x,t)\leq C(t+1)^{-1/2}|x|^{-7/4}\int_{x+4x^{1/2}}^{t}e^{-\frac{(s+1-x)^2}{C(t+1)}}\, ds\leq C|x|^{-7/4}\leq C\psi_{7/4}(x,t;0).
  \end{equation}
  Next, for $K_2(x,t)$, we again divide this into four terms:
  \begin{align}
    K_2(x,t)
    & \leq C(t+1)^{-1/2}\int_{t^{1/2}}^{x-4x^{1/2}}\int_{s+1+(s+1)^{1/2}}^{(s+1+x)/2}e^{-\frac{(x-(s+1))^2}{\mu(t-s)}}(s+1)^{-11/8}(y-(s+1))^{-7/4}\, dyds \\
    & \quad +C\int_{t^{1/2}}^{x-4x^{1/2}}\int_{y>(s+1+x)/2}(t-s)^{-1/2}e^{-\frac{(x-y)^2}{\mu(t-s)}}(s+1)^{-11/8}(x-(s+1))^{-7/4}\, dyds \\
    & \quad +C\int_{x-4x^{1/2}}^{x+x^{1/2}}(s+1)^{-9/4}\, ds \\
    & \quad +C(t+1)^{-1/2}\int_{x+x^{1/2}}^{t}e^{-\frac{(s+1-x)^2}{C(t-s)}}(s+1)^{-7/4}\, ds \\
    & \eqqcolon K_{21}(x,t)+K_{22}(x,t)+K_{23}(x,t)+K_{24}(x,t).
  \end{align}
  For $K_{21}(x,t)$, we have
  \begin{equation}
    K_{21}(x,t)\leq C(t+1)^{-1/2}\int_{t^{1/2}}^{x-4x^{1/2}}e^{-\frac{(x-(s+1))^2}{\mu(t-s)}}(s+1)^{-7/4}\, ds,
  \end{equation}
  and we notice that the right-hand side is almost identical to $J_{21}(x,t)$; we omit the rest of the calculations. For $K_{22}(x,t)$, we have
  \begin{align}
    K_{22}(x,t)
    & \leq C\int_{t^{1/2}}^{x-4x^{1/2}}(s+1)^{-11/8}(x-(s+1))^{-7/4}\, ds \\
    & \leq C|x|^{-7/4}\int_{t^{1/2}}^{x/2-1}(s+1)^{-11/8}\, ds+C|x|^{-11/8}\int_{x/2-1}^{x-4x^{1/2}}(x-(s+1))^{-7/4}\, ds \\
    & \leq C|x|^{-7/4}\leq C\psi_{7/4}(x,t;0).
  \end{align}
  The term $K_{23}(x,t)$ is almost identical to $J_{22}(x,t)$; we omit the rest of the calculations. Finally, for $K_{24}(x,t)$, we have
  \begin{equation}
    K_{24}(x,t)\leq C(t+1)^{-1/2}|x|^{-7/4}\int_{x+x^{1/2}}^{t}e^{-\frac{(s+1-x)^2}{C(t+1)}}\, ds\leq C|x|^{-7/4}\leq C\psi_{7/4}(x,t;0).
  \end{equation}
  This ends the proof.
\end{proof}

\begin{proof}[Proof of Lemma~\ref{lem:LZ98_lemma3.5_modified}]
  We only consider the case of $\lambda=0$ since the other cases are similar. We also assume that $t\geq 4$ since the case when $t<4$ can be handled easily. Let
  \begin{align}
    I(x,t)
    & \coloneqq \int_{0}^{t}\int_{-\infty}^{\infty}(t-s)^{-1}(t+1-s)^{-\alpha/2}e^{-\frac{(x-y)^2}{\mu(t-s)}}(s+1)^{-\beta/2}\psi_{7/4}(y,s;0)\, dyds \\
    & =\int_{0}^{t}\int_{|y|\leq (s+1)^{1/2}}(t-s)^{-1}(t+1-s)^{-\alpha/2}e^{-\frac{(x-y)^2}{\mu(t-s)}}(s+1)^{-\beta/2}\psi_{7/4}(y,s;0)\, dyds \\
    & \quad +\int_{0}^{t}\int_{|y|>(s+1)^{1/2}}(t-s)^{-1}(t+1-s)^{-\alpha/2}e^{-\frac{(x-y)^2}{\mu(t-s)}}(s+1)^{-\beta/2}\psi_{7/4}(y,s;0)\, dyds \\
    & \eqqcolon I_1(x,t)+I_2(x,t).
  \end{align}

  For $I_1(x,t)$, we have
  \begin{align}
    I_1(x,t)
    & \leq C\int_{0}^{t/2}\int_{|y|\leq (s+1)^{1/2}}(t-s)^{-1}(t+1-s)^{-\alpha/2}e^{-\frac{(x-y)^2}{\mu(t-s)}}\Theta_{\beta+7/4}(y,s;0,\mu)\, dyds \\
    & \quad +C\int_{t/2}^{t}\int_{|y|\leq (s+1)^{1/2}}(t-s)^{-1}(t+1-s)^{-\alpha/2}e^{-\frac{(x-y)^2}{\mu(t-s)}}\Theta_{\beta+7/4}(y,s;0,\mu)\, dyds.
  \end{align}
  Then, by~\cite[Lemma~3.2]{LZ97}, we see that the first and the second term on the right-hand side are bounded by the right-hand sides of~\eqref{lem:LZ98_lemma3.5_modified:eq1} and~\eqref{lem:LZ98_lemma3.5_modified:eq2}, respectively.

  For $I_2(x,t)$, let us first consider the case of (i) $|x|\leq 2(t+1)^{1/2}$. Then
  \begin{align}
    I_2(x,t)
    & \leq C\int_{0}^{t/2}\int_{|y|>(s+1)^{1/2}}(t-s)^{-1}(t+1-s)^{-\alpha/2}(s+1)^{-\beta/2}|y|^{-7/4}\, dyds \\
    & \quad +C\int_{t/2}^{t}\int_{|y|>(s+1)^{1/2}}(t-s)^{-1}(t+1-s)^{-\alpha/2}(s+1)^{-(\beta+7/4)/2}e^{-\frac{(x-y)^2}{\mu(t-s)}}\, dyds \\
    & \leq C(t+1)^{-\alpha/2-1}\int_{0}^{t/2}(s+1)^{-(\beta+3/4)/2}\, ds+C(t+1)^{-(\beta+7/4)/2}\int_{t/2}^{t}(t-s)^{-1/2}(t+1-s)^{-\alpha/2}\, ds \\
    & \leq 
    \begin{dcases}
      C(t+1)^{-\gamma_1/2}(t+1)^{-7/8} & \text{if $\beta \neq 5/4$} \\
      C\log(t+2)(t+1)^{-\gamma_1/2}(t+1)^{-7/8} & \text{if $\beta=5/4$}
    \end{dcases} \\
    & \quad +
    \begin{dcases}
      C(t+1)^{-\gamma_2/2}(t+1)^{-7/8} & \text{if $\alpha \neq 1$}, \\
      C\log(t+2)(t+1)^{-\gamma_2/2}(t+1)^{-7/8} & \text{if $\alpha=1$}.
    \end{dcases}
  \end{align}
  This shows that the first and the second term on the right-hand side are bounded by the right-hand sides of~\eqref{lem:LZ98_lemma3.5_modified:eq1} and~\eqref{lem:LZ98_lemma3.5_modified:eq2}, respectively. Let us next consider the case of (ii) $|x|>2(t+1)^{1/2}$. We assume that $x>2(t+1)^{1/2}$ since the case of $x<-2(t+1)^{1/2}$ is similar. Let
  \begin{align}
    I_2(x,t)
    & =\int_{0}^{t}\int_{-\infty}^{-(s+1)^{1/2}}(t-s)^{-1}(t+1-s)^{-\alpha/2}e^{-\frac{(x-y)^2}{\mu(t-s)}}(s+1)^{-\beta/2}\psi_{7/4}(y,s;0)\, dyds \\
    & \quad +\int_{0}^{t}\int_{(s+1)^{1/2}}^{x/2}(t-s)^{-1}(t+1-s)^{-\alpha/2}e^{-\frac{(x-y)^2}{\mu(t-s)}}(s+1)^{-\beta/2}\psi_{7/4}(y,s;0)\, dyds \\
    & \quad +\int_{0}^{t}\int_{x/2}^{\infty}(t-s)^{-1}(t+1-s)^{-\alpha/2}e^{-\frac{(x-y)^2}{\mu(t-s)}}(s+1)^{-\beta/2}\psi_{7/4}(y,s;0)\, dyds \\
    & \eqqcolon I_{21}(x,t)+I_{22}(x,t)+I_{23}(x,t).
  \end{align}
  For $I_{21}(x,t)$ and $I_{22}(x,t)$, using the change of variable $\eta=x/(t-s)^{1/2}$, we obtain
  \begin{align}
    I_{21}(x,t)+I_{22}(x,t)
    & \leq C\int_{0}^{t}(t-s)^{-1}(t+1-s)^{-\alpha/2}e^{-\frac{x^2}{C(t-s)}}(s+1)^{-(\beta+3/4)/2}\, ds \\
    & \leq C(t+1)^{-\alpha/2-1}e^{-\frac{x^2}{C(t+1)}}\int_{0}^{t/2}(s+1)^{-(\beta+3/4)/2}\, ds \\
    & \quad +C(t+1)^{-(\beta+3/4)/2}e^{-\frac{x^2}{C(t+1)}}\int_{t/2}^{t}(t-s)^{-1}(t+1-s)^{-\alpha/2}e^{-\frac{x^2}{C(t-s)}}\, ds \\
    & \leq 
    \begin{dcases}
      C(t+1)^{-\gamma_1/2}(t+1)^{-7/8}e^{-\frac{x^2}{C(t+1)}} & \text{if $\beta \neq 5/4$} \\
      C\log(t+2)(t+1)^{-\gamma_1/2}(t+1)^{-7/8}e^{-\frac{x^2}{C(t+1)}} & \text{if $\beta=5/4$}
    \end{dcases} \\
    & \quad +C(t+1)^{-(\alpha+\beta-1)/2}(t+1)^{-7/8}e^{-\frac{x^2}{C(t+1)}}.
  \end{align}
  This shows that the first and the second term on the right-hand side are bounded by the right-hand sides of~\eqref{lem:LZ98_lemma3.5_modified:eq1} and~\eqref{lem:LZ98_lemma3.5_modified:eq2}, respectively. For $I_{23}(x,t)$, we have
  \begin{align}
    I_{23}(x,t)
    & \leq Cx^{-7/4}\int_{0}^{t}(t-s)^{-1/2}(t+1-s)^{-\alpha/2}(s+1)^{-\beta/2}\, ds \\
    & \leq C(t+1)^{-\alpha/2-1/2}x^{-7/4}\int_{0}^{t/2}(s+1)^{-\beta/2}\, ds \\
    & \quad +C(t+1)^{-\beta/2}x^{-7/4}\int_{t/2}^{t}(t-s)^{-1/2}(t+1-s)^{-\alpha/2}\, ds \\
    & \leq
    \begin{dcases}
      C(t+1)^{-(\alpha+\min(\beta,2)-1)/2}x^{-7/4} & \text{if $\beta \neq 2$} \\
      C\log(t+2)(t+1)^{-(\alpha+\min(\beta,2)-1)/2}x^{-7/4} & \text{if $\beta=2$}
    \end{dcases} \\
    & \quad +
    \begin{dcases}
      C(t+1)^{-\gamma_2/2}x^{-7/4} & \text{if $\alpha \neq 1$} \\
      C\log(t+2)(t+1)^{-\gamma_2/2}x^{-7/4} & \text{if $\alpha=1$}
    \end{dcases} \\
    & \leq C(t+1)^{-\gamma_1/2}x^{-7/4} \\
    & \quad +
    \begin{dcases}
      C(t+1)^{-\gamma_2/2}x^{-7/4} & \text{if $\alpha \neq 1$}, \\
      C\log(t+2)(t+1)^{-\gamma_2/2}x^{-7/4} & \text{if $\alpha=1$}.
    \end{dcases}
  \end{align}
  This shows that the first and the second term on the right-hand side are bounded by the right-hand sides of~\eqref{lem:LZ98_lemma3.5_modified:eq1} and~\eqref{lem:LZ98_lemma3.5_modified:eq2}, respectively. This ends the proof.
\end{proof}

\begin{proof}[Proof of Lemma~\ref{lem:LZ98_lemma3.6_modified}]
  We only consider the case of $(\lambda,\lambda')=(0,1)$ since the other cases are similar. We also assume that $t$ is sufficiently large since the statement of the lemma is otherwise easy to check. Let
  \begin{align}
    I(x,t)
    & \coloneqq \int_{0}^{t}\int_{-\infty}^{\infty}(t-s)^{-1}(t+1-s)^{-\alpha/2}e^{-\frac{(x-y)^2}{\mu(t-s)}}(s+1)^{-\beta/2}\psi_{7/4}(y,s;1)\, dyds \\
    & =\int_{0}^{t}\int_{|y-(s+1)|\leq (s+1)^{1/2}}(t-s)^{-1}(t+1-s)^{-\alpha/2}e^{-\frac{(x-y)^2}{\mu(t-s)}}(s+1)^{-\beta/2}\psi_{7/4}(y,s;1)\, dyds \\
    & \quad +\int_{0}^{t}\int_{|y-(s+1)|>(s+1)^{1/2}}(t-s)^{-1}(t+1-s)^{-\alpha/2}e^{-\frac{(x-y)^2}{\mu(t-s)}}(s+1)^{-\beta/2}\psi_{7/4}(y,s;1)\, dyds \\
    & \eqqcolon I_1(x,t)+I_2(x,t).
  \end{align}

  For $I_1(x,t)$, we have
  \begin{align}
    I_1(x,t)
    & \leq C\int_{0}^{t/2}\int_{|y-(s+1)|\leq (s+1)^{1/2}}(t-s)^{-1}(t+1-s)^{-\alpha/2}e^{-\frac{(x-y)^2}{\mu(t-s)}}\Theta_{\beta+7/4}(y,s;1,\mu)\, dyds \\
    & \quad +C\int_{t/2}^{t}\int_{|y-(s+1)|\leq (s+1)^{1/2}}(t-s)^{-1}(t+1-s)^{-\alpha/2}e^{-\frac{(x-y)^2}{\mu(t-s)}}\Theta_{\beta+7/4}(y,s;1,\mu)\, dyds \\
    & \eqqcolon I_{11}(x,t)+I_{12}(x,t).
  \end{align}
  We next show that $I_{11}(x,t)$ is bounded by the right-hand side of~\eqref{lem:LZ98_lemma3.6_modified:eq1}; we can similarly show that $I_{12}(x,t)$ is bounded by the right-hand side of~\eqref{lem:LZ98_lemma3.6_modified:eq2} (we omit the calculations for brevity).\footnote{The following calculations are almost identical to those found in the proof of~\cite[Lemma~5.2, or more specifically, Eqs.~(5.13) and (5.14) in Corollary~5.3]{Liu97}.}~First, note that by~\eqref{eq:LZ97_39}, we have
  \begin{equation}
    \label{eq:I11_integrated}
    I_{11}(x,t)\leq C(t+1)^{-1/2}\int_{0}^{t/2}(t-s)^{-1/2}(t+1-s)^{-\alpha/2}(s+1)^{-(\beta+3/4)/2}e^{-\frac{(x-(s+1))^2}{\mu(t+1)}}\, ds.
  \end{equation}
  Let us first start with the case of (i) $x<(t+1)^{1/2}$. In this case, by~\eqref{eq:I11_integrated}, we have
  \begin{align}
    I_{11}(x,t)
    & \leq C(t+1)^{-1/2}e^{-\frac{x^2}{\mu(t+1)}}\int_{0}^{t/2}(t-s)^{-1/2}(t+1-s)^{-\alpha/2}(s+1)^{-(\beta+3/4)/2}e^{-\frac{(s+1)^2}{C(t+1)}}\, ds \\
    & \leq C(t+1)^{-\alpha/2-1}e^{-\frac{x^2}{\mu(t+1)}}\int_{0}^{t^{1/2}}(s+1)^{-(\beta+3/4)/2}\, ds \\
    & \quad +C(t+1)^{-\alpha/2-1}e^{-\frac{x^2}{\mu(t+1)}}\int_{t^{1/2}}^{t/2}(s+1)^{-(\beta+3/4)/2}e^{-\frac{(s+1)^2}{C(t+1)}}\, ds \\
    & \leq C(t+1)^{-\alpha/2-1}e^{-\frac{x^2}{\mu(t+1)}}\int_{0}^{t^{1/2}}(s+1)^{-(\beta+3/4)/2}\, ds \\
    & \quad +C(t+1)^{-\gamma_1/2}\Theta_{7/4}(x,t;0,\mu)\int_{2^{1/2}}^{\infty}\eta^{-(\beta+3/4)/2}e^{-\frac{\eta^2}{C}}\, d\eta \\
    & \leq
    \begin{dcases}
      C(t+1)^{-\gamma_1/2}\Theta_{7/4}(x,t;0,\mu) & \text{if $\beta \neq 5/4$}, \\
      C\log(t+2)(t+1)^{-\gamma_1/2}\Theta_{7/4}(x,t;0,\mu) & \text{if $\beta=5/4$}.
    \end{dcases}
  \end{align}
  This shows that $I_{11}(x,t)$ is bounded by the right-hand side of~\eqref{lem:LZ98_lemma3.6_modified:eq1}. Next, we consider the case of (ii)~$(t+1)^{1/2}\leq x\leq t-(t+1)^{1/2}$. By~\eqref{eq:I11_integrated}, we have
  \begin{align}
    I_{11}(x,t)
    & \leq C(t+1)^{-\alpha/2-1}e^{-\frac{x^2}{C(t+1)}}\int_{0}^{x/2}(s+1)^{-(\beta+3/4)/2}\, ds \\
    & \quad +C(t+1)^{-1/2}x^{-\beta/2-3/8}(t-x)^{-\alpha/2-1/2}\int_{x/2}^{t-(t-x)/2}e^{-\frac{(x-(s+1))^2}{\mu(t+1)}}\, ds \\
    & \leq 
    \begin{dcases}
      C(t+1)^{-\gamma_1/2}\Theta_{7/4}(x,t;0,\nu^*) & \text{if $\beta \neq 5/4$} \\
      C\log(t+2)(t+1)^{-\gamma_1/2}\Theta_{7/4}(x,t;0,\nu^*) & \text{if $\beta=5/4$}
    \end{dcases} \\
    & \quad +Cx^{-\beta/2-3/8}(t-x)^{-\alpha/2-1/2}.
  \end{align}
  Here, we used the bound $xe^{-x^2/C(t+1)}\leq C(t+1)^{1/2}e^{-x^2/C(t+1)}$. This shows that $I_{11}(x,t)$ is bounded by the right-hand side of~\eqref{lem:LZ98_lemma3.6_modified:eq1}. Finally, we consider the case of (iii) $x>t-(t+1)^{1/2}$. In this case, by~\eqref{eq:I11_integrated}, we have
  \begin{equation}
    I_{11}(x,t)\leq Ce^{-\frac{t}{C}}e^{-\frac{x^2}{C(t+1)}}\leq Ce^{-\frac{t}{C}}\Theta_{7/4}(x,t;0,\nu^*).
  \end{equation}
  This shows that $I_{11}(x,t)$ is bounded by the right-hand side of~\eqref{lem:LZ98_lemma3.6_modified:eq1}.

  Let us next consider $I_2(x,t)$. We first treat the case of (i) $x<-K(t+1)^{1/2}$. Here, $K$ is a sufficiently large constant. In this case, we have
  \begin{align}
    I_2(x,t)
    & \leq C\int_{0}^{t}\int_{-\infty}^{x/2}(t-s)^{-1}(t+1-s)^{-\alpha/2}e^{-\frac{(x-y)^2}{\mu(t-s)}}(s+1)^{-\beta/2}(s+1-x/2)^{-7/4}\, dyds \\
    & \quad +C\int_{0}^{t}\int_{x/2}^{s+1-(s+1)^{1/2}}(t-s)^{-1}(t+1-s)^{-\alpha/2}e^{-\frac{x^2}{C(t-s)}}(s+1)^{-\beta/2}(s+1-y)^{-7/4}\, dyds \\
    & \quad +C\int_{0}^{t}\int_{s+1+(s+1)^{1/2}}^{\infty}(t-s)^{-1}(t+1-s)^{-\alpha/2}e^{-\frac{x^2}{\mu(t-s)}}(s+1)^{-\beta/2}(y-(s+1))^{-7/4}\, dyds \\
    & \leq C|x|^{-7/4}\int_{0}^{t}(t-s)^{-1/2}(t+1-s)^{-\alpha/2}(s+1)^{-\beta/2}\, ds \\
    & \quad +C\int_{0}^{t}(t-s)^{-1}(t+1-s)^{-\alpha/2}e^{-\frac{x^2}{C(t-s)}}(s+1)^{-(\beta+3/4)/2}\, ds.
  \end{align}
  The integrals on the right-hand side are already treated in the proof of Lemma~\ref{lem:LZ98_lemma3.5_modified}; see the calculations for $I_{21}(x,t)+I_{22}(x,t)$ and $I_{23}(x,t)$. Consequently, it follows that $I_2(x,t)$ is bounded by the right-hand sides of~\eqref{lem:LZ98_lemma3.6_modified:eq1} or~\eqref{lem:LZ98_lemma3.6_modified:eq2}, respectively, when the domain of temporal integration is restricted to $[0,t/2]$ or $[t/2,t]$.

  We next consider the case of (ii) $|x|\leq K(t+1)^{1/2}$. In this case, we have
  \begin{align}
    I_2(x,t)
    & \leq C\int_{0}^{t/2}(t-s)^{-1}(t+1-s)^{-\alpha/2}(s+1)^{-(\beta+3/4)/2}\, ds \\
    & \quad +C\int_{t/2}^{t}(t-s)^{-1/2}(t+1-s)^{-\alpha/2}(s+1)^{-(\beta+7/4)/2}\, ds \\
    & \leq C(t+1)^{-\alpha/2-1}\int_{0}^{t/2}(s+1)^{-(\beta+3/4)/2}\, ds+C(t+1)^{-(\beta+7/4)/2}\int_{t/2}^{t}(t-s)^{-1/2}(t+1-s)^{-\alpha/2}\, ds.
  \end{align}
  Again, the integrals on the right-hand side are already treated in the proof of Lemma~\ref{lem:LZ98_lemma3.5_modified}; see the calculations for $I_2(x,t)$ for Case (i). Consequently, it follows that $I_2(x,t)$ is bounded by the right-hand sides of~\eqref{lem:LZ98_lemma3.6_modified:eq1} or~\eqref{lem:LZ98_lemma3.6_modified:eq2}, respectively, when the domain of temporal integration is restricted to $[0,t/2]$ or $[t/2,t]$.

  We next consider the case of (iii) $K(t+1)^{1/2}<x<t+1-K(t+1)^{1/2}$. This case is the lengthiest of all the cases. Let
  \begin{align}
    I_2(x,t)
    & \leq C\int_{0}^{t}\int_{y<s+1-(s+1)^{1/2}}(t-s)^{-1}(t+1-s)^{-\alpha/2}e^{-\frac{(x-y)^2}{\mu(t-s)}}(s+1)^{-\beta/2}|y-(s+1)|^{-7/4}\, dyds \\
    & \quad +C\int_{0}^{t}\int_{y>s+1+(s+1)^{1/2}}(t-s)^{-1}(t+1-s)^{-\alpha/2}e^{-\frac{(x-y)^2}{\mu(t-s)}}(s+1)^{-\beta/2}|y-(s+1)|^{-7/4}\, dyds \\
    & \eqqcolon I_{21}(x,t)+I_{22}(x,t).
  \end{align}
  We first consider $I_{21}(x,t)$. Let
  \begin{align}
    I_{21}(x,t)
    & \leq C\int_{0}^{x-x^{1/2}}(t-s)^{-1}(t+1-s)^{-\alpha/2}e^{-\frac{(x-(s+1))^2}{\mu(t-s)}}(s+1)^{-(\beta+3/4)/2}\, ds \\
    & \quad +C\int_{x-x^{1/2}}^{x+4x^{1/2}}(t-s)^{-1/2}(t+1-s)^{-\alpha/2}(s+1)^{-(\beta+7/4)/2}\, ds \\
    & \quad +C\int_{x+4x^{1/2}}^{t}\int_{y\leq (s+1+x)/2}(t-s)^{-1}(t+1-s)^{-\alpha/2}e^{-\frac{(x-y)^2}{\mu(t-s)}}(s+1)^{-\beta/2}(s+1-x)^{-7/4}\, dyds \\
    & \quad +C\int_{x+4x^{1/2}}^{t}\int_{(s+1+x)/2}^{s+1-(s+1)^{1/2}}(t-s)^{-1}(t+1-s)^{-\alpha/2}e^{-\frac{(s+1-x)^2}{C(t-s)}}(s+1)^{-\beta/2}(s+1-y)^{-7/4}\, dyds \\
    & \eqqcolon I_{211}(x,t)+I_{212}(x,t)+I_{213}(x,t)+I_{214}(x,t).
  \end{align}
  For $I_{211}(x,t)$, using the change of variable $\eta=(x-(s+1))/(t-s)^{1/2}$,\footnote{Use also the inequality $2(t-s)-(x-(s+1))\geq t-s$.}~we obtain
  \begin{align}
    I_{211}(x,t)
    & \leq C\int_{0}^{x/2}(t-x/2)^{-\alpha/2-1}e^{-\frac{x^2}{C(t-s)}}(s+1)^{-(\beta+3/4)/2}\, ds \\
    & \quad +C\int_{x/2}^{x-x^{1/2}}(t-x)^{-\alpha/2-1/2}(t-s)^{-1/2}e^{-\frac{(x-(s+1))^2}{\mu(t-s)}}x^{-(\beta+3/4)/2}\, ds \\
    & \leq C(t+1)^{-\alpha/2-1}e^{-\frac{x^2}{C(t+1)}}\int_{0}^{t/2}(s+1)^{-(\beta+3/4)/2}\, ds \\
    & \quad +Cx^{-\beta/2-3/8}(t-x)^{-\alpha/2-1/2} \\
    & \leq 
    \begin{dcases}
      C(t+1)^{-\gamma_1/2}\Theta_{7/4}(x,t;0,\nu^*) & \text{if $\beta \neq 5/4$} \\
      C\log(t+2)(t+1)^{-\gamma_1/2}\Theta_{7/4}(x,t;0,\nu^*) & \text{if $\beta=5/4$}
    \end{dcases} \\
    & \quad +Cx^{-\beta/2-3/8}(t-x)^{-\alpha/2-1/2}.
  \end{align}
  Note that the first term on the right-hand side does not appear when the domain of temporal integration is restricted to $[t/2,t]$. Hence, it follows that $I_{211}(x,t)$ is bounded by the right-hand sides of~\eqref{lem:LZ98_lemma3.6_modified:eq1} or~\eqref{lem:LZ98_lemma3.6_modified:eq2}, respectively, when the domain of temporal integration is restricted to $[0,t/2]$ or $[t/2,t]$. Next, for $I_{212}(x,t)$, we have
  \begin{equation}
    I_{212}(x,t)\leq C\int_{x-x^{1/2}}^{x+4x^{1/2}}(t-x-4x^{1/2})^{-\alpha/2-1/2}(x-x^{1/2}+1)^{-(\beta+7/4)/2}\, ds\leq Cx^{-\beta/2-3/8}(t-x)^{-\alpha/2-1/2}.
  \end{equation}
  This shows that $I_{212}(x,t)$ is bounded by the right-hand sides of~\eqref{lem:LZ98_lemma3.6_modified:eq1} and~\eqref{lem:LZ98_lemma3.6_modified:eq2}. We next consider $I_{213}(x,t)$. Let $\varepsilon$ be a sufficiently small positive number. Then we have
  \begin{align}
    I_{213}(x,t)
    & \leq C\int_{x+4x^{1/2}}^{t}(t-s)^{-1/2}(t+1-s)^{-\alpha/2}x^{-\beta/2}(s+1-x)^{-7/4}\, ds \\
    & \leq C\int_{x+4x^{1/2}}^{t-\varepsilon(t-x)}(t-x)^{-\alpha/2-1/2}x^{-\beta/2}(s+1-x)^{-7/4}\, ds \\
    & \quad +C\int_{t-\varepsilon(t-x)}^{t}(t-s)^{-1/2}(t+1-s)^{-\alpha/2}x^{-\beta/2}(t-x)^{-7/4}\, ds.
  \end{align}
  The first term on the right-hand side is bounded by $Cx^{-\beta/2-3/8}(t-x)^{-\alpha/2-1/2}$. For the second term, let us first assume that $\alpha \neq 1$. Then, noting that $(t-x)^{-1}\leq Ct^{-1/2}\leq Cx^{-1/2}$, we have
  \begin{align}
    I_{213}^{(2)}(x,t)
    & \coloneqq \int_{t-\varepsilon(t-x)}^{t}(t-s)^{-1/2}(t+1-s)^{-\alpha/2}x^{-\beta/2}(t-x)^{-7/4}\, ds \\
    & \leq Cx^{-\beta/2}(t-x)^{-\min(\alpha,1)/2-5/4}\leq Cx^{-\beta/2-3/8}(t-x)^{-\min(\alpha,1)/2-1/2}.
  \end{align}
  On the other hand, when $\alpha=1$ and $x\leq t/2$, we have
  \begin{align}
    I_{213}^{(2)}(x,t)
    & \leq C\log(t-x)x^{-\beta/2}(t-x)^{-7/4} \\
    & \leq C\log(t+2)x^{-\beta/2}(t-x)^{-7/4+\beta/2}(t-x)^{-\beta/2} \\
    & \leq C\log(t+2)(t+1)^{-\beta/2}x^{-7/4}\leq C\log(t+2)(t+1)^{-\gamma_2/2}\psi_{7/4}(x,t;0),
  \end{align}
  and when $\alpha=1$ and $x>t/2$,
  \begin{align}
    I_{213}^{(2)}(x,t)
    & \leq C\log(t-x)x^{-\beta/2}(t-x)^{-7/4} \\
    & \leq C\log(t+2)(t+1)^{-\beta/2}(t-x)^{-7/4}\leq C\log(t+2)(t+1)^{-\gamma_2/2}\psi_{7/4}(x,t;1).
  \end{align}
  Note that we used the assumption $\beta \leq 7/2$ here. Now, noting that $t-\varepsilon(t-x)\geq t/2$, it follows that $I_{213}(x,t)$ is bounded by the right-hand sides of~\eqref{lem:LZ98_lemma3.6_modified:eq1} or~\eqref{lem:LZ98_lemma3.6_modified:eq2}, respectively, when the domain of temporal integration is restricted to $[0,t/2]$ or $[t/2,t]$. Next, we consider $I_{214}(x,t)$. Using again the change of variable $\eta=(x-(s+1))/(t-s)^{1/2}$, we obtain
  \begin{align}
    I_{214}(x,t)
    & \leq C\int_{x+4x^{1/2}}^{t}(t-s)^{-1}(t+1-s)^{-\alpha/2}e^{-\frac{(s+1-x)^2}{C(t-s)}}(s+1)^{-(\beta+3/4)/2}\, ds \\
    & \leq Cx^{-\beta/2-3/8}(t-x)^{-\alpha/2-1/2}\int_{x+4x^{1/2}}^{t-\varepsilon(t-x)}(t-s)^{-1/2}e^{-\frac{(s+1-x)^2}{C(t-s)}}\, ds \\
    & \quad +Cx^{-\beta/2-3/8}\int_{t-\varepsilon(t-x)}^{t}(s+1-x)^{-2}(t+1-s)^{-\alpha/2}\, ds \\
    & \leq Cx^{-\beta/2-3/8}(t-x)^{-\alpha/2-1/2}+Cx^{-\beta/2-3/8}(t-x)^{-1}.
  \end{align}
  Since $t-\varepsilon(t-x)\geq t/2$, it follows that $I_{214}(x,t)$ is bounded by the right-hand sides of~\eqref{lem:LZ98_lemma3.6_modified:eq1} or~\eqref{lem:LZ98_lemma3.6_modified:eq2}, respectively, when the domain of temporal integration is restricted to $[0,t/2]$ or $[t/2,t]$.

  Next, we consider $I_{22}(x,t)$. Let
  \begin{align}
    I_{22}(x,t)
    & \leq C\int_{0}^{x-4x^{1/2}}\int_{s+1+(s+1)^{1/2}}^{(s+1+x)/2}(t-s)^{-1}(t+1-s)^{-\alpha/2}e^{-\frac{(x-(s+1))^2}{C(t-s)}}(s+1)^{-\beta/2}(y-(s+1))^{-7/4}\, dyds \\
    & \quad +C\int_{0}^{x-4x^{1/2}}\int_{y>(s+1+x)/2}(t-s)^{-1}(t+1-s)^{-\alpha/2}e^{-\frac{(x-y)^2}{\mu(t-s)}}(s+1)^{-\beta/2}(x-(s+1))^{-7/4}\, dyds \\
    & \quad +C\int_{x-4x^{1/2}}^{x+x^{1/2}}(t-s)^{-1/2}(t+1-s)^{-\alpha/2}(s+1)^{-(\beta+7/4)/2}\, ds \\
    & \quad +C\int_{x+x^{1/2}}^{t}(t-s)^{-1}(t+1-s)^{-\alpha/2}e^{-\frac{(s+1-x)^2}{\mu(t-s)}}(s+1)^{-(\beta+3/4)/2}\, ds \\
    & \eqqcolon I_{221}(x,t)+I_{222}(x,t)+I_{223}(x,t)+I_{224}(x,t).
  \end{align}
  For $I_{221}(x,t)$, we have
  \begin{equation}
    I_{221}(x,t)\leq C\int_{0}^{x-4x^{1/2}}(t-s)^{-1}(t+1-s)^{-\alpha/2}e^{-\frac{(x-(s+1))^2}{C(t-s)}}(s+1)^{-(\beta+3/4)/2}\, ds.
  \end{equation}
  The right-hand side is almost identical to $I_{211}(x,t)$, so we omit the rest of the calculations. For $I_{222}(x,t)$, we have
  \begin{align}
    I_{222}(x,t)
    & \leq C\int_{0}^{x-4x^{1/2}}(t-s)^{-1/2}(t+1-s)^{-\alpha/2}(s+1)^{-\beta/2}(x-(s+1))^{-7/4}\, ds \\
    & \leq C(t-x)^{-\alpha/2-1/2}\left[ x^{-7/4}\int_{0}^{x/2}(s+1)^{-\beta/2}\, ds+x^{-\beta/2}\int_{x/2}^{x-4x^{1/2}}(x-(s+1))^{-7/4}\, ds \right] \\
    & \leq Cx^{-\min(\beta,2)/2-3/4}(t-x)^{-\alpha/2-1/2}+Cx^{-\beta/2-3/8}(t-x)^{-\alpha/2-1/2} \\
    & \leq Cx^{-\min(\beta,11/4)/2-3/8}(t-x)^{-\alpha/2-1/2}.
  \end{align}
  Note that we used the assumption $\beta \neq 2$ here. Since $x/2 \leq t/2$, it follows that $I_{222}(x,t)$ is bounded by the right-hand sides of~\eqref{lem:LZ98_lemma3.6_modified:eq1} or~\eqref{lem:LZ98_lemma3.6_modified:eq2}, respectively, when the domain of temporal integration is restricted to $[0,t/2]$ or $[t/2,t]$. Next, note that $I_{223}(x,t)$ is almost identical to $I_{212}(x,t)$, so we omit the necessary calculations. Finally, $I_{224}(x,t)$ can be treated almost in the same way as $I_{214}(x,t)$, and we omit the rest of the calculations.

  We next consider the case of (iv) $|x-(t+1)|\leq K(t+1)^{1/2}$. Let
  \begin{align}
    & I_2(x,t) \\
    & \leq C\int_{0}^{t/2}\left( \int_{y<s+1-(s+1)^{1/2}}+\int_{s+1+(s+1)^{1/2}}^{(s+1+x)/2} \right) (t-s)^{-1}(t+1-s)^{-\alpha/2}e^{-\frac{(x-(s+1))^2}{\mu(t-s)}}(s+1)^{-\beta/2}|y-(s+1)|^{-7/4}\, dyds \\
    & \quad +C\int_{0}^{t/2}\int_{y>(s+1+x)/2}(t-s)^{-1}(t+1-s)^{-\alpha/2}e^{-\frac{(x-y)^2}{\mu(t-s)}}(s+1)^{-\beta/2}(x-(s+1))^{-7/4}\, dyds \\
    & \quad +C(t+1)^{-(\beta+7/4)/2}\int_{t/2}^{t}(t-s)^{-1/2}(t+1-s)^{-\alpha/2}\, ds \\
    & \eqqcolon I_{21}(x,t)+I_{22}(x,t)+I_{23}(x,t).
  \end{align}
  For $I_{21}(x,t)$, we have
  \begin{equation}
    I_{21}(x,t)\leq Ce^{-\frac{t}{C}}\int_{0}^{t/2}(t-s)^{-1}(t+1-s)^{-\alpha/2}(s+1)^{-(\beta+3/4)/2}\, ds\leq Ce^{-\frac{t}{C}}.
  \end{equation}
  Hence, $I_{21}(x,t)$ is bounded by the right-hand side of~\eqref{lem:LZ98_lemma3.6_modified:eq1}. For $I_{22}(x,t)$, we have
  \begin{align}
    I_{22}(x,t)
    & \leq C\int_{0}^{t/2}(t-s)^{-1/2}(t+1-s)^{-\alpha/2}(s+1)^{-\beta/2}(x-(s+1))^{-7/4}\, ds \\
    & \leq C(t+1)^{-\gamma_{1}'/2}(t+1)^{-7/4},
  \end{align}
  and we observe that $I_{22}(x,t)$ is bounded by the right-hand side of~\eqref{lem:LZ98_lemma3.6_modified:eq1}. Finally, for $I_{23}(x,t)$, by similar calculations for the bound of $I_2(x,t)$ in Case (ii), we can show that $I_{23}(x,t)$ is bounded by the right-hand side of~\eqref{lem:LZ98_lemma3.6_modified:eq2}.

  Finally, we consider the case of (v) $x>t+1+K(t+1)^{1/2}$. Using the change of variable $\eta=(x-t)/(t-s)^{1/2}$, we obtain
  \begin{align}
    & I_2(x,t) \\
    & \leq C\int_{0}^{t}\left( \int_{y<s+1-(s+1)^{1/2}}+\int_{s+1+(s+1)^{1/2}}^{(s+1+x)/2} \right) (t-s)^{-1}(t+1-s)^{-\alpha/2}e^{-\frac{(x-(s+1))^2}{C(t-s)}}(s+1)^{-\beta/2}|y-(s+1)|^{-7/4}\, dyds \\
    & \quad +C\int_{0}^{t}\int_{y>(s+1+x)/2}(t-s)^{-1}(t+1-s)^{-\alpha/2}e^{-\frac{(x-y)^2}{\mu(t-s)}}(s+1)^{-\beta/2}(x-(s+1))^{-7/4}\, dyds \\
    & \leq C\int_{0}^{t}(t-s)^{-1}(t+1-s)^{-\alpha/2}e^{-\frac{(x-(s+1))^2}{C(t-s)}}(s+1)^{-(\beta+3/4)/2}\, ds \\
    & \quad +C\int_{0}^{t}(t-s)^{-1/2}(t+1-s)^{-\alpha/2}(s+1)^{-\beta/2}(x-(s+1))^{-7/4}\, ds \\
    & \leq Ce^{-\frac{t}{C}}e^{-\frac{(x-(t+1))^2}{C(t+1)}}\int_{0}^{t/2}(s+1)^{-(\beta+3/4)/2}\, ds \\
    & \quad +C(t+1)^{-\beta/2-3/8}e^{-\frac{(x-(t+1))^2}{C(t+1)}}\int_{t/2}^{t}(t-s)^{-1}(t+1-s)^{-\alpha/2}e^{-\frac{(x-t)^2}{C(t-s)}}\, ds \\
    & \quad +C(x-t)^{-7/4}\int_{0}^{t}(t-s)^{-1/2}(t+1-s)^{-\alpha/2}(s+1)^{-\beta/2}\, ds \\
    & \leq C(t+1)^{-(\alpha+\beta-1)/2}(t+1)^{-7/8}e^{-\frac{(x-(t+1))^2}{C(t+1)}}+C(x-t)^{-7/4}\int_{0}^{t}(t-s)^{-1/2}(t+1-s)^{-\alpha/2}(s+1)^{-\beta/2}\, ds.
  \end{align}
  The second term on the right-hand side is bounded by
  \begin{align}
    & C(t+1)^{-\alpha/2-1/2}(x-t)^{-7/4}\int_{0}^{t/2}(s+1)^{-\beta/2}\, ds \\
    & +C(t+1)^{-\beta/2}(x-t)^{-7/4}\int_{t/2}^{t}(t-s)^{-1/2}(t+1-s)^{-\alpha/2}\, ds \\
    & \leq C(t+1)^{-\gamma_{1}'/2}(x-t)^{-7/4}+
    \begin{dcases}
      C(t+1)^{-\gamma_2/2}(x-t)^{-7/4} & \text{if $\alpha \neq 1$}, \\
      C\log(t+2)(t+1)^{-\gamma_2/2}(x-t)^{-7/4} & \text{if $\alpha=1$}.
    \end{dcases}
  \end{align}
  Hence, it follows that $I_2(x,t)$ is bounded by the right-hand sides of~\eqref{lem:LZ98_lemma3.6_modified:eq1} or~\eqref{lem:LZ98_lemma3.6_modified:eq2}, respectively, when the domain of temporal integration is restricted to $[0,t/2]$ or $[t/2,t]$. This ends the proof.
\end{proof}

\begin{proof}[Proof of Lemma~\ref{lem:boundary1}]
  Denote by $I(x,t)$ the integral appearing in the statement of the lemma. First, let us consider the case of (i) $|x-\lambda(t+1)|\leq (t+1)^{1/2}$. We have
  \begin{equation}
    I(x,t)\leq C(t+1)^{-1/2}\int_{t^{1/2}}^{t/2}(s+1)^{-21/8}\, ds+C(t+1)^{-21/8}\int_{t/2}^{t}(t-s)^{-1/2}\, ds\leq C(t+1)^{-21/16}\leq C\bar{\psi}(x,t;\lambda).
  \end{equation}
  Next, when (ii) $|x-\lambda(t+1)|>(t+1)^{1/2}$, let
  \begin{equation}
    A_1 \coloneqq \{ t^{1/2}\leq s\leq t \mid |\lambda|s\leq |x-\lambda t|/2 \}, \quad A_2 \coloneqq \{ t^{1/2}\leq s\leq t \mid |\lambda|s>|x-\lambda t|/2 \}.
  \end{equation}
  If $s\in A_1$, we have
  \begin{equation}
    |x-\lambda(t-s)|\geq |x-\lambda t|/2,
  \end{equation}
  and if $s\in A_2$, we have\footnote{Note that $A_2=\emptyset$ when $\lambda=0$.}
  \begin{equation}
    (s+1)^{-21/8}\leq C(s+1)^{-7/8}|x-\lambda t|^{-7/4}.
  \end{equation}
  Therefore,
  \begin{align}
    I(x,t)
    & \leq Ce^{-\frac{(x-\lambda t)^2}{Ct}}\int_{t^{1/2}}^{t}(t-s)^{-1/2}(s+1)^{-21/8}\, ds+C|x-\lambda t|^{-7/4}\int_{t^{1/2}}^{t}(t-s)^{-1/2}(s+1)^{-7/8}\, ds \\
    & \leq C(t+1)^{-21/16}e^{-\frac{(x-\lambda t)^2}{Ct}}+C(t+1)^{-3/8}|x-\lambda t|^{-7/4}\leq C\bar{\psi}(x,t;\lambda).
  \end{align}
  This ends the proof.
\end{proof}

\begin{proof}[Proof of Lemma~\ref{lem:boundary2}]
  Let
  \begin{align}
    \int_{t^{1/2}}^{t}\partial_x \left\{ (t-s)^{-1/2}e^{-\frac{(x-\lambda(t-s))^2}{\mu(t-s)}} \right\} f(s)\, ds
    & =\int_{t^{1/2}}^{t-1}\partial_x \left\{ (t-s)^{-1/2}e^{-\frac{(x-\lambda(t-s))^2}{\mu(t-s)}} \right\} f(s)\, ds \\
    & \quad +\int_{t-1}^{t}\partial_x \left\{ (t-s)^{-1/2}e^{-\frac{(x-\lambda(t-s))^2}{\mu(t-s)}} \right\} f(s)\, ds \\
    & \eqqcolon I_1(x,t)+I_2(x,t).
  \end{align}
  The first term $I_1(x,t)$ can be treated by calculations similar to those in the proof of Lemma~\ref{lem:boundary1}. For $I_2(x,t)$, we use a differential equation technique~(see~\cite[p.~410--411]{Koike21}). Observe that for $x\neq 0$, we have
  \begin{align}
    \frac{\mu}{4}\partial_x I_2(x,t)
    & =\int_{t-1}^{t}(\partial_t \Theta_1+\lambda \partial_x \Theta_1)(x,t-s-1;\lambda,\mu)f(s)\, ds \\
    & =\Theta_1(x,0;\lambda,\mu)f(t-1)+\int_{t-1}^{t}\Theta_1(x,t-s-1;\lambda,\mu)\partial_t f(s)\, ds+\lambda I_2(x,t) \\
    & \eqqcolon \frac{\mu}{4}w(x,t)+\lambda I_2(x,t).
  \end{align}
  It is easy to see that
  \begin{equation}
    \label{eq:w_bound}
    |w(x,t)|\leq C(t+1)^{-7/4}e^{-\frac{x^2}{C}}.
  \end{equation}
  Without loss of generality, we may assume that $\lambda>0$. Then, since $I_2$ solves the ordinary differential equation $\partial_x I_2=w+(4\lambda/\mu)I_2$ and vanishes as $|x|\to \infty$, we obtain
  \begin{equation}
    \label{eq:u_rep}
    I_2(x,t)=
    \begin{dcases}
      -\int_{x}^{\infty}e^{\frac{4\lambda}{\mu}(x-y)}w(y,t)\, dy & (x>0), \\
      \int_{-\infty}^{x}e^{\frac{4\lambda}{\mu}(x-y)}w(y,t)\, dy & (x<0).
    \end{dcases}
  \end{equation}
  From this representation and~\eqref{eq:w_bound}, it follows that
  \begin{equation}
    |I_2(x,t)|\leq C(t+1)^{-7/4}e^{-\frac{x^2}{C}}\leq C\bar{\psi}(x,t;\lambda).
  \end{equation}
  This ends the proof.
\end{proof}

\section{Pointwise estimates of products of certain functions}
\label{AppendixD}
In this appendix, we gather pointwise estimates of products of certain functions for the reader's convenience. First, for $\lambda \in \mathbb{R}$ and $\alpha,\mu>0$, we set
\begin{equation}
  \Theta_{\alpha}(x,t;\lambda,\mu)\coloneqq (t+1)^{-\alpha/2}e^{-\frac{(x-\lambda(t+1))^2}{\mu(t+1)}}, \quad \psi_{\alpha}(x,t;\lambda)\coloneqq [(x-\lambda(t+1))^2+(t+1)]^{-\alpha/2},
\end{equation}
and
\begin{equation}
  \bar{\psi}(x,t;\lambda)\coloneqq [|x-\lambda(t+1)|^7+(t+1)^5]^{-1/4}, \quad \tilde{\psi}(x,t;\lambda)\coloneqq [|x-\lambda(t+1)|^3+(t+1)^2]^{-1/2}.
\end{equation}
We note that some of these functions already appeared in the paper and the definitions above are consistent with them. Note also that $\Phi_i$ define by~\eqref{def:Phi} can be written as $\Phi_i(x,t)=\psi_{3/2}(x,t;\lambda_i)+\tilde{\psi}(x,t;\lambda_{i'})$, where $i'=3-i$. We then have the following lemma.

\begin{lem}
  \label{lem:product}
  Let $\lambda \neq \lambda'$ and $\alpha,\beta,\mu>0$. In addition, let $M$ be an arbitrary positive number. Then, for functions $f$, $g$, and $h$ listed in Table~\ref{table2}, we have
  \begin{equation}
    (fg)(x,t)\leq Ch(x,t)
  \end{equation}
  for some $C>0$.
\end{lem}

\begin{table}[htbp]
  \caption{Pointwise estimates of products of certain functions: $(fg)(x,t)\leq Ch(x,t)$}
  \centering
  \begin{tabular}{l l | l}
    \hline
    \addstackgap{$f(x,t)$} & \addstackgap{$g(x,t)$} & \addstackgap{$h(x,t)$} \\ [0.5ex]
    \hline
    \addstackgap{$\Theta_{\alpha}(x,t;\lambda,\mu)$} & \addstackgap{$1$} & \addstackgap{$\psi_{\alpha}(x,t;\lambda)$} \\ [0.5ex]

    \addstackgap{$\bar{\psi}(x,t;\lambda)$} & \addstackgap{$1$} & \addstackgap{$\psi_{7/4}(x,t;\lambda)$} \\ [0.5ex]

    \addstackgap{$\tilde{\psi}(x,t;\lambda)$} & \addstackgap{$1$} & \addstackgap{$\psi_{3/2}(x,t;\lambda)$} \\ [0.5ex]

    \addstackgap{$\Theta_{\alpha}(x,t;\lambda,\mu)$} & \addstackgap{$\Theta_{\beta}(x,t;\lambda',\mu)$} & \addstackgap{$\Theta_M(x,t;\lambda,2\mu)$} \\ [0.5ex]

    \addstackgap{$\Theta_{\alpha}(x,t;\lambda,\mu)$} & \addstackgap{$\psi_{\beta}(x,t;\lambda)$} & \addstackgap{$\Theta_{\alpha+\beta}(x,t;\lambda,\mu)$} \\ [0.5ex]

    \addstackgap{$\Theta_{\alpha}(x,t;\lambda,\mu)$} & \addstackgap{$\psi_{\beta}(x,t;\lambda')$} & \addstackgap{$\Theta_{\alpha+2\beta}(x,t;\lambda,2\mu)$} \\ [0.5ex]

    \addstackgap{$\Theta_{\alpha}(x,t;\lambda,\mu)$} & \addstackgap{$\bar{\psi}(x,t;\lambda)$} & \addstackgap{$\Theta_{\alpha+5/2}(x,t;\lambda,\mu)$} \\ [0.5ex]

    \addstackgap{$\Theta_{\alpha}(x,t;\lambda,\mu)$} & \addstackgap{$\tilde{\psi}(x,t;\lambda)$} & \addstackgap{$\Theta_{\alpha+2}(x,t;\lambda,\mu)$} \\ [0.5ex]

    \addstackgap{$\psi_{\alpha}(x,t;\lambda)$} & \addstackgap{$\psi_{\beta}(x,t;\lambda)$} & \addstackgap{$(t+1)^{-(\alpha+\beta)/2+7/8}\psi_{7/4}(x,t;\lambda)$ \, (\text{if $\alpha+\beta \geq 7/4$})} \\ [0.5ex]

    \addstackgap{$\psi_{\alpha}(x,t;\lambda)$} & \addstackgap{$\psi_{\beta}(x,t;\lambda')$} & \addstackgap{$(t+1)^{-(\alpha+\beta)+7/4}[\psi_{7/4}(x,t;\lambda)+\psi_{7/4}(x,t;\lambda')]$} \\ [0.5ex]
    \addstackgap{} & \addstackgap{} & \addstackgap{\phantom{$(t+1)^{-(\alpha+\beta)/2+7/8}\psi_{7/4}(x,t;\lambda)$} \, (\text{if $\max(\alpha,\beta)\leq 7/4$ and $\alpha+\beta \geq 7/4$})} \\ [0.5ex]

    \addstackgap{$\psi_{\alpha}(x,t;\lambda)$} & \addstackgap{$\bar{\psi}(x,t;\lambda)$} & \addstackgap{$(t+1)^{-5\alpha/7}\psi_{7/4}(x,t;\lambda)$} \\ [0.5ex]

    \addstackgap{$\psi_{\alpha}(x,t;\lambda)$} & \addstackgap{$\tilde{\psi}(x,t;\lambda)$} & \addstackgap{$(t+1)^{-2\alpha/3+1/6}\psi_{7/4}(x,t;\lambda)$ \, (\text{if $\alpha \geq 1/4$})} \\ [1.0ex]
    \hline
  \end{tabular}
  \label{table2}
\end{table}

\begin{proof}
  For brevity, we only consider the case of $f(x,t)=\psi_{\alpha}(x,t;\lambda)$, $g(x,t)=\psi_{\beta}(x,t;\lambda')$, and $(\lambda,\lambda')=(1,-1)$. We also restrict our attention to the case of $x>0$; the case of $x\leq 0$ can be treated in a similar manner. As described in Table~\ref{table2}, we assume that $\max(\alpha,\beta)\leq 7/4$ and $\alpha+\beta \geq 7/4$. Then, with the convention that $\psi_0(x,t;\lambda)=1$, we have
  \begin{align}
    (fg)(x,t)
    & \leq \psi_{\alpha}(x,t;1)\psi_{7/4-\alpha}(x,t;-1)\psi_{\alpha+\beta-7/4}(x,t;-1) \\
    & \leq \psi_{\alpha}(x,t;1)\psi_{7/4-\alpha}(x,t;1)\psi_{\alpha+\beta-7/4}(x,t;-1)\leq C(t+1)^{-(\alpha+\beta)+7/4}\psi_{7/4}(x,t;1)\leq Ch(x,t).
  \end{align}
\end{proof}

We next recall the definition of $\chi_K(x,t;\lambda,\lambda')$:
\begin{equation}
  \chi_K(x,t;\lambda,\lambda')\coloneqq \mathrm{char}\left\{ \min(\lambda,\lambda')(t+1)+K(t+1)^{1/2}\leq x\leq \max(\lambda,\lambda')(t+1)-K(t+1)^{1/2} \right\},
\end{equation}
where $K>0$ and $\mathrm{char}\{ S \}$ is the indicator function of a set $S$. We then have the following lemma.

\begin{lem}
  \label{lem:indicator}
  Let $\lambda \neq \lambda'$ and $\alpha,\beta,K>0$. If $\alpha+\beta \geq 7/4$ and $\alpha+\beta/2 \geq 5/4$, then we have
  \begin{equation}
    |x-\lambda(t+1)|^{-\alpha}|x-\lambda'(t+1)|^{-\beta}\chi_K(x,t;\lambda,\lambda')\leq C[\psi_{7/4}(x,t;\lambda)+\bar{\psi}(x,t;\lambda')]
  \end{equation}
  for some $C>0$.
\end{lem}

\begin{proof}
  Assume for simplicity that $(\lambda,\lambda')=(1,-1)$. Let us first consider the case of $x>0$. In this case, we simply have
  \begin{equation}
    |x-(t+1)|^{-\alpha}|x+(t+1)|^{-\beta}\chi_K(x,t;1,-1)\leq |x-(t+1)|^{-(\alpha+\beta)}\chi_K(x,t;1,-1)\leq C\psi_{7/4}(x,t;1)
  \end{equation}
  since $\alpha+\beta \geq 7/4$. We next consider the case of $x\leq 0$. In this case, we first have
  \begin{equation}
    |x-(t+1)|^{-\alpha}|x+(t+1)|^{-\beta}\chi_K(x,t;1,-1)\leq |x+(t+1)|^{-(\alpha+\beta)}\chi_K(x,t;1,-1)\leq |x+(t+1)|^{-7/4}
  \end{equation}
  as above. Secondly, since $\chi_K(x,t;1,-1)\neq 0$ implies $x+(t+1)\geq K(t+1)^{1/2}$, we have
  \begin{equation}
    |x-(t+1)|^{-\alpha}|x+(t+1)|^{-\beta}\chi_K(x,t;1,-1)\leq C(t+1)^{-(\alpha+\beta/2)}\leq C(t+1)^{-5/4}
  \end{equation}
  since $\alpha+\beta/2 \geq 5/4$. Hence, we have
  \begin{equation}
    |x-(t+1)|^{-\alpha}|x+(t+1)|^{-\beta}\chi_K(x,t;1,-1)\leq C\bar{\psi}(x,t;-1).
  \end{equation}
  This proves the lemma.
\end{proof}

\bibliographystyle{amsplain}
\bibliography{kai-2020-1}

\providecommand{\bysame}{\leavevmode\hbox to3em{\hrulefill}\thinspace}
\providecommand{\MR}{\relax\ifhmode\unskip\space\fi MR }
\providecommand{\MRhref}[2]{%
  \href{http://www.ams.org/mathscinet-getitem?mr=#1}{#2}
}
\providecommand{\href}[2]{#2}
\begin{thebibliography}{10}

\bibitem{Deng16}
S.~Deng, \emph{{Initial-boundary value problem for p-system with damping in
  half space}}, Nonlinear Anal. \textbf{143} (2016), 193--210.

\bibitem{DWY15}
S.~Deng, W.~Wang, and S.-H. Yu, \emph{{Green's functions of wave equations in
  $\mathbb{R}_{+}^{n}\times \mathbb{R}_+$}}, Arch. Ration. Mech. Anal.
  \textbf{216} (2015), 881--903.

\bibitem{DW18}
L.~Du and H.~Wang, \emph{{Pointwise wave behavior of the Navier-Stokes
  equations in half space}}, Discrete Contin. Dyn. Syst. \textbf{38} (2018),
  1349--1363.

\bibitem{FMNT18}
E.~Feireisl, V.~M{\'{a}}cha, {\v{S}}.~Ne{\v{c}}asov{\'{a}}, and M.~Tucsnak,
  \emph{{Analysis of the adiabatic piston problem via methods of continuum
  mechanics}}, Ann. Inst. H. Poincar{\'{e}} Anal. Non Lin{\'{e}}aire
  \textbf{35} (2018), 1377--1408.

\bibitem{Hoff92}
D.~Hoff, \emph{{Global well-posedness of the Cauchy problem for the
  Navier--Stokes equations of nonisentropic flow with discontinuous initial
  data}}, J. Differential Equations \textbf{95} (1992), 33--74.

\bibitem{IK02}
T.~Iguchi and S.~Kawashima, \emph{{On space-time decay properties of solutions
  to hyperbolic-elliptic coupled systems}}, Hiroshima Math. J. \textbf{32}
  (2002), 229--308.

\bibitem{Kanel68}
Ya.~I. Kanel', \emph{{A model system of equations for the one-dimensional
  motion of a gas}}, Differ. Uravn. \textbf{4} (1968), 721--734.

\bibitem{Kawashima86}
S.~Kawashima, \emph{{Large-time behavior of solutions for hyperbolic-parabolic
  systems of conservation laws}}, Proc. Japan Acad. \textbf{62} (1986),
  285--287.

\bibitem{KS77}
A.~V. Kazhikhov and V.~V. Shelukhin, \emph{{Unique global solution with respect
  to time of initial-boundary value problems for one-dimensional equations of a
  viscous gas}}, Prikl. Mat. Mekh. \textbf{41} (1977), 282--291.

\bibitem{Koike21RIMS}
K.~Koike, \emph{{Refined pointwise estimates for 1D viscous compressible flows
  and the long-time behavior of a point mass}}, S\=urikaisekikenky\=usho
  K\=oky\=uroku, (to appear).

\bibitem{Koike21}
\bysame, \emph{{Long-time behavior of a point mass in a one-dimensional viscous
  compressible fluid and pointwise estimates of solutions}}, J. Differential
  Equations \textbf{271} (2021), 356--413.

\bibitem{Lequeurre20}
J.~Lequeurre, \emph{{Weak solutions for a system modeling the movement of a
  piston in a viscous compressible gas}}, J. Math. Fluid Mech. \textbf{22}
  (2020), https://doi.org/10.1007/s00021--020--0481--y.

\bibitem{LL16}
J.~Li and Z.~Liang, \emph{{Some uniform estimates and large-time behavior of
  solutions to one-dimensional compressible Navier--Stokes system in unbounded
  domains with large data}}, Arch. Ration. Mech. Anal. \textbf{220} (2016),
  1195--1208.

\bibitem{Liu78}
T.-P. Liu, \emph{{The free piston problem for gas dynamics}}, J. Differential
  Equations \textbf{30} (1978), 175--191.

\bibitem{Liu97}
\bysame, \emph{{Pointwise convergence to shock waves for viscous conservation
  laws}}, Comm. Pure Appl. Math. \textbf{50} (1997), 1113--1182.

\bibitem{LY11}
T.-P. Liu and S.-H. Yu, \emph{{On boundary relation for some dissipative
  systems}}, Bull. Inst. Math. Acad. Sin. (N.S.) \textbf{6} (2011), 245--267.

\bibitem{LY12}
\bysame, \emph{{Dirichlet-Neumann kernel for hyperbolic-dissipative system in
  half-space}}, Bull. Inst. Math. Acad. Sin. (N.S.) \textbf{7} (2012),
  477--543.

\bibitem{LZ97}
T.-P. Liu and Y.~Zeng, \emph{{Large time behavior of solutions for general
  quasilinear hyperbolic-parabolic systems of conservation laws}}, Mem. Amer.
  Math. Soc. \textbf{125} (1997), no.~599.

\bibitem{LZ09}
\bysame, \emph{{On Green's function for hyperbolic-parabolic systems}}, Acta
  Math. Sci. Ser. B (Engl. Ed.) \textbf{29} (2009), 1556--1572.

\bibitem{MTT17}
D.~Maity, T.~Takahashi, and M.~Tucsnak, \emph{{Analysis of a system modelling
  the motion of a piston in a viscous gas}}, J. Math. Fluid Mech. \textbf{19}
  (2017), 551--579.

\bibitem{Nishida86}
T.~Nishida, \emph{{Equations of motion of compressible viscous fluids}},
  Pattern and Waves (T.~Nishida, M.~Mimura, and H.~Fujii, eds.),
  Kinokuniya/North-Holland, Tokyo/Amsterdam, 1986, pp.~97--128.

\bibitem{Shelukhin77}
V.~V. Shelukhin, \emph{{The unique solvability of the problem of motion of a
  piston in a viscous gas}}, Dinamika Sploshn. Sredy \textbf{31} (1977),
  132--150.

\bibitem{Shelukhin78}
\bysame, \emph{{Stabilization of the solution of a model problem on the motion
  of a piston in a viscous gas}}, Dinamika Sploshn. Sredy \textbf{33} (1978),
  134--146.

\bibitem{Shelukhin82}
\bysame, \emph{{Motion with a contact discontinuity in a viscous heat
  conducting gas}}, Dinamika Sploshn. Sredy \textbf{57} (1982), 131--152.

\bibitem{Shelukhin83}
\bysame, \emph{{Evolution of a contact discontinuity in the barotropic flow of
  a viscous gas}}, Prikl. Mat. Mekh. \textbf{47} (1983), 870--872.

\bibitem{VZ03}
J.~L. V{\'{a}}zquez and E.~Zuazua, \emph{{Large time behavior for a simplified
  1D model of fluid--solid interaction}}, Comm. Partial Differential Equations
  \textbf{28} (2003), 1705--1738.

\bibitem{Zeng94}
Y.~Zeng, \emph{{$L^1$ asymptotic behavior of compressible, isentropic, viscous
  1-D flow}}, Comm. Pure Appl. Math. \textbf{47} (1994), 1053--1082.

\end{thebibliography}

\end{document}